\documentclass[10pt,openany]{bookk}
\usepackage{amsmath,amsthm,amsfonts,amssymb,amscd}
\usepackage[verbose,includemp=false,paperwidth=15.50cm,paperheight=21cm,body={10.50cm,16cm},rmargin=1.5cm,twoside]{geometry}

\usepackage[cam,letter,center]{crop}

\usepackage{graphicx}
\usepackage{float}

\usepackage[utf8]{inputenc}
\usepackage{lmodern}

\usepackage{hyperref}

\usepackage{array}

\def\CC{\mathbb{ C}}

\def\RR{\mathbb{ R}}

\newtheorem{theorem}{Theorem}[chapter]
\newtheorem{proposition}[theorem]{Proposition}
\newtheorem{corollary}[theorem]{Corollary}
\newtheorem{lemma}[theorem]{Lemma}
\newtheorem{remark}{Remark}
\newtheorem{definition}{Definition}[chapter]

\def\al{\alpha}
\def\be{\beta}

\def\de{\delta}  
    \def\si{\sigma}
   \def\la{\lambda}

\def\Om{\Omega}

\newcommand{\ii}{\'{\i}}

\newcommand{\R}{{\mathbb{R}}}
\newcommand{\CP}{{\mathbb{CP}}}

\newcommand{\ie}{{\it{i.e.}}: }

\newtheorem{crit}{Crit\'erio}[chapter]

\let\originalleft\left
\let\originalright\right
\renewcommand{\left}{\mathopen{}\mathclose\bgroup\originalleft}
\renewcommand{\right}{\aftergroup\egroup\originalright}

\def\ie{{\em i.e.,\ }}

\newfont\bbf{msbm10 at 12pt}
\def\eps{\varepsilon}

\def\R{{\mathbb R}}

\def\N{{\mathbb N}}

\def\K{{\mathbb K}}

\def\supp{\mbox{\rm supp}}

\def\le{\leqslant}
\def\ge{\geqslant}

\newdimen\AAdi%
\newbox\AAbo%
\def\AArm{\fam0 }%\tenrm}%
\def\AAk#1#2{\setbox\AAbo=\hbox{#2}\AAdi=\wd\AAbo\kern#1\AAdi{}}%
\def\AAr#1#2#3{\setbox\AAbo=\hbox{#2}\AAdi=\ht\AAbo\raise#1\AAdi\hbox{#3}}%
\def\BBone{{\AArm 1\AAk{-.8}{I}I}}%

\newcommand {\CA}{{\mathcal A}}
\newcommand {\CB}{{\mathcal B}}
\renewcommand {\CC}{{\mathcal C}}

\newcommand {\CH}{{\mathcal H}}

\newcommand {\CL}{{\mathcal L}}
\newcommand {\CM}{{\mathcal M}}

\newcommand {\CO}{{\mathcal O}}
\renewcommand {\CP}{{\mathcal P}}

\newcommand{\disp}{\displaystyle}
\newcommand{\8}{\infty}

\def\m1{{-1}}

\newcommand{\ninf}{{n\rightarrow\8}}

\def\S{\Sigma}
\def\s{\sigma}
\def\l{\lambda}
\def\supp{\mbox{supp}\,}

\newcommand{\wt}{\widetilde}

\def\be{\beta}
\def\al{\alpha}
\def\de{\delta}
\def\Om{\Omega}

\def\1{{1^{\8}}}
\def\2{{2^{\8}}}

\def\cpb{\CP(\be)}
\def\binf{{\be\to+\8}}

\newcommand{\Eg}{\vskip 0.2cm \noindent{\bf Example.}\hskip 0.2cm}
\newcommand{\Egs}{\vskip 0.2cm \noindent{\bf Examples}\\ \noindent}

\newcounter{exo}
\newcounter{qe}
\newcounter{eg}
\newcommand{\exo}{\stepcounter{exo} \setcounter{qe}{0}\vskip 0.5cm
\noindent {\bf Exercise \arabic{exo}} \\ \noindent}
\newcommand{\exobis}{\stepcounter{exo} \setcounter{qe}{0}\vskip
-0.2cm\noindent {\bf Exercise \arabic{exo}} \\ \noindent}

\title{Ergodic optimization, zero temperature limits and the Max-Plus algebra\\
$29^{\underline{\text{o}}}$\, Coloquio Brasileiro de Matematica, IMPA,  2013}
\bigskip

\author{A. T. Baraviera,  \; \; \;R. Leplaideur\,   \;  \; and \; \;  A. O. Lopes\, \\
Inst.  Mat. - UFRGS\, and Dept. Math. - Univ. de Brest}

\usepackage{subeqnarray}
\begin{document}
\synctex=1

\maketitle

 \frontmatter

  \tableofcontents
\mainmatter
  %  \include{capmaximize}
    %%%%%%%%%%%%%%%%%%%%%%%%%%%%%%%%%%%%
    %%%%%%%%%%%%%%%%%%%%%%%%%%%%%%%%%%%%%

%\end{document}

%teste

 %This is a book. The book is on the table

\chapter{Some preliminaires}\label{chap-preli}
%%%%%%%%%%

The purpose of this book is to present for math students some of the main ideas and issues  which
are considered in Ergodic Optimization. This text is also helpful for a mathematician with no experience in the area.
Our focus here are the questions which are associated to selection of probabilities when temperature goes to zero.
We believe that concrete examples can be very useful for any person which is reading this topic for the first time. We care about this point.
The reader will realize that the use of the so called  Max-Plus Algebra resulted in a very helpful tool for computing  explicit solutions for the kind of problems we are interested here.
We are not concerned in our text in presenting results in their more general form nor point out who were the person that did this or that first.
By the other hand, in the bibliography, we try to mention all results which appear in the literature. The reader  can look there, and
find other references which complement our short exposition. We are sorry for the case we eventually do not mention some paper in the area.
Simply we were not aware of the work.
We believe there is a need for a kind of text like ours. The purpose is to present some of the basic ideas of this beautiful theory for a more broad audience.We would like to thank several colleagues who read drafts of this book and made several suggestions for improvement.

\smallskip

\centerline{\,\,\,\,\,\,\,\,\,\,\,\,\,\,\,\,\,\,\,\,\,\,\,\,\,A. T. Baraviera, R. Leplaideur and A. O. Lopes}

\smallskip

\centerline{\,\,\,\,\,\,\,\,\,\,\,\,\,\,\,\,\,\,\,\,\,\,\,\,\,Rio de Janeiro, may 7, 2013.}

%%%%%%%%%%%
 \section{The configuration's space}\label{sec-space}

 %%%%%%%%%%%%%%
 \subsection{Topological properties}

We consider the space $ \Omega= \{1,2,...,k\}^{ \mathbb{N}}$ where the elements are sequences $x=(x_0,x_1,x_2,x_3...)$, in which $x_i\in \{1,2,...,k\}, i\in \mathbb{N}$. An element in $\Omega$ will also be called an \emph{infinite word} over the alphabet $\{1,\ldots, k\}$, and $x_{i}$ will be called a {\em digit} or a {\em symbol}.

The distance between two points $x=x_{0},x_{1},\ldots$ and $y=y_{0},y_{1},\ldots$ is given by
$$d(x,y)=\frac1{2^{\min\{n,\ x_{n}\neq y_{n}\}}}.$$

\Eg
In the case $k=4$, $d(1,2,1,3,4..),(1,2,1,2,3,..))=\frac{1}{2^3}$.

\medskip
We can represent this distance graphically as shown in figure 1.1.

\begin{figure}[H]
\unitlength=6mm
\begin{picture}(12,4.5)(-2,0)
\put(-1,1.5){\small$x_{0}=y_{0}$}
\put(0,2){\line(1,0){5}}
\put(5,2){\line(1,1){1}} \put(5,2){\line(1,-1){1}}
\put(5,1){\dashbox{0.2}(0.1,3)}\put(4.5,4.3){\small $n-1$}
\put(4,0){\small $x_{n-1}=y_{n-1}$}
\put(6,3){\line(1,0){3}}
\put(9.5,3){\small $y$}
\put(6,1){\line(1,0){3}}
  \put(9.5,1){\small $x$}
\end{picture}
\caption{The sequence $x$ and $y$ coincide from the digit $0$ up to the digit $n-1$, and then split.}\label{fig:distance}
\end{figure}
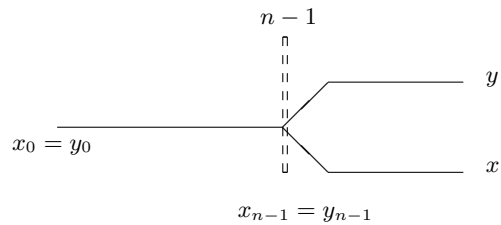

A finite string of symbols $x_{0}\ldots x_{n-1}$ is also called a \emph{word}, of length $n$. For a word $w$, its length is denoted by $|w|$. A \emph{cylinder} (of length $n$) is denoted by $[x_{0}\ldots x_{n-1}] \subset \Omega$. It is the set of points $y$ such that $y_{i}=x_{i}$ for $i=0,\ldots n-1$. For example, the cylinder $[2,1,1]$ is the set of $y=(y_0,y_1,y_2,y_3,y_4,...)$ such that $y_0=2,y_1=1,y_2=1.$

Note that cylinders of length $n$ form a partition of $\Omega$. Given $x \in \Omega$, then $C_{n}(x)$ will denote the unique element of this partition which contains $x$. That is, $C_{n}(x)=[x_{0}\ldots x_{n}]$, when  $x=x_0x_1x_2x_3\ldots$.
The cylinder $C_{n}(x)$ also coincides with the ball $B(x,\frac1{2^{n}})$. The set of cylinders of length $n$ will be denoted by $\CC_{n}(\Omega)$.

\medskip
If $\omega=\omega_{0}\dots\omega_{n-1}$ is a finite word of length $n$ and $\omega'=\omega'_{0}\omega'_{1}\ldots$ is a word (of any length possibly infinite), then $\omega\omega'$ is the word
$$\omega_{0}\ldots\omega_{n-1}\omega'_{0}\omega'_{1}\ldots$$
It is called the {\em concatenation} of $\omega$ and $\omega'$.

\medskip
The set $(\Omega,d)$ is a compact metric space. Compactness also follows from the fact that $\Omega$ is a product of compact spaces. Note that the topology induced by the distance $d$ coincides with the product topology.
The cylinders are clopen (closed and open) sets and they generate the topology.

\subsubsection{The general subshifts of finite type}

We will need to consider later the general subshift of finite type. A good reference for symbolic dynamics is \cite{lind-marcus}.

\begin{definition}
\label{def-transimatrix}
A {\em transition matrix} is a $d\times d$ matrix with entries in $\{0,1\}$.
\end{definition}
If $T=(T_{ij})$ is a  $d\times d$ transition matrix, the subshift of finite type $\S_{T}$ associated to $T$ is the set of sequences $x=x_{0}x_{1}x_{2}\ldots x_{n}\ldots$ such that for every $j$,
$$T_{x_{j}x_{j+1}}=1.$$
It is exactly the set of infinite words such that a subword $ij$ appears only if $T_{ij}=1$. Equivalently, it is the set of words such that the subword $ij$ never appears if $T_{ij}=0$.

\Eg
The full shift $\Om=\{1,\ldots d\}^{\N}$ is the subshift  of finite type associated to the $d\times d$ matrix with all entries equal to 1.

The best way to understand what is the subshift of finite associated to  a matrix $T$ is to consider paths: for a given $i$, the set of $j$ such that  $T_{ij}=1$, is the set of letters $j$ authorized to follow $i$. Then, $\S_{T}$ is the set of infinite words we can write respecting theses rules, or equivalently, the set of infinite paths that we can do.

\Eg
The transition matrix is $\disp T=\left(\begin{array}{cccc}1 & 0 & 1 & 0 \\1 & 1 & 0 & 1 \\1 & 1 & 1 & 1 \\0 & 1 & 1 & 0\end{array}\right)$.

 When there are more 1's than 0's in $T$ it is simpler to describe $\S_{T}$ (there are less restrictions). In the present case we get  the set of infinite words $x$ with letters in the alphabet 1, 2, 3 and 4, such that, $12$, $14$, $23$, $41$ and $44$, never appear in each sequence $x=x_0x_1x_2x_3\ldots$.

\begin{definition}
\label{def-reladigitshift}
Let $T$ be a $d\times d$ transition matrix, and $\S_{T}$ be the associated subshift of finite type.
Two digits $i$ and $j$ are said to be {\em associated}, if there exists a path from $i$ to $j$, and, another from $j$ to $i$. We set $i\sim j$.
\end{definition}
Equivalently, $i\sim j$ means that there exists a word in $\S_{T}$ of the form
$$i\ldots j\ldots i.$$
Obviously $\sim$ is an equivalence relation, and this defines  classes of equivalence. It can happen that there are equivalence classes strictly contained in $\{1,\ldots d\}$.

\exo Find examples of transition matrices such that one digit is associated only to itself and not to any other digit.

\begin{definition}
\label{def-irreducshift}
The subsets of words in $\S_T$ such that all their digits belong to the same
equivalence classes of $\sim$ are called the irreducible components of $\S_{T}$.
\end{definition}

%%%%%%%%%%%%%%%%%%%%%%%%%%
\subsection{Dynamics}

\begin{definition}
\label{def-shift}
The shift $\sigma:\Omega\to \Omega$, is defined by
$$\sigma(x_0x_1x_2x_3\ldots)= x_1x_2x_3\ldots.$$
\end{definition}

The shift expands distance by a factor 2:
$$d( \sigma(x), \sigma(y)) = 2\,d(x,y).$$
Hence, it is Lipschitz and thus continuous.

\begin{definition} Given $x\in \Omega$ the set $\{\sigma^n(x), n \geq 0\}$ is called the orbit of $x$. It is denoted by $\CO(x)$.
\end{definition}

The main goal in Dynamical Systems is to describe orbits and their behaviors.
Let us first present the simplest of all.

\begin{definition} A point $x\in \Omega$ is said to be periodic if there exists $k>0$, such that, $\s^{k}(x)=x$. In that case the period of $x$ is the smaller positive integer $k$ with this property.

A periodic point of period 1 is called a fixed point.
\end{definition}

\Egs
$111\ldots$ is a fixed point. \\
In $\{1,2,3,4\}^{\N}$, $x:=1323132313231323\ldots$ has period $4$ and the four points
$$13231323\ldots,\quad 32313231\ldots, \quad 23132313\ldots, \quad 31323132\ldots $$
form the orbit of $x$.

\medskip
A $n$-periodic point is entirely determined by its first $n$-digits; actually it is the infinite concatenation of these first digits:
$$x=\underbrace{x_{0}\ldots x_{n-1}}\underbrace{\ x_{0}\ldots x_{n-1}}\underbrace{\ x_{0}\ldots x_{n-1}}\ldots.$$
As a notation we shall set $x=(x_{0}\ldots x_{n-1})^{\8}$.

\bigskip

\begin{definition} Given a point $x$ in $\Omega$, a point $y\in \Omega$, such that, $\sigma(y)=x$ is called a {\em preimage of $x$}.

A point $y$ such that $\sigma^n(y)=x$ is called a {\em $n$-preimage of $x$}. The set
$\disp\{y\,|\ \text{there exists an }\,\, n \, \, \text{such that}\,\,\sigma^n(y)=x\,\}$
is called the {\em preimage set of $x$}.
\end{definition}

For the case of the full shift in $\Omega=\{1,\ldots k \}^{\N}$, each point $x$ has exactly $k$-preimages. They are obtained by the concatenation process
$ix$, with $i=1,\ldots , k$. The set of 1-preimage is $\s^{-1}(\{x\})$. The set of $n$-preimages is $(\s^{n})^{-1}(\{x\})$ which is simply denoted by $\s^{-n}(\{x\})$.

\begin{definition}
\label{def-s-invset}
A Borel set $A$ is said to be $\s$-invariant if it satisfies one of the following equivalent properties:
\begin{enumerate}
\item For any $x\in A$, $\s(x)$ belongs to $A$.
\item $\s^{-1}(A)\supset A$.
\end{enumerate}
\end{definition}
\Eg If $x$ is periodic, $\CO(x)$ is $\s$-invariant.
\exo
Show equivalence of both properties mentioned  in Definition \ref{def-s-invset}.

\exo
If $x$ is periodic, do we have $\s^{-1}(\CO(x))=\CO(x)$ ?

\subsubsection{Back to  the general subshifts of finite type}
We have defined above the general subshift of finite type. For such a subshift, we have also defined the irreducible components.
Here, we give a better description of these components with respect to the dynamics.
\begin{definition}[and proposition\footnote{We do not prove the proposition part.}]
\label{def-irreducible}
A  $\s$-invariant compact set $\K$ is called transitive if it satisfies one of the two equivalent properties:
\begin{itemize}
\item[(i)] For every pair of open sets of $\K$, $U$  and $V$, there exists $n>0$, such that, $\s^{-n}(U)\cap V\neq\emptyset$.
\item[(ii)] There exists a dense orbit.
\end{itemize}
\end{definition}
\exo
Show that if $\K$ is transitive, then the set of points in $\K$ with dense orbit is a $G_{\delta}$-dense set.

\medskip
\noindent
We claim that irreducible components are also transitive components. Indeed, any open set contains a cylinder and it is thus sufficient to prove $(i)$ with cylinders.
Now, considering two cylinders of the form $[x_{0}\ldots x_{k}]$ and $[y_{0}\ldots y_{n}]$, the relation defining irreducible components shows that there exists a connection
$$x_{0}\ldots x_{k}z_{1}\ldots z_{m}y_{0}\ldots y_{n},$$
remaining in $\K$.

%%%%%%%%%%%%%%%%%%%%%
\subsection{Measures}

We refer the reader to \cite{Bartle} and \cite{Fer} for general results in measure theory.

We denote by $\mathcal{B}$ the Borel $\s$-algebra over $\Omega$, that is, the one generated by the open sets.

We will only consider signed measures (probabilities) $\mu$ on $\Omega$ over this sigma-algebra $\mathcal{B}$, which we call Borel signed measures (probabilities). Due to the fact that cylinders are open sets and generate the topology, they also generate the $\s$-algebra $\CB$. Therefore, the values $\mu(C_{n})$, where $C_{n}$ runs over all the cylinders of length $n$, and $n$ runs over $\N$, determine uniquely $\mu$.

We remind the relation between Borel measures and continuous functions:
\begin{theorem}[Riesz]
\label{th-riesz}
The set of Borel signed measures is the dual of the set $\CC^{0}(\Omega)$.
\end{theorem}
\medskip

In other words, this  means that any linear bounded transformation $G:\CC^{0}(\Omega)\to \mathbb{R}$, is of the form
$$G:f\mapsto G(f)=\int f\, d \nu,$$
where $\nu$ is a fixed signed measure. The map  $G \to \nu$ is a bijection.
\medskip

Probabilities $\nu$ are characterized by the two properties: for any $f\geq 0$, we have that $G(f)\geq 0$, and $G(1)=1$.
\medskip

A subset $K$ of a linear space is convex, if, for any given $x,y\in K$, and any $\lambda,$ such that $0\leq \lambda\leq 1$, we have that
$\lambda\, x \,+\, (1-\lambda )\, y$ is in $K$.
\bigskip

\begin{corollary}
\label{coro-setproba}
The set of all probabilities on a compact space  is compact and convex for the weak*-topology.
\end{corollary}
We remind that $\disp\mu_{n}\stackrel{\text{w*}}{\longrightarrow}\mu$ means  that for every continuous function $f:\Omega\to\R$,
$$\int f\,d\mu_{n}\to_{\ninf}\int f\,d\mu.$$
We point out that any indicator function  of a cylinder, $\BBone_{C_{n}}$, is continuous. We recall that the support of a (probability) measure $\mu$ is the set of points $x$ such that
$$\forall\,\eps,\ \mu(B(x,\eps))>0.$$
In our case, $x$ belongs to the support, if and only, if $\mu(C_{n}(x))>0$, for every $n$. The support is denoted by $\supp(\mu)$.

\exo
Show that $\supp(\mu)$ is compact.

\medskip
We recall that a probability over $\Omega$ is a measure $\mu$ such that $\mu(\Omega)=1$. In this book we shall only consider  probabilities. Moreover, most of the time, we shall be interested in  properties holding {\em almost everywhere}, that is, properties which are true for all points $x\in \Omega$, up to a set of probability zero.

\begin{definition}\label{def-probainv} We say that a probability $\mu$  is {\em invariant for the shift $\sigma$}, if for any $A\in \CB$,
$$\mu(\sigma^{-1}(A)) = \mu(A).$$
We will also say that $\mu$ is $\s$-invariant, or simply invariant (as $\s$ is the unique dynamics we shall consider).
\end{definition}

\medskip
To consider invariant measure means the following thing: if we see the action of $\N$ as a temporal action on the system, the systems is closed, in the sense that along the time, there is neither creation nor disappearance of mass in the system. In other words, the mass of a certain given set is constant along the time evolution.

\medskip

Using the vocabulary from Probability Theory, the study of invariant probabilities correspond to the study of Stationary Processes (see \cite{KT} \cite{LL}).

\exo
Show that if $\mu$ is invariant, $\supp\mu$ is invariant. Is it still the case if $\mu$ is not invariant ?

\medskip
In Ergodic Theory one is mainly interested in invariant probabilities, and the properties that are true  for points $x$ which are in a set which have mass equal to one.

%%%%%%%%%%%%%%%%%%%%%%%%%%%
%%%%%%%%%%%%%%%%%%%%%%%%%%%%%%
\section{Invariant measures}
In this section we present some particular invariant measures in our setting and we will also present some more general results.

%%%%%%%%%%%%%%%%%%%%%%%%%%%%%%%
\subsection{Examples of invariant measures}\label{subsec-examp-meas}
\subsubsection{Periodic measures}
If $x$ is a point in $\Omega$, $\delta_{x}$ is the {\em Dirac} measure at $x$, that is
$$\delta_{x}(A)=
\begin{cases}
 1\text{ if }x\in A,\\
 0\text{ otherwise}.
\end{cases}
$$
Then, if $x$ is $n$-periodic,
$$\mu:=\frac1n\sum_{j=0}^{n-1}\delta_{\s^{j}(x)}$$
is $\s$-invariant.

\subsubsection{The Bernoulli product measure}
Let consider $\Omega=\{1,2\}^{\N}$, pick two positive numbers $p$ and $q$ such that $p+q=1$.  Consider the measure $\mathbb{P}$ on $\{1,2\}$ defined by
$$\mathbb{P}(\{1\})=p,\ \mathbb{P}(\{2\})=q.$$
Then, consider the measure $\mu:=\otimes\mathbb{P}$ on the product space $\{1,2\}^{\N}$. We remind that such $\mu$  is defined in such way that
$$\mu([x_{0}\ldots x_{n-1}])=p^{\#\text{ of 1's in the word }\,x}\,\,\,q^{\#\text{ of 2's in the word}\,x},$$
where $x$ is the finite word $x_{0}\ldots x_{n-1}.$

We claim that such $\mu$ is an invariant measure. Indeed, we have for any cylinder $[x_{0}\ldots x_{n-1}]$,
$$\s^{-1}([x_{0}\ldots x_{n}])=[1x_{0}\ldots x_{n}]\sqcup [2x_{0}\ldots x_{n}].$$
Then,
$$\disp \mu([1x_{0}\ldots x_{n}])=p\mu([1x_{0}\ldots x_{n}])$$ and
$$\disp \mu([2x_{0}\ldots x_{n}])=q\mu([1x_{0}\ldots x_{n}]).$$

In this way $\mu(\sigma^{-1}[x_{0}\ldots x_{n}])= \mu([x_{0}\ldots x_{n}]).$

\medskip
This example corresponds to the model of tossing a coin (head identified with $1$ and tail identified with $2$) in an independent way a certain number of times. We are assuming that each time we toss the coin the probability of head is $p$ and the probability of tail is $q$.

Therefore, $\mu ([211])$ describes the probability of getting tail in the first time and head in the two subsequent times we toss the coin, when we toss the coin three times.

\begin{remark}
\label{rem-uncountablemeas}
The previous example shows that there are uncountably many $\s$-invariant probabilities on $\{1,2\}^\mathbb{N}$.
$\blacksquare$\end{remark}

\subsubsection{Markov chain}
Let us start with some example.
Again, we consider the case $\Omega=\{1,2\}^{\N}$. Pick $p$ and $q$ two positive numbers in $]0,1[$, and set
$$P=\left(
\begin{array}{cc}
p & 1-p \\
1-q & q \\
\end{array}
\right)\,=\left(
\begin{array}{cc}
P(1,1) & P(1,2) \\
P(2,1) & P(2,2) \\
\end{array}
\right).$$
The first writing of $P$ shows that 1 is an eigenvalue. If we solve the equation
$$(x,y).P=(x,y),$$
we find a one-dimensional eigenspace (directed by a {\em left eigenvector}) with $y=\disp\frac{1-p}{1-q}x$. Therefore, there exists a unique left eigenvector $(\pi_{1},\pi_{2})$ such that
$$(\pi_{1},\pi_{2}).P=(\pi_{1},\pi_{2})\text{ and }\pi_{1}+\pi_{2}=1.$$
Note that $\pi_{1}$ and $\pi_{2}$ are both positive.

The measure $\mu$  is  then defined by
$$\mu([x_{0}\ldots x_{n}])=\pi_{x_{0}}P(x_{0},x_{1})P(x_{1},x_{2})\ldots P(x_{n-2},x_{n-1}).$$

A simple way to see the measure $\mu$ is the following: a  word $\omega=\omega_0\ldots\omega_{n-1}$ has to be seen as a path of length $n$, starting at state $\omega_0\in\{1,2\}$ and finishing at state $\omega_{n-1}$.
The measure $\mu([\omega])$ is the probability of this space among all the paths of length $n$. This probability is then given by the initial probability of being in state $\omega_0$ (given by $\pi_{\omega_0}$) and then probabilities of transitions from the state $\omega_{j}$ to $\omega_{j+1}$ (equal to $P(\omega_{j},\omega_{j+1})$), these events being independent.

\bigskip
A probability of this form is called the {\em Markov measure obtained from the $2\times2$ line stochastic matrix $P$ and  the initial vector of probability $\pi$}.

\exo
Show that the Bernoulli measure  constructed above is also a Markov measure.

\medskip

More generally we have:
\begin{definition} A $d\times d$ matrix $P$ such that all entries are positive and the sum of the elements in each line is equal to $1$ is called a {\bf line stochastic matrix}.
\end{definition}

One can show\footnote{Actually we will give a proof of that result in Theorem \ref{perron}.} that for a line stochastic matrix (with all entries strictly positive) there exist only one vector $\pi=(\pi_1,\pi_2,...,\pi_d)$, such that all $\pi_j>0$, $j\in \{1,2,..,d\}$, $\sum_{j=1}^d \pi_j=1,$ and
$$ \pi\, =\, \,\pi\,\,P.$$
$\pi$ is called the {\bf left invariant probability vector for the Markov Chain defined by $P$}.

\begin{definition}[and proposition] Given a  $d\times d$ line stochastic matrix $P$, and its left invariant probability vector $\pi= (\pi_1,\pi_2,...,\pi_d)$, we define, $\mu$ on $\Omega=\{1,2,\ldots,d\}^{\N}$,  in the following way: for any cylinder $[x_0x_1\ldots x_k]$
$$ \mu ([x_0,x_1\cdots x_k])\,=\pi_{x_0}\,P(x_0,x_1)\,P(x_1,x_2)\,P(x_2,x_3)\,...\,P(x_{k-1},x_k).$$
This measure $\mu$ {\bf is invariant for the shift} and it is called the {\em Markov measure associated to $P$ and $\pi$}. For a fixed $d$ we denote the set of Markov measures over $\Omega=\{1,2,..,d\}^\mathbb{N}$ by  $\overline{\mathcal{G}}$
\end{definition}

Note that for a Markov probability  any cylinder set has positive measure

%%%%%%%%%%%%%%%%%%%%%%%%%%%%%%%%%%%%%%%%%%%%%%%%%%%%%%%%%%%%%%%%%%%%%%%%%%%%%
\subsection{General results on invariant measures}
The definition of invariant measure involves Borel sets. The next result gives another characterization of invariant measures (see \cite{PY}):
\begin{proposition}\label{prop-meas-inv-fonc}
The measure $\mu$ is invariant, if and only if,  for any continuous function $f$ $$\int f(x)\, d\mu(x)=\int f (\sigma(x))\, d\mu(x).$$
\end{proposition}

It is easy to see from the above that in the case  $\disp\mu_{n}\stackrel{\text{w*}}{\longrightarrow}\mu$, and each $\mu_{n}$ is invariant, then, $\mu$ is invariant.

\medskip
We denote by $\CM_\s$ the set of invariant probabilities on $\Omega$. It is a closed subset of probabilities for the weak*-topology, hence it is compact and convex.
\medskip

We say that a point $z$ in a convex compact set $C$ is extremal, if it is not a trivial convex combination of other elements in the set $C$. That is, we can not write $z$ as
$z=\lambda x + (1-\lambda) y$, where $0<\lambda<1$, and $x,y$ are in the convex set $C$.
\medskip

\begin{definition}
An extremal measure in $\CM_\s$ is called ergodic.
\end{definition}

This definition is however not useful and  clearly not easy to check. The next proposition gives other criteria for a measure to be ergodic.

\begin{proposition}\label{prop-crit-ergodi}
A probability $\mu$ is ergodic, if and only if, it satisfies one of the following properties:
\begin{enumerate}
\item Any invariant Borel set has full measure or zero measure.
\item For  any  continuous $f:\Omega\to\R$, if $f=f\circ \s$ $\mu$-a.e., then,  $f$ is constant.
\end{enumerate}
\end{proposition}

\exo
Show that a Markov measure is ergodic.

\bigskip
We can now state the main theorem in Ergodic Theory:
\begin{theorem}\label{birkhoff}
{\bf (Birkhoff Ergodic Theorem)}  Let $\mu$ be $\s$-invariant  and ergodic. Then, for every continuous function $f:\Omega \to \mathbb{R}$ there exist  a Borel set  $K$, such that $\mu(K)=1$, and for every $x$ in $K$
$$
\lim_{n\rightarrow +\infty}\frac{1}{n}\sum_{k=1}^n f(\sigma^{k-1}(x))=\int f\,d\mu.
$$
\end{theorem}

The Birkhoff Theorem  says that, under the assumption of ergodicity, a time average is equal to a spatial average. Here is an example of application:

``Average cost of car ownership rises to \$8,946 per year.''

What does the term average mean ? One can imagine that we count how much a certain person spends every year for its car, and then do the average cost. This is a time-average. The main problem of this average is to know if it representative of the cost of anybody.

On the contrary, one can pick some region, then count how many people spend in 1 year for their car, and take the average amount. This is a spatial average. The main problem is to know if it represents how much each person is going to spend along the years (at beginning the car is new, and then get older !).

The ergodic assumption means that the repartition of old and new cars in the space is ``well'' distributed and/or  that the chosen person in the first way to compute the average cost is ``typical''. Then, the Birkhoff theorem says that both averages are equal.

\paragraph{Notation.} We set $S_{n}(f)(x):=f(x)+\ldots+f\circ \s^{n-1}(x)$.
\bigskip

\Eg
In a previous example we considered a Bernoulli measure modeling the tossing of a coin with probabilities $p$ of head and $q$ of tail. Consider the indicator function $\BBone_{[2]}$  of the cylinder $[2]$. For a fixed
$n$ and for a fixed $x=x_0x_1x_2\ldots$ the value $\sum_{k=1}^n \BBone_{[2]}(\sigma^{k-1}(x))$, counts the number of times we get
tail (or, the value $2$) in the finite string $x_0x_1x_2\ldots x_{n-1}$.

Note that $\int \BBone_{[2]}\, \,d \mu$ by definition is equal to $\mu([2])=q.$ One can show that the measure $\mu$ we get is ergodic.

Therefore, from Birkhoff Theorem, we can say that for $\mu$-almost every $x$, we have that
$$\lim_{n\rightarrow +\infty}\frac{1}{n}\sum_{k=1}^n  \BBone_{[2]}(\sigma^{k-1}(x))=\int \BBone_{[2]} \, d \mu=q.$$

The value $\frac{1}{n}\sum_{k=1}^n  \BBone_{[2]}(\sigma^{k-1}(x))$ is the empirical mean value of number of times we get tail if the sequence of events is obtained from flipping the coin $n$ times, which is  described by $x_0x_1x_2\ldots x_{n-1}$, where $x=(x_0,x_1,x_2,...)$.

\bigskip

\medskip
We finish this section with another application of ergodicity:
\begin{proposition}
\label{prop-return}
Let $\mu$ be an invariant ergodic probability. Let $x$ be a ``generic '' point in $\supp\mu$. Then, $x$ returns infinitely many times, as closed as wanted, to itself.
\end{proposition}
\begin{proof}
Pick $\eps>0$, and consider the ball $B(x,\eps)$. It is a clopen set, hence $\BBone_{B(x,\eps)}$ is continuous.
The point $x$ is generic for $\mu$, then $\disp\lim_{\ninf}\frac1nS_{n}(\BBone_{B(x,\eps)})(x)=\mu(B(x,\eps))$, and this last term is positive. Therefore, there are infinitely many $n$ such that $\BBone_{B(x,\eps)}(\s^{n}(x))=1$.
\end{proof}

\begin{remark}
\label{rem-genericset}
Actually, one can get a stronger result than the one claimed by the above  Birkhoff Theorem: for every $\mu$ ergodic, there exists a set  $G_{\mu}$ of full $\mu$-measure, such that for every $x\in G_{\mu}$  and for every continuous function $f$,
$$\lim_{\ninf}\frac1nS_{n}(f)(x)=\int f\,d\mu.$$
A very important property is that for two distinct ergodic probabilities $\mu$ and $\nu$,  $G_{\mu}\cap G_{\nu}=\emptyset$ (see \cite{PY}).
$\blacksquare$\end{remark}

%%%%%%%%%%%%%%%%%%%%%%%%%%%%%%%%%%%%%%%%%%%%%%%%%%%%%%%%%5
%%%%%%%%%%%%%%%%%%%%%%%%%%%%%%%%%%%%%%%%%%%%%%%%%%%%%%%%%%%%
\section{Ergodic optimization and temperature zero}

The set $\CM_\s$ of invariant measures is quite large. It is thus natural to ask about measures with special properties. In this direction, a well known class
 is the one which can be obtained from {\em Thermodynamic Formalism}. The first results on this topic are from the 70's. We will briefly describe some basic results on this setting  below and we will present some more details  in Chapter \ref{chap-thermo}. Anyway, our goal here is to focus on {\em Ergodic Optimization}.

\begin{definition}\label{def-meas-max}
Let $A:\Omega\to\R$ be a continuous function. An invariant measure $\mu$ is said to be $A$-maximizing if
$$\int A\,d\mu=\max\left\{\int A\,d\nu,\ \nu\in\CM_\s\right\}=:m(A).$$
\end{definition}

Note that this maximum is well defined because $A$ is continuous and $\CM_\s$ is compact for the weak*-topology.

\Egs
$\bullet$ Consider
$\Omega=\{1,2 \}^\mathbb{N}$ and $A= \disp\frac12\BBone_{[1,2]}+ \frac12\BBone_{[2,1]}$.
In this case the maximizing probability is unique and has support in the periodic orbit of period  two $12121212\ldots$. This maximizing measure is $\disp\mu:=\frac12\delta_{(12)^{\8}}+\frac12\delta_{(21)^{\8}\ldots}$.

\noindent $\bullet$
There can be several  maximizing measures. Consider $A(x)=-\min(d(x,1\8), d(x,2^{\8}))\le 0$. The two measures $\mu_{1}:=\de_{\1}$ and $\mu_{\2}:=\de_{\2}$ are $A$-maximizing. Consequently, any convex combination
$$\mu_{t}:=t\mu_{\1}+(1-t)\mu_{\2}, \ t\in[0,1],$$
is also $A$-maximizing. It can be shown, in this particular case, that any   maximizing probability is of this form. Note that $\mu_{t}$ is not ergodic if $t\neq 0,1$.

\noindent $\bullet$
 Suppose $\Omega=\{1,2,3,4\}^\mathbb{N}$ and consider $A=\BBone_{[1]\cup[2]}$. any measure with support in $\{1,2\}^\mathbb{N}$  is $A$-maximizing. This example shows that there can be uncountably many ergodic maximizing measures.

\bigskip
Chapters \ref{chap-maxmeas} and \ref{chap-peierl} are devoted to more general results and tools for the study of maximizing measures. Chapter \ref{chap-explexam} is devoted to the study of an explicit example.

% We emphasize that \emph{generically} for the $\CC^0$ topology there exists a unique maximizing measure.
%
%Nevertheless it is extremely easy to construct non-generic examples: take two disjoint periodic orbits, say $\CO(x)$ and $\CO(y)$,  set $K=\CO(x)\sqcup \CO(y)$, and then pick $A(\omega)=-d(\omega,K)$.
%It has two invariant ergodic maximizing measures, the periodic measure associated to $\CO(x)$ and the periodic measure associated to $\CO(y)$.
%
%We also emphasize that any convex combination of these two measures is also invariant and $A$-maximizing.

\bigskip
As we mentioned above,
another important class of  invariant measures appears in  Thermodynamic Formalism. This will be the topic of Chapter \ref{chap-thermo}.

Consider a fixed function $A:\Omega \to \mathbb{R}$, which is called a potential.
Without entering too much into the theory, for a real parameter $\be$ we shall associate to $\be.A$ a functional $\CP(\be)$, and for each $\be$ some measure $\mu_\be$ called \emph{equilibrium state for the potential $\be.A$}. In Statistical Mechanics, $\be$ represents the inverse of the temperature. Then, $\be\to+\8$ means that the temperature goes to 0.

We want to point out the relations between the two different ways to singularize measures:
\begin{enumerate}
\item $\be\mapsto \CP(\be)$ is convex and admits an asymptote for $\be\to+\8$. The slope is given by $\disp m(A)=\max\int A\,d\mu$.
\item Any accumulation point for $\mu_\be$, as $\be\to+\8$, is $A$-maximizing. Then, the main question is to know if there is \emph{convergence}, and, in the affirmative case, how does $\mu_\be$ select the limit ?
\end{enumerate}

These are the mains points which we consider in Chapter \ref{chap-maxmeas} (general results) and  Chapter \ref{chap-explexam} (an specific example).

%%%%%%%%%%%%%%%%%%%%%%%%%%%%%%%%%%%%%%%%%%%%%%%%%%%%%%%%%%%%%%%%%%%%%%%%%%%%%%
%%%%%%%%%%%%%%%%%%%%%%%%%%%%%%%%%%%%%%%%%%%%%%%%%%

%%%%%%%%%%%%%%%%%%%%%%%%%%%%%%%%%%%%%%%%%
%%%%%%%%%%%%%%%%%%%%%%%%%%%%%%%%%%%%%%%%%
%%%%%%%%%%%%%%%%%%%%%%%%%%%%%%%%%%%%%%%%%
\chapter{Thermodynamic Formalism and Equilibrium states \label{chap-thermo}}

%%%%%%%%%%%%%%%%
%%%%%%%%%%%%%%%%%%
\section{General results}

%%%%%%%%%%%%%%%%%%
\subsection{Motivation and definitions}

A dynamical system, in general, admits a large number of different ergodic measures. A natural problem would be to find a way to singularize some special ones among the others.

Given $A:\Omega\to\R$, the thermodynamic formalism aims to singularize measures via the following principle:
\begin{equation}
\label{equ-def-pressure}
\CP(A):=\sup_{\mu\in\CM_{\s}}\left\{h_{\mu}+\int A\,d\mu\right\}.
\end{equation}
The quantity $h_{\mu}$ is called the Kolmogorov entropy. It is a non-negative real number, bounded by $\log d$ if $\Om=\{1,\ldots d\}^{\N}$. Roughly  speaking, it measures the chaos seen by the measure $\mu$. It is a concept defined for invariant probabilities. If some measure has positive and large entropy, this means that the system ruled by this measure is very complex. If the invariant measure has support in a periodic orbit its entropy is zero.

\begin{definition}
\label{def-press-et equil}
Any measure which realizes the maximum in \eqref{equ-def-pressure} is called an {\em equilibrium state} for $A$. The function $A$ is called the {\em potential} and $\CP(A)$ is {\em the pressure} of the potential.
\end{definition}

For a given $\s$-invariant measure $\mu$, the quantity $h_{\mu}+\disp\int A\,d\mu$ is called the {\em free-energy} of the measure (with respect to the potential $A$).

 This study is inspired by Statistical Mechanics: the set $\Omega:=\{1,\ldots k\}^{\N}$ is the one-dimensional lattice  with $k$ possible values for the spin in each site. The potential $A$ measures the interaction between sites. At a macroscopic level, the lattice has a stable behavior, at the microcosmic scale, due to internal agitation, the configuration of the material changes. Therefore, a {\em state of the system} is a probability on $\Omega$, the set of possible configurations; it is obtained by a principle of the  kind of the law  of large numbers.
The equilibrium at the macroscopic scale is exactly given by the states which maximize the free energy.

In Statistical Mechanics people usually consider the influence of the temperature $T$ by introducing a  term $\beta:=\frac1T$ and studying the equilibrium for $\be. A$.
For some fixed $A:\Om\to\R$ and $\be\in \R$ , we shall write $\CP(\be)$ instead of $\CP(\be.A)$.
It is an easy exercise to check that $\be\mapsto\CP(\be)$ is convex and  continuous. Therefore, the slope is non-decreasing and a simple computation shows that it converges to $m(A):=\disp\sup_{\mu}\int A\,d\mu$ as $\beta \to \infty$.

It seems clear that a motivation for people who introduced the thermodynamic formalism theory into the Dynamical systems (Bowen,  Ruelle  and Sinai) was to study the functional $\CP(\be)$, $\beta \in \mathbb{R}$. Given that,  one  natural question is to understand the limit of $\CP(\be)$, when $\be$ goes to $+\8$, that is the zero temperature case. This is one of the purposes of this text.

%%%%%%%%%%%%%%%%%%%%
\subsection{Entropy and existence of equilibrium states}

We
refer the reader to \cite{Bow}\cite{PP} \cite{Bal} \cite{Kel} \cite{LoT}
for the some of the results we use from
Thermodynamic Formalism. We present here some results which are related to  optimization in ergodic theory.

We emphasize that the results we present here are stated for $\Om$ but they holds for any general irreducible subshift of finite type.

\subsubsection{Entropy}
For general results on entropy see also \cite{BCLMS} and \cite{LMMS}. The complete description  of entropy is somehow complicated and not relevant for the purpose of this course.
A simple definition  for the general case could be the following:

\begin{theorem}[and definition]
\label{th-def-entropy}
Let $\mu$ be a $\s$-invariant ergodic probability. Then, for $\mu$-a.e. $x=x_{0}x_{1}x_{2}\ldots$,
$$\lim_{\ninf}-\frac1n\log\mu([x_{0}x_{1}\ldots x_{n}])$$
exists and is independent of $x$. It is equal to $h_{\mu}$.

If $\mu_{0}$ and $\mu_{1}$ are both invariant and ergodic probabilities, for every $\al\in[0,1]$ set $\mu_{\al}:=\al.\mu_{1}+(1-\al).\mu_{0}$. Then,
$$h_{\mu_{\al}}=\al h_{\mu_{1}}+(1-\al)h_{\mu_{0}}.$$
\end{theorem}

Roughly speaking, for $\mu$ ergodic and for $\mu$-typical $x=x_{0}x_{1}\ldots$,  $h_{\mu}$ is the exponential growth for $\mu([x_{0}\ldots x_{n-1}])$, with $n\to \infty$.

\Egs
$\bullet$ Let us consider $\Om=\{1,2\}^{\N}$ and $\mu$ the Bernoulli measure given by the line stochastic matrix $\disp P=\left(
\begin{array}{cc}
p & q \\
p & q \\
\end{array}\right).
$
We have seen before that
$$\mu([x_{0}\ldots x_{n-1}])=p^{\#\text{ of 1's in the word}}q^{\#\text{ of 2's in the word}}.$$
We also have seen before that $\disp \frac1n{\#\text{ of 1's in the word}}\to_{\ninf}\mu([1])=p$, and, $\disp \frac1n{\#\text{ of 2's in the word}}\to_{\ninf}\mu([2])=q$. Therefore,
$$h_{\mu}=-p\log p-q\log q.$$

$\bullet$
More generally, if $\mu$ is the  Markov measure associated to  a line stochastic matrix $P$, and,  $\pi=(\pi_{1}, \pi_{2}, ,,,\pi_{d})$ is the stationary vector, then,
$$ h(\mu)= - \sum_{i,j=1}^n \, \pi_i \,\,P(i,j)\,\, \log P(i,j).$$

\begin{proposition}
\label{prop-semiconti}
The (metric) entropy is upper-semicontinuous :
If $\mu_{n}$ converges to $\mu$ in the weak*-topology, then
$$h_{\mu}\ge \limsup_{\ninf}h_{\mu_{n}}.$$
\end{proposition}

\subsubsection{Existence of equilibrium states}

This is an immediate consequence of Proposition \ref{prop-semiconti}.
\begin{theorem}
\label{th-exist-equil}
If $A$ is continuous, then there exists  at least one equilibrium state for $A$.
\end{theorem}
\begin{proof}
The function $\mu \to  h_{\mu}+\int A\,d\mu$ is upper-semicontinuous in the compact set of the invariant probabilities. Therefore attains the maximal value in at least one invariant probability.
\end{proof}

At the end of this chapter we will study the special case where  $A$ depend just on the two first coordinates. In particular, we will be able to exhibit an explicit expression for   the corresponding
equilibrium state.

\section{Uniqueness of the Equilibrium State}
Uniqueness  does not always hold. Nevertheless a key result is the following:
\begin{theorem}
\label{th-uniqu-equil}
If $A:\Om\to\R$ is H\"older continuous, then there is a unique equilibrium state for $A$. Moreover, it is a {\em Gibbs measure} and $\be\to\CP(\be)$ is analytic.
\end{theorem}

The concept of Gibbs measure will be explained below.
We recall  that $A: \Omega \to \mathbb{R}$ is said to be {\bf $\alpha$-H\"older}, $0<\alpha<1$, if there exists $C>0$, such that, for all $x,y$ we have
$ |\, A(x)-A(y)\,|\leq  C\, d(x,y)^\alpha$.

For a fixed value $\alpha$, we denote by $\mathcal{H}_\alpha$ the set of $\alpha$-H\"older functions $A:\Omega \to\mathbb{R}$. $\mathcal{H}_\alpha$ is a vector space.

For a fixed $\alpha$,
the norm we consider in the set $\mathcal{H}_\alpha$  of $\alpha$-H\"older potentials $A$ is
$$
||A||_\alpha=\sup_{x\neq y}\frac{| A(x)-A(y)|}{d(x,y)^\alpha}+ \sup_{x \in \Omega}\, |A(x)|. $$

For a fixed $\alpha$, the vector space
$\mathcal{H}_\alpha$  is complete for the above norm.

We want to emphasize Theorem \ref{th-uniqu-equil}. If existence of equilibrium state is done via a general result of maximization of a semi-continuous function, uniqueness, by the other hand, is obtained for H\"older continuous potential via a completely different way: the key tool is an operator acting on continuous and H\"older continuous functions. Then, the pressure and the equilibrium state are related to the spectral properties of this operator which had its origin in Statistical Mechanics where is called the transfer operator.

The hypothesis of $A$ being  H\"older continuous means (in the Statistical Mechanics setting) that the interactions described by it decay very fast (in an exponential way) for spins which are located more and more  distant in the lattice $\mathbb{N}$. This decay is not so fast for a potential which is continuous but not H\"older.

In the general case,
for a fixed H\"older  potential $A$, given $\beta_1 \neq \beta_2$, we have that $\mu_{\beta_1} \neq \mu_{\beta_2}$. In this case the generic sets of Remark \ref{rem-genericset} are disjoint. Moreover, for any $\be$, the probability   $\mu_\beta$ has support on the all set $\Omega$.

\begin{remark}
\label{rem-phasetransi}
The non-uniqueness of the equilibrium state for a potential $A$ is associated to the phenomena of phase transition (see \cite{Hof} \cite{Ge} \cite{LZ} \cite{LF} \cite{BLL1}).
$\blacksquare$\end{remark}
%%%%%%%%%%%%%%%%%%%%%%%%%%%%%%%%%%%%
\subsection{The Transfer operator}
In this section we consider a fixed $\alpha$-H\"older potential $A:\Omega\to \mathbb{R}.$

\begin{definition}
We denote by $\CL_A:\mathcal{C}^{0} (\Omega)\to \mathcal{C}^{0}(\Omega)$ the Transfer operator corresponding to the potential
$A$, which is given in the following way: for a given $\phi$ we will get another function $\CL_A(\phi)=\varphi$, such that,
$$
\varphi(x)=\sum_{a \in \{1,2,...,d\}}
e^{ A(ax)}\,
\phi(ax).
$$
\end{definition}

In another form
$$
\varphi(x) =\varphi(x_0x_1\cdots)=\sum_{a \in \{1,2,...,d\}}
e^{ A(ax_0x_1x_2,...)}\,
\phi(ax_0x_1x_2...).
$$

The transfer operator is also called the Ruelle-Perron-Frobenius operator.

%It was introduced by Ruelle and extends to an operator (acting on continuous functions)  the action of a matrix with positive entries.   We remind that for such matrices,
%the  Perron-Frobenius theorem  gives information on the spectrum. We can get a similar result in this new setting.

It is immediate to check that $\CL_{A}$ acts on continuous functions. It also acts on $\al$-H\"older functions if $A$ is $\al$-H\"older.

If $\mu$ is a probability, $f\mapsto \disp\int \CL_{A}(f)\,d\mu$ is a bounded linear form on $\CC^{0}(\Om)$. Therefore, by Riesz Theorem (see Theorem \ref{th-riesz}), there exists $\nu$ such that for every $f\in\CC^{0}(\Om)$, $\disp\int f d\nu= \int \CL_{A}(f)\, d \mu$.  We set $\CL_{A}^* (\mu)=\nu$, and we call $\CL_{A}^*:\mu\mapsto \nu$ the dual operator of $\CL_{A}.$

\bigskip

Note that if $\CL_{A}(1)=1$, then,  $\CL_{A}^* (\mu)$ is a probability, in the case $\mu$ is a probability.

\begin{theorem}[see \cite{krerley}]
\label{th-eigenmeas}
Let $\lambda_{A}$ be the spectral radius of $\CL_{A}$. Then, $\lambda_{A}$ is an eigenvalue for $\CL_{A}^{*}$: there exists a probability measure  $\nu_{A}$ such that
$$\CL_{A}^{*}(\nu_{A})=\lambda_{A}\nu_{A}.$$
This probability is called the eigenmeasure and/or the conformal measure.
\end{theorem}
We remind that the spectral radius is given by
$$\l_{A}:=\limsup_{\ninf}\frac1n\log|||\CL_{A}^{n}|||,\text{ and }|||\CL_{A}^{n}|||=\sup_{||\psi||=1}||\CL^{n}(\psi)||_{\8}.$$

Note that the Schauder-Tychonov Theorem shows that there exists an eigenmeasure. Indeed, consider the function acting on the convex and compact set of probabilities $\mu \to \frac{\CL_{A}^{*}(\mu)}{\int \,\CL_{A}( 1) \, d \mu }.$ It is however not clear that the associated eigenvalue is the spectral radius of $\CL_{A}$.

\bigskip
The main ingredient to prove uniqueness of the equilibrium state is:
\begin{theorem}
\label{theo-trouspect}
The operator $\CL_{A}$ is quasi-compact on $\CH_{\al}$: $\lambda_{A}$ is simple isolated  and the unique eigenvalue with maximal radius. The rest of the spectrum is contained in a disk $\mathbb{D}(0,\rho\lambda_{A})$ with $0<\rho<1$.
\end{theorem}
From Theorem \ref{theo-trouspect} we get a unique $H_{A}$, up to the normalization $\disp\int H_{A}\,d\nu_{A}=1$, such that
$$\CL_{A}(H_{A})=\lambda_{A}H_{A}.$$
\exobis
Show that the measure defined by $\mu_{A}=H_{A}\nu_{A}$ is $\s$-invariant.

\bigskip
It can be proved that this measure is actually a {\em Gibbs} measure:  which means, there exists $C_{A}>0$ such that for every $x=x_{0}x_{1}\ldots$ and for every $n$,

\begin{equation}
\label{equ-def-gibbs}
e^{-C_{A}}\le \frac{\mu_{A}([x_{0}\ldots x_{n-1}])}{e^{S_{n}(A)(x)-n\log\lambda_{A}}}\le e^{C_{A}}.
\end{equation}

Actually, these two inequalities yields that the free energy for $\mu_{A}$ is $\log\lambda_{A}$. Moreover, the left-side inequality yields that for any other ergodic measure $\nu$,
$$h_{\nu}+\int A\,d\nu<\log\l_{A}.$$
In particular we get $\CP(A)=\l_{A}$ and $\mu_{A}$ is the unique equilibrium state for $A$.

\medskip
The same kind of results can be get  for $\be.A$ instead of $A$. Now, the spectral gap obtained in Theorem \ref{theo-trouspect} and general results for perturbations of spectrum of operators yield that $\be\mapsto \CP(\be)$ is locally analytic. A simple   argument of connectness shows that it is globally analytic.

\begin{remark}
\label{rem-cas-conti}
The corresponding theory when $A$ is just continuous (not H\"older) is quite different (see \cite{Wal})
$\blacksquare$\end{remark}

We point out that different kinds of Transfer operators has been used in other areas and other settings: Differential and Complex Dynamics, Differential Geometry, Number Theory, Eigenvalues of the Laplacian,
Zeta functions, $C^*$-Algebras, Computer Science, Quantum Computing, Economics, Optimization, etc... (see \cite{PP}, \cite{Bow}, etc...)

%%%%%%%%%%%%%%%%%%%%%%%%%%%%%
\subsection{Some more results}\label{subsec-moreresulpressure}
The theory described above can be generalized to other cases. In the special case of $\Om$ (the full shift Bernoulli space) we can actually prove that $\l_{A}$ is the unique dominating eigenvalue.
Moreover, we get for every $\psi$ H\"older continuous,
\begin{equation}
\label{equ-trouspec}
\CL_{A}^{n}(\psi)=e^{n\CP(A)}\int\psi\,d\nu_{\be}.\phi+e^{n(\CP(A)-\eps)}\psi_{n},
\end{equation}
where $\eps$ is a positive real number (depending on $A$), $\psi_{n}$ is continuous and $\disp||\psi_{n}||_{\8}\le C. ||\psi||_{\8}$, for every $n$ and $C$ is a constant (depending on $A$).
From this one gets
$$\CP(A)=\lim_{\ninf}\frac1n\log\CL_{A}^{n}(\BBone),$$
which yields $\disp\frac{d\CP}{d\be}(\be)=\int A\,d\mu_{\be}$.

We have already  mentioned in the beginning of this chapter that for convexity reason, if $\be$ goes to $+\8$, the slope of the graph of $\CP(\be)$ goes to $\disp m(A)=\sup \int A\,d\mu$. We can now give a more precise result:

\begin{proposition}
\label{prop-asymptopress}
The graph of $\CP(\be)$ admits an asymptote if $\be$ goes to $+\8$. The slope is given by $\disp m(A)=\sup \int A\,d\mu$. Any accumulation point for $\mu_{\be}$ is a $A$-maximizing measure.
\end{proposition}
\begin{proof}
Convexity is a consequence of the definition of the pressure via a supremum. Let $\mu_{\8}$ be any $A$-maximizing measure.
Since the entropy is bounded we immediately get
$$m(A)\le \frac{h_{\mu_{\8}}}{\be}+\int A\,d\mu_{\8}\le \frac{\CP(\be)}{\be}\le
\frac{\log d}\be+\int A\,d\mu_{\be}\le  \frac{\log d}\be+m(A).$$
This proves that the asymptotic slope for $\CP(\be)$ is $m(A)$:
\begin{equation}
\label{equ-slopepress}
\lim_{\be\to+\8}\frac{\CP(\be)}{\be}=m(A).
\end{equation}
\begin{figure}[H]
\begin{center}
\includegraphics[scale=0.5]{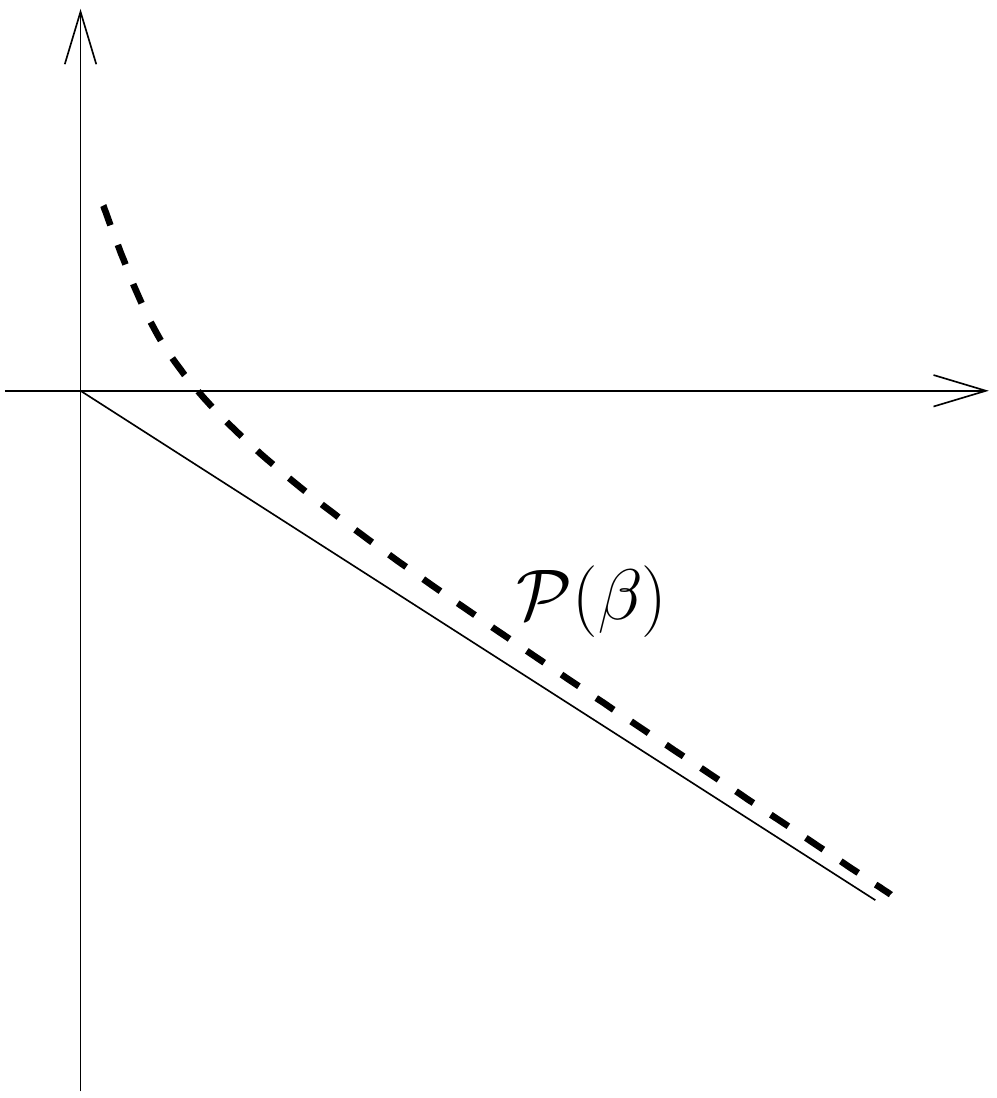}
\end{center}
\end{figure}

Now, $\be\mapsto \CP(\be)-\be m(A)$ has non-positive derivative, and, it is thus a decreasing function. It is non-negative too (as the previous series of inequalities shows), and then it admits a non-negative limit as $\be\to+\8$.

We have mentioned above $\disp\frac{d\CP}{d\be}(\be)=\int A\,d\mu_{\be}$. If we consider any accumulation point $\mu_{\8}$ for $\mu_{\be}$, $\disp\int A\,d\mu_{\be}$ converges to the slope of the asymptote, that is, to $m(A)$, and then, $\mu_{\8}$ is $A$-maximizing.
\end{proof}

We finish this subsection with an important remark.
We have seen the double-inequality \eqref{equ-def-gibbs}
$$e^{-C_{A}}\le \frac{\mu_{A}([x_{0}\ldots x_{n-1}])}{e^{S_{n}(A)(x)-n\log\lambda_{A}}}\le e^{C_{A}}.$$
We emphasize here that $C_{A}$ is proportional to $||A||_{\8}$. Therefore, replacing $A$ by $\be.A$, and letting $\be\to+\8$, this implies that $C_{\be.A}\to+\8$. In other words,  we will get that the constant explodes when $\be\to+\8$,.

Nevertheless we have:
\begin{proposition}
\label{prop-regulholde}
Assume that $A$ is $\al$-H\"older.
There exists a universal constant $C$ such that $\disp\frac1\be\log H_{\be.A}$ is $\al$-H\"older with norm bounded by $C. ||A||_{\al}$.
 \end{proposition}
 Consequently, ``at the scale '' $\disp\frac1\be\log$ we recover bounded quantities. We will elaborate more on this point in the future.

%%%%%%%%%%%%%%%%%%%%%%%%%%%%%%%%%
\subsection{An explicit computation for a particular case}
Let us now assume  that $A$ depends on two coordinates, that is
$$A(x_{0}x_{1}x_{2}\ldots)=A(x_{0},x_{1}).$$

%%%%%%%%%%%%%%%%%%%%%
\subsubsection{The exact form of the equilibrium state}
Here we show how we get an explicit expression for the equilibrium state using the Transfer Operator. Our computation does not prove that the measure we construct is an equilibrium state, and  neither proves uniqueness of the equilibrium state. However, the example gives a good intuition of the main issues on this class of problems.

We denote by $A(i,j)$ the value of $A$ in the cylinder $[ij]$, $i, j\in\{1,2,..,d\}.$
In this case, the Transfer operator  takes a simple form:

$$
\CL_{A}(\phi)(x_0x_1x_2\ldots)=\sum_{a \in \{1,2,...,d\}}
e^{ A(ax_0)}\,
\phi(ax_0x_1x_2\ldots).
$$

Let $M$ be the matrix will all positive entries given by $M_{i,j}=e^{A(i,j)}$.

\begin{lemma}
\label{lem-spectradiuCLA}
The spectral radius of $\CL_{A}$ is also the spectral radius of $M$.
\end{lemma}
\begin{proof}
Assume that $\phi$ is a function depending only on one coordinate, \ie
$$\phi(x_{0}x_{1}x_{2}\ldots)=\phi(x_{0}).$$
Then, by abuse of notation, the function $\phi$ is described by the vector
$(\phi(1), \phi(2),..., \phi(d))$. For every $j$
$$\CL_{A}(\phi)(j)=\sum_{j=1}^{s} M_{ij}.\phi(i),$$
which can be written as the matrix $M$ acting on the line vector $\phi$, that is, $\CL_{A}(\phi)=\phi\, M.$  In an alternative way we can also write $\CL_{A}(\phi)=M^{*}.\phi$, where $M^*$ denotes de transpose of the matrix $M$. This yields that the spectral radius $\l_{M}$ of $M$ is lower or equal to $\l_{A}$.

We remind that the spectral radius is given by
$$\l_{A}:=\limsup_{\ninf}\frac1n\log|||\CL_{A}^{n}|||,\text{ and }|||\CL_{A}^{n}|||=\sup_{||\psi||=1}||\CL^{n}(\psi)||_{\8}.$$
The operator $\CL_{A}$ is positive and this shows that for every $n$, $\disp |||\CL_{A}^{n}|||=||\CL_{A}^{n}(\BBone)||_{\8}$.
Now, $\BBone$ depends only on 1 coordinate, which then shows $\CL_{A}^{n}(\BBone)=M^{n}(\BBone)$. This yields $\l_{A}\le \l_{M}$.
\end{proof}

\begin{theorem}\label{perron}
{\bf (Perron-Frobenius)}  Suppose $B=(b_{ij})$ is a $d\times d$ matrix
with all entries strictly positive, $1\leq i,j\leq d$. Then,
there exists $\lambda >0$, and, vectors $\mathbf{l}=(l_1,\cdots,l_d)$ and
$\mathbf{r}=(r_1,\cdots,r_d)$ such that

\begin{enumerate}
\item  for all $i$  $l_i>0$, and  $r_i>0$.
\item For all $i$ $\sum_{j=1}^d b_{ij}r_j=\lambda r_i$ and for all $j$
$ \sum_{i=1}^d l_i b_{ij}=\lambda l_j $.
\end{enumerate}
(i.e., $\mathbf{r}$ is a right eigenvector of  $B$ and  $\mathbf{l}$ is a left eigenvector for $B$).
\end{theorem}

\begin{proof}
We first show  that there exists at least one vector  $\mathbf{r}$ with all coordinates positive, and $\lambda\geq
0$, such that
$\disp \sum_{j=1}^d b_{ij}r_j=\lambda r_i$.

 Consider the convex set  $\mathcal{H}$ of vectors
$h=(h_1,\cdots,h_d)$ such that $h_i\geq 0$, $1\leq i\leq d$ and
$\sum_{i=1}^{d}h_i=1$. The matrix $B$ determines a continuous transformation
$G:\mathcal{H}\rightarrow\mathcal{H}$, given by
$Gh=h'$, where
\begin{displaymath}
h'_i=\frac{\sum_{j=1}^{d}b_{ij}h_j}{\sum_{i=1}^{d}\sum_{j=1}^{n}b_{ij}h_j}
\end{displaymath}

Note that $h'$ has all entries strictly positive. In this way the image of $\mathcal{H}$ by $G$ is strictly inside
$\mathcal{H}$.

The Brouwer fixed point (see cap VII \cite{Lima}) assures that
there exists at least on fixed point.

As the image of $\mathcal{H}$ by $G$ is strictly inside
$\mathcal{H}$ we have that the fixed point has all entries positive.

If $\mathbf{r}$ is such fixed point, that is,
$G(\mathbf{r})=\mathbf{r}$, which means,
\begin{displaymath}
r_i=\frac{\sum_{j=1}^{d}b_{ij}r_j}{\sum_{i=1}^{n}\sum_{j=1}^{d}b_{ij}r_j}
\end{displaymath}

Taking $\lambda=\sum_{i=1}^{d}\sum_{j=1}^{d}b_{ij}r_j$, we get the right eigenvector.

Considering $B^*$ (the transpose of $B$) instead of $B$, we get  another eigenvector $\mathbf{l}$.
Note that $\langle \mathbf{r},\mathbf{l}\rangle$ is a positive number.
Then,
$$\lambda\langle \mathbf{r},\mathbf{l}\rangle=\langle B\mathbf{r},\mathbf{l}\rangle= \langle
\mathbf{r},B^*\mathbf{l}\rangle=\lambda^*\langle \mathbf{r},\mathbf{l}\rangle
$$
and $\lambda=\lambda^*$.
\end{proof}

\begin{remark}
\label{rem-finerperron}
\begin{enumerate}
\item Actually, it is possible to prove more. The eigenvalue $\lambda$ is the spectral radius of the matrix $B$ and $\mathbf{r}$ and $\mathbf{l}$ are simple eigenvectors.
\item We have seen  that in the case $A$ depends only on 2 coordinates,  the Transfer operator acts as a matrix. In the general case, Theorem \ref{theo-trouspect} extends   Theorem \ref{perron}.
\item A line stochastic matrix (see  Subsec. \ref{subsec-examp-meas}) has spectral radius equal to 1 and satisfies hypotheses of Theorem \ref{perron}. It thus admits a left-eigenvector associated to the eigenvalue 1.
\end{enumerate}
$\blacksquare$\end{remark}

\bigskip
Let us consider the matrix $d\times d$ matrix $M$ such that for each entry $M_{i,j}= e^{A(i,j)}.$
By Theorem \ref{perron}, there exist left and right eigenvectors $\mathbf{l}=(l_{1},\ldots l_{d})$ and $\mathbf{r}=(r_{1},\ldots r_{d})$, both associated to the spectral radius of $\CL_{A}$ (due to Lemma \ref{lem-spectradiuCLA}), say $\l_{A}$.

Let us define the  $2\times 2$ matrix $P_{A}=P_{A}(i,j)$ with
\begin{equation}
\label{equ-stochmatlocconst}
P_{A}(i,j)=\frac{e^{A(i,j)}r_{j}}{\lambda_{A}r_{i}}.
\end{equation}
Note that $P_{A}$ is a line stochastic matrix:
$$\sum_{j}P_{A}(i,j)=\disp\frac{\sum_{j}e^{A(i,j)r_{j}}}{\l_{A}r_{i}}=1.$$

Let us fix the normalization  in the  following way: $\sum_{j}r_{j}=1$ and $\sum_{i}l_{i}r_{i}=1$. Then, the vector $\pi=(\pi_{1},\ldots \pi_{d})$ defined by $\pi_{i}=l_{i}r_{i}$ satisfies
$$\pi. P_{A}=\pi.$$

That is, such $\pi$ is the stationary vector for $P_A$.

The right-eigenvector $\mathbf{r}$ has to be seen as the eigenmeasure $\nu_{A}$ from Theorem \ref{th-eigenmeas}, and actually $r_{j}=\nu_{A}([j])$. The normalized left-eigenvector $l$ has to be seen as the eigenfunction $H_{A}$ from Theorem \ref{theo-trouspect}, with normalization $\disp\int H_{A}\,d\nu_{A}=1$.

The associated invariant Markov measure $\mu_{A}$ is defined by
$$\mu_{A}([x_{0}\ldots x_{n-1}])=\pi_{x_{0}}P_{A}(x_{0},x_{1})\ldots P_{A}(x_{n-2},x_{n-1}).$$
The exact computation yields
$$\disp \mu_{A}([x_{0}\ldots x_{n-1}])=\phi_{A}(x_{0})e^{S_{n}(A)(x)-n\log\l_{A}}.\nu_{A}([x_{n-1}]).$$

Since $\nu$  and $\phi_{A}$ have positive entries, this shows that $\mu_{A}$ is a Gibbs measure.

Things can be summarized as follows:

\begin{center}
 \begin{tabular}{|p{10cm}|}
\hline
Let $M=(M_{ij})$ be the matrix with entries $e^{A(i,j)}$. Let $\mathbf{r}=(r_{1},\ldots r_{d})$ be the right-eigenvector associated to $\lambda$ with normalization $\sum r_{i}=1$. Let $\mathbf{l}=(l_{1},\ldots, l_{d})$ be the left-eigenvector for $\lambda$ with renormalization  $\sum l_{i}r_{i}$=1. Then, $\mathbf{r}$ gives the eigenmeasure  $\nu_{A}$, and $\mathbf{l}$ gives the density $H_{A}$. \\
The Gibbs measure of the cylinder $[x_{0}\ldots x_{n-1}]$ is
$\mu_{A}([x_{0}\ldots x_{n-1}])=l_{x_{0}}e^{S_{n}(A)(x)-n\log\,\lambda}r_{x_{n-1}}$\\
\hline
\end{tabular}
\end{center}

\begin{remark}
\label{rem-mesmaxentro}
Note that in the case $A\equiv0$, the measure $\mu_A$ is the Markov probability associated to the line stochastic matrix
with all entries equal to $\frac{1}{d}.$ We denote by the probability $\mu_{top}$ the  measure of maximal entropy.
It is the unique invariant probability with entropy  $\log d$.
$\blacksquare$\end{remark}

%%%%%%%%%%%%%%%%%%%%%%%%%
\subsubsection{About uniqueness}
This subsection can be avoid in a first lecture. The reasoning bellow can be seen as the description of a simple computable way  to get an equilibrium state, and, moreover, to show that it is the unique one.

First, we remind
$$\mu_{A}([i_{0}\ldots i_{n-1}])=l_{i_{0}}e^{S_{n}(A)(x)-n\log\l_{A}}r_{i_{n-1}}.$$
Taking the logarithm of this equality and dividing by $n$, we get
$$\frac1n\log\mu_{A}([x_{0}\ldots x_{n-1}])=\frac1n S_{n}(A)(x)-\log\l_{A}+\frac1n(\log l_{x_{0}}+\log r_{x_{n-1}}).$$
If $n\to+\8$, the left hand side term goes to $-h_{\mu_{A}}$ (for $\mu$ generic $x$). The first summand in the right hand side goes to $\disp\int A\,d\mu_{A}$, by Birkhoff theorem. The last summand goes to zero because $l_{x_{0}}$ is a constant and $\mathbf{r}$ has all its entries bound away from 0.  This shows that
$$h_{\mu_{A}}+\int A\,d\mu_{A}=\log\l_{A}.$$

We will now show that for every Markov measure $\mu$
\begin{equation}
\label{equ-pressuremarkov}
h(\mu)  +\int A d\mu \leq \lambda(A).
\end{equation}

This will show that  the probability $\mu_A$  realizes the supremum of the free energy among the Markov measures (see \cite{Spit}).  This is a restrictive result (with respect to the initial problem of maximization of the free energy among "all" invariant probabilities) but this gives a hint of the
main issues and ideas we are interested here. Actually, the key point is the relative entropy of two measures (see \cite{krengel}).

We need the next lemma whose proof can be obtained in \cite{PY}.
\begin{lemma}\label{lem-techpiqi}
If $p_1,\cdots,p_d$ and $q_1,\cdots,q_d$, satisfy\footnote{By convention $0\log{0}=0$.}
$$\sum_{j=1}^d p_i=\sum_{j=1}^d  q_j=1,$$
with $p_i \geq 0$, $q_i \geq 0$, $i=1,\cdots,d$,
then,
$$\sum_{i=1}^d q_i\log{q_i} - \sum_{i=1}^d q_i\log{p_i}=\sum_{i=1}^d q_i\log{\frac{q_i}{p_i}} \geq 0,
$$
and the equality holds, if and only if, $p_i=q_i$, $i=1\cdots,d$.
\end{lemma}

Consider $\mu$ a Markov measure.
We fix
the vector $\pi=(\pi_1,...,\pi_d)$,  and the matrix $P$,  satisfying
$\pi \, P = \pi$ (which defines this Markov measure $\mu$).

Equation \eqref{equ-stochmatlocconst} yields $A(i,j)=\log P_{A}(i,j)+\log\l_{A} +\log r_{i}-\log r_{j}$.
By abuse of notation we set $r_{i}=r(x)$ if $x\in [i]$. Then, we get
\begin{eqnarray*}
h_{\mu}+\int A\,d\mu&=&-\sum_{i}\pi_{i}P_{ij}\log P_{ij}+\int \log P_{A}(x_{0},x_{1})\\
&&+\log\l_{A}+\int \log r-\log r\circ\s\,d\mu\\
&=& -\sum_{i,j}\pi_{i}P_{ij}\log P_{ij}+\sum_{i,j}\pi_{i}P_{ij}\log P_{A}(i,j)+\log\l_{A}\\
&=& -\sum_{i}\pi_{i}\sum_{j}P_{ij}\log P_{ij}-P_{ij}\log P_{A}(i,j)+\log\l_{A}\\
&\le & \log \l_{A},
\end{eqnarray*}
by Lemma \ref{lem-techpiqi}. Then, \eqref{equ-pressuremarkov} holds.

%%%%%%%%%%%%%%%%%%%%%%%
%%%%%%%%%%%%%%%%%%%%%%%%%%
%%%%%%%%%%%%%%%%%%%%%%%%%%%%%%%%
\chapter{Maximizing measures and Ground states}\label{chap-maxmeas}

%%%%%%%%%%%%%%%%%%
%%%%%%%%%%%%%%%
\section{Selection at temperature zero}

%%%%%%%%%%%%%%%%%%%%%
\subsection{The main questions}

We remind Definition \ref{def-meas-max} which address  to the concept  of $A$-maximizing measure: it is a $\s$-invariant probability $\mu$ such that
$$\int A\,d\mu=\max_{\nu\in\CM_{\s}}\int A\,d\nu.$$
Existence of maximizing probabilities follows from the compactness of $\CM_{\s}$
(and the continuity of the map $\nu \mapsto \int A d \nu$).

The first kind of problems we are interested
in are related to maximizing measures. We can for instance address the questions:
\begin{enumerate}
\item For a given potential $A$, how large is the set of maximizing measures ?
\item How can we construct/get maximizing measures ?
\item For a given maximizing measure, what we can say about its support ?
\end{enumerate}

We have already mentioned above the relation between maximizing measures and equilibrium states (see Proposition \ref{prop-asymptopress}). Assume that $A$ is H\"older continuous, then, any accumulation point for $\mu_{\be}$ is a $A$-maximizing measure. We point out that, for  simplicity, from now on, we replace all the subscribed $\be. A$  by $\be$.

The relation between equilibrium states and maximizing measures motivates a new definition:
\begin{definition}
\label{def-groundstate}
Let $A$ be H\"older continuous.
A $\s$-invariant probability measure $\mu$ is called a {\em ground state} (for $A$) if it is an accumulation point for $\mu_{\be}$, as $\be$ goes to $+\8$.
\end{definition}
\begin{remark}
\label{rem-groundstate}
The H\"older continuity of $A$ is only required to get uniqueness for the equilibrium state for $\be.A$.
$\blacksquare$\end{remark}

After last chapter, clearly a ground state is a $A$-maximizing measure, but {\em a priori} a maximizing measure is not necessarily a ground state. The natural questions we can address are then:
\begin{enumerate}
\item Is there convergence of $\mu_{\be}$, as $\be$ goes to $+\8$ ?
\item If there is convergence, how does the family of measures $\mu_{\be}$ select the limit ?
\end{enumerate}

If  the limit for $\mu_{\be}$, as $\be$ goes to $+\8$, exists,  we say that $\mu=\lim_{\binf}\mu_{\be}$ is {\bf selected
when temperature goes to zero}.

We remind that if $\mu_{\8,1}$ and $\mu_{\8,2}$ are two invariant maximizing probabilities, then any convex combinations of these two measures is also a maximizing measure. This gives a hint of how difficult the problem of selection is. There are too many possibilities.

\bigskip
In case of convergence,
one natural question is then to study the speed of convergence for $\mu_{\be}$. More precisely, if $C$ is a cylinder such that $C\cap\CA=\emptyset$, then $\disp\lim_{\be\to+\8}\mu_{\be}(C)=0$ (because any possible accumulation point gives 0-measure to $C$). It is thus natural to study the possible limit:
$$\lim_{\be\to+\8}\frac1\be\log(\mu_{\be}(C)).$$

We will see that this question  is related to the behavior of $\disp\frac1\be\log H_{\be}$. Then we address the question:
\begin{enumerate}
\item Is there convergence of the family $\disp\frac1\be\log H_{\be}$, as $\be$ goes to $+\8$?
\item If there is convergence, how does the family of functions $ \disp\frac1\be\log H_{\be}$ select the limit ?
\end{enumerate}

In the case $V= \lim_{\beta \to \infty} \disp\frac1\be\log H_{\be}$ we say that {\bf $V$ was selected when temperature goes to zero}.

%%%%%%%%%%%%%%%%%%%%%%%%%%%%%%%%%%%%%%%%%%%
%%%%%%%%%%%%%%%%%%%%%%%%%%%%%%%%%%%%%%%%%%%
\subsection{A naive computation. The role of subactions}

We explore here
a different but related question:
consider $A:\Om\to\R$ Lipschitz  and $\be>0$. The transfer operator yields for every $x$,
\begin{equation}
\label{equ-opera-hbeta}
e^{\CP(\be)}H_{\be}(x)=\sum_{i=1}^{d}e^{\be.A(ix)}H_{\be}(ix),
\end{equation}
where all the subscribed $\be. A$ have been replaced by $\be$ for  simplicity. $H_{\be}$ denotes the main eigenfunction (properly normalized) for
$\be A$.

We have seen that, if $\be$ goes to $+\8$, there is control on the Lipschitz (or, H\"older)  constant, more precisely,  this control exists at the $\disp\frac1\be\log$-scale.

Actually, Proposition \ref{prop-regulholde} shows that the $\disp\frac1\be\log H_{\be}$ form an equicontinuous family. There are different choices for normalizing the family of eigenfunction of the Ruelle operator in order to get that the family is bounded (independent of $\beta$), but we will not elaborate on this in the moment.  We can consider an accumulation point of this sequence, say $V$, and we can show  it is Lipschitz continuous. For simplicity we keep writing $\be\to+\8$, even if we actually consider subsequences.
Then, taking $\disp\frac1\be\log$ of \eqref{equ-opera-hbeta}, considering $\be\to+\8$, equality \eqref{equ-slopepress} (and some more effort) implies
\begin{equation}
\label{equ1-subcohomo}
m(A)+V(x)=\max_{i}\left\{V(ix)+A(ix)\right\}.
\end{equation}

\bigskip

We will show in a moment  a proof of this result in a particular case

\bigskip
\begin{definition}\label{def-sub}
A continuous function $u: \Omega \to \mathbb{R} $ is called a
{\em calibrated subaction} for $A:\Omega\to \mathbb{R}$, if for any $x\in \Omega$, we have

\begin{equation}\label{c} u(x)=\max_{\sigma(y)=x} [A(y)+ u(y)-m(A)].\end{equation}

\end{definition}

We will show the proof of a particular case of the general result we mentioned above.

\begin{theorem}
Assume  $A$ is a potential which depends on two coordinates.
Suppose $\beta_n$ is such that for each $j$
$$ \lim_{\beta_n \to \infty} \frac{1}{\beta_n} \log l^{\beta_n\,A}(j)= V(j).$$

Then, $V$ is a calibrated subaction.
\end{theorem}

To get this result we
consider the eigenvalue equation: for each $\beta$ and for each $1\leq j\leq d$, we have
$$\sum_{i=1}^d \frac{l_i^{\beta A} \,e^{\beta\, A(i,j)}}{ \lambda(\beta\, A) \,\,l_j^{\beta A}}=1 .$$

For each $j$ and  $\beta_n$, there exist a $i_n=i_n^j$, such that, $\frac{l_{i_n}^{\beta A} \,e^{\beta\, A(i_n,j)}}{ \lambda(\beta\, A) \,\,l_j^{\beta A}}$ attains
the maximal possible value among the $i\in\{1,2.,,.d\}$.

For a fixed $j$ there exists a $i_j\in\{1,2.,,.d\}$ such that the value $i_n^j$ is attained an infinite number of times.

\bigskip

Remember that
$\lim_{\beta \to \infty} \frac{P(\beta\, A)}{\beta}= \lim_{\beta \to \infty} \frac{\log (\lambda(\beta\, A))}{\beta}=m(A).$

\bigskip
Therefore,

$$0=\frac{1}{\beta_n} \, \log (\sum_{i=1}^d \frac{l_i^{\beta_n A} \,e^{\beta_n\, A(i,j)}}{ \lambda(\beta_n\, A) \,\,l_j^{\beta_n A}}) \leq$$
$$\frac{1}{\beta_n} \, \log (d \,\frac{l_{i_j}^{\beta_n A} \,e^{\beta_n\, A(i_j,j)}}{ \lambda(\beta_n\, A) \,\,l_j^{\beta_n A}}) \leq$$
$$ A(i_j,j) +  \frac{1}{\beta_n} \log l^{\beta_n\,A}(i_j)-  \frac{1}{\beta_n} \log l^{\beta_n\,A}(j)- \frac{1}{\beta_n} \log \lambda( \beta_n A).$$

Now, taking limit in the above expression, when $n\to \infty$, we get for each fixed $j$
$$0\leq A(i_j,j) +V(i_j) - V(j) - m(A).$$

In this way, for each $j$ we get
$$V(j)+ \,m(A)\leq \max_{i \in \{1,2,...,d\}} \{A(i,j)+  V (i)  \}\, .$$

In the case, for a given $j$ there exists $i$ and $\epsilon>0$, such that,
$$V(j)+ \,m(A)+  \epsilon<A(i,j)+  V (i) \, ,$$

then,

$$1< \frac{l_i^{\beta_n A} \,e^{\beta_n\, A(i,j)}}{ \lambda(\beta_n\, A) \,\,l_j^{\beta_n A}},$$
for some large $n$. But, this is a contradiction.

\qed

In the case the maximizing probability for $A$ is unique, then, from \cite{CLT} \cite{GL1}, there exists the limit
$$ \lim_{\beta \to \infty} \frac{1}{\beta} \log l^{\beta\,A}=V.$$

Consequently, any accumulation point for $\disp\frac1\be\log H_{\be}$ is a calibrated subaction.
Now, Equality \eqref{equ1-subcohomo} yields for every $i$ and every $x$,
$$A(ix)\le m(A)+V(x)-V(ix),$$
which can be rewritten as $A(y)=m(A)+V\circ\s(y)-V(y)+g(y),$
where $g$ is a non-positive function. By Proposition \ref{prop-regulholde} it is also Lipschitz continuous.

Note that in the case $u$ is a calibrated subaction, then, $u+c$, where $c$ is a constant, is also a calibrated subation.
We say that the calibrated subaction is unique, if it is unique up to an additive constant. One can show that in the case there is more than one maximizing probability, the calibrated subaction is not unique.

\bigskip

A calibrated subaction $u$ satisfies
$$ u(\sigma(x)) -  u(x) - A(x) + m(A) \geq0.$$

Remember that if $\nu$ is invariant for $\sigma$, then for any continuous function $u:\Omega\to \mathbb{R}$ we have
$$ \int \,[ u(\sigma(x)) -  u(x)]\, d \nu=0.$$

Suppose $\mu$ is maximizing for $A$ and $u$ is a calibrated subaction for $A$.
It follows that for any $x$ in the support of $\mu_\infty$ we have
\begin{equation}\label{funcaoR} u(\sigma(x)) -  u(x) - A(x) + m(A)=0. \end{equation}
Indeed,
$g(x) \,=\,u(\sigma(x)) -  u(x) - A(x) + m(A) \geq0$, and the integral $\int g(x) d \mu(x)=0$.

 In this way if we know the value $m(A)$, then a calibrated subaction $u$ for $A$ help us to identify the support of maximizing probabilities. The equality to zero in the above equation can be true outside the union of the supports of the  maximizing probabilities  $\mu$ (see example of R. Leplaideur in \cite{LOS}). It is known that generically on the Holder class on $A$ the equality is true just on the support of the maximizing probability.

The study of the ergodic properties of maximizing probabilities is the purpose of  Ergodic Optimization.

\bigskip

Now we recall some questions already stated in the introduction of
this chapter:
\begin{enumerate}
\item Is there convergence for $\disp\frac1\be\log H_{\be}$ as $\be$ goes to $+\8$ ?
\item If there is convergence, how does the family of functions$ \disp\frac1\be\log H_{\be}$ select the limit ?
\end{enumerate}

We also remind the reader that
in the case $V= \lim_{\beta \to \infty} \disp\frac1\be\log H_{\be}$ we say that {\bf $V$ was selected when temperature goes to zero}.

Given a general H\"older potential $A$
can be exist calibrated subactions which are not selected. The ones that are selected are special among the calibrated subactions.
\bigskip

\Eg For  $\Omega=\{1,2\}^{\N}$ we study the convergence and the selection at temperature zero.

Suppose $A$ is a two by two matrix.

If $A(1,1)$ is strictly bigger than the other $A(1,j)$ there is a unique maximizing probability with support on $\delta_{1^\infty}$.
There is only one calibrated subaction up to an additive constant, therefore we get selection of subaction and probability.
A similar result happen when $A(2,2)$ is strictly bigger than all the other $A(i,j)$.

Suppose now that $A(1,1)=0=A(2,2)$, and, moreover that $A(1,2),A(2,1)<0.$

Using the notation of previous sections we  have that $H_i^\beta=l^\beta_i$, $i=1,2$. The equation for the right eigenvalue $r^\beta_i$  of $A$ is
the equation of the left eigenvalue for the symmetric matrix $A^t$.

Then,
we get the system
\begin{equation}\label{A1}
(\lambda_\beta - 1 ) H_1^\beta = e^{\beta A(2,1)} H_2^\beta,
\end{equation}

\begin{equation}\label{A2}
(\lambda_\beta - 1 ) H_2^\beta = e^{\beta A(1,2)} H_1^\beta
\end{equation}

The trace of the matrix with entries $e^{\beta A(i,j)}$ is $2$ and the determinant is $ 1 - e^{\beta\,(A(1,2)+A(2,1))}$.
Solving a polynomial of degree two we get the maximal eigenvalue
$$\lambda_\beta= 1+ \sqrt{e^{\beta [A(1,2)+A(2,1)]}}.$$

We can take $H_2^\beta= 1$ and $H_1^\beta = e^{\beta\, \frac{1}{2}\, [A(2,1)-A(1,2)]}.$

In this case $V(2)=\lim_{\beta \to \infty} \frac{1}{\beta}\log H_2^\beta=0$ and, moreover, we have  $V(1)=\lim_{\beta \to \infty}\frac{1}{\beta}\log H_1^\beta= \frac{1}{2}\, [A(2,1)-A(1,2)].$

In this case we have selection of subaction assuming the normalization $H_2^\beta= 1$ for all $\beta$.

An easy computation shows that in this case $l_1^\beta =r_2^\beta $ and $l_2^\beta =r_1^\beta $. We can assume the normalization $r_1^\beta+r_2^\beta=1$ and $l_1^\beta \,r_1^\beta+ l_2^\beta \,r_2^\beta=1$.

Therefore, $\mu_\beta( [1])=l_1^\beta \,r_1^\beta= l_2^\beta \,r_2^\beta=\mu_\beta( [2]).$ In this case we have selection of the probability $\frac{1}{2} \delta_{1^\infty} + \frac{1}{2} \delta_{2^\infty}$ in the zero temperature limit.

\bigskip

Concerning the first question which concerns  measures (question 1), there is one known example of non-convergence (see \cite{CH}), when $\beta \to \infty$.
Let us now gives some more  details in the last question. Obviously, if there exists a unique maximizing measure,
we have convergence because there is a unique possible accumulation point. In this case the question is not so interesting.

Let us then
assume that there are at least two different maximizing measures $\mu_{max,1}$ and $\mu_{max,2}$. Just by linearity, any convex combination of both measure $\mu_{t}=t\mu_{max,1}+(1-t)\mu_{\max,2}$, $t\in[0,1]$, is also a $A$-maximizing measure.
The question of  selection is then to determine why the family (or, even a subfamily) choose to converge to some specific limit, if there are so many possible choices.

 Among maximizing measures a special subset is the one whose elements are  {\em ground states},
a notion (and a terminology) that  comes from
Statistical Mechanics,  as we said before. One of the most famous examples is the so called
spontaneous magnetization:
consider a one-dimensional lattice with spins which can be up (say +1) or down (say -1). At high temperature there is a lot of randomness in the spins. But, if the temperature decreases, then, the behavior of the  spins change in the lattice, and, at very low temperature, a given arrangement of spins in the lattice, is such that  they are all up or they are all down. This procedure modifies the magnetic properties of the material.
In Statistical Mechanics $\be$ is the inverse of the temperature, thus, $\be\to+\8$ means to approximate zero temperature.
The main goal is thus to develop mathematical tools for understanding why materials have a strong tendency  to be highly ordered at low temperature. They reach a crystal or quasi crystal configuration. This is also called the spontaneous magnetization when temperature decreases. An invariant probability with support in a union of periodic orbits plays the role of a "magnetic state".

A main conjecture in Ergodic Optimization claims that for a generic potential in the Holder topology the maximizing probability is a unique periodic orbit (for a partial result see \cite{CLT}).

Note that Gibbs states for H\"older potentials give positive mass to any cylinder set. In this case, for any positive temperature, there is no magnetization. In the limit when temperature goes to zero the Gibbs state may, or may not,  split in one or more ground states (which can have, or not, support in periodic orbits). All these questions are important issues here.

We point out that when the potential is not of H\"older class it is
possible to get "Gibbs states" which have support in periodic orbits
even at positive temperature.

Questions about selection of subactions naturally appears in this problem.

%%%%%%%%%%%%%%%%%%%%%%%%%%%%%%%%%%%%%%%%%%%%%
%%%%%%%%%%%%%%%%%%%%%%%%%%%%%%%%%%%%%%%%%%%%%
\subsection{Large deviations for the family of Gibbs measures}
\label{ldev}

In the study of Large Deviations when temperature goes to zero one is interested in the limits of the following form:
\begin{equation} \label{ldp} \lim_{\beta
\to \infty} \frac{1}{\beta} \log \mu_{\beta} (C),
\end{equation}
where $C$ is a
fixed cylinder set  on $\Omega$.

In principle, the limit may not exist.
We remind that general references in Large deviations are \cite{DZ} \cite{Ellis}).

\begin{definition}
\label{def-GDPtzero}
We say there exists a Large Deviation Principle for the one parameter family $\mu_\beta,$ $\beta>0$, if there exists a non-negative function $I$, where $I:\Sigma \to \mathbb{R} \cup \{\infty\}$ (which can have  value equal to infinity in some points), which it is lower semi-continuous, and satisfies the property that for any cylinder set $C\subset \Sigma$,
$$
\lim_{\beta \to +\infty} \frac{1}{\beta} \log
\mu_{\beta\, A}(C)=-\inf_{ x \in C} I(x).
$$
\end{definition}

In the affirmative case an important point is to be able to identify such function $I$.

\begin{theorem}[see \cite{BLT,LM3}]
\label{theo-GD}
Assume that $A$ admits a unique maximizing measure. Let $V$ be a calibrated subaction.
Then, for any cylinder $[i_0 i_1
...i_n]$, we have
$$
\lim_{\beta
\to \infty} \frac{1}{\beta} \log \mu_{\beta} ([i_0 i_1
...i_n])=-\inf_{x\in [i_0 i_1
...i_n]} \{ I(x)\},$$

where
$$I(x) =\sum_{n=0}^\infty \, [ \,V\, \circ \,\sigma - V - (\,A-m(A))\,]\,
\sigma^n\, (x).
$$

\end{theorem}

In the proof of the above theorem in \cite{BLT} it is used the involution kernel which will be briefly  described in a future section. For a proof of this result without using the involution kernel see \cite{LM3}.
\medskip

In the case the potential $A$ depends on two coordinates and the maximizing probability is unique we get
$$ I(x)=I(x_0,x_1...,x_k,..)=\sum_{j=0}^\infty[\,V(x_{j+1})-V(x_j) - A(x_j,x_{j+1})+ m(A)\,] .$$

\Eg  Assume that the potential $A$ depends only on two coordinates,
say $A(x) = A(x_1, x_2)$.

We have seen that in that case
$$\mu_{\be}=([i_{0}\ldots i_{n-1}])=l^{\be}_{i_{0}}e^{\be.S_{n}(A)(x)-n\log\l_{\be.A}}r^{\be}_{i_{n-1}},$$
where $\mathbf l^{\be}=(l^{\be}_{1},\ldots, l^{\be}_{d})$ and $\mathbf r^{\be}=(r^{\be}_{1},\ldots, r^{\be}_{d})$ are respectively the left eigenvector and the right eigenvector for the matrix with entries $e^{\be. A(i,j)}$ and $\l_{\be.A}$  is the associated eigenvalue (and the spectral radius of the matrix).

Studying the Large Deviation,  we are naturally led to set

$$\lim_{\beta
\to \infty} \frac{1}{\beta} \log \mu_{\beta} ([i_0 i_1
...i_{n-1}])=[\,A(i_0,i_1) +...+A(i_{n-2},i_{n-1})\,] +
$$
\begin{equation}\label{valueIlr}
 \lim_{\beta
\to \infty} \frac{1}{\beta}  \log \,l^\beta_{(i_0)i_{0}}- \lim_{\beta
\to \infty} \,\frac{1}{\beta}\log \,r^\beta_{i_{n-1}}  -n. m(A).
\end{equation}

The main problem in the above discussion is the existence of the limits, when $\beta\to \infty.$

 {\bf In the  case} there exists the limits

$$ \lim_{\beta
\to \infty} \frac{1}{\beta}  \log \,l^\beta=V$$
and
$$ \lim_{\beta
\to \infty} \frac{1}{\beta}  \log \,r^\beta=V^*,$$
we get the explicit expression
$$
\lim_{\beta
\to \infty} \frac{1}{\beta} \log \mu_{\beta} ([i_0 i_1
...i_{n-1}])=$$
\begin{equation}
\label{ldpVV}
[\,A(i_0,i_1) +...+A(i_{n-2},i_{n-1})\,] +
V (i_0)- V^* (i_{n-1})  -n. m(A).
\end{equation}

There is a need for understanding better the relation of $l^\beta$ and $r^\beta$. This is sometimes not so easy as we will see in an example which will be presented in the  last chapter.

\medskip

%%%%%%%%%%%%%%%%%%%%%%%%%%%%%%%
\subsection{Uniqueness and lock-up on periodic orbits}

 One can ask questions about the behavior of typical trajectories of maximizing probabilities. From the classical result given by Birkhoff's  we know that  in the case $\mu$ is a maximizing measure, then, for $\mu$-almost every $x$,
$$
 \lim_{\ninf}\frac1n\sum_{j=0}^{n-1}A\circ \s^{j}(x)=m(A) = \int A d \mu.
$$

In this case we say that this orbit beginning in $x$ is $A$-maximizing.

It is natural to analyze the problem under this viewpoint:
\begin{enumerate}
\item How can we detect that an certain given orbit is $A$-maximizing?
\item What is the relation  of maximizing measures with maximizing orbits?
\item What are the main properties of the $A$-maximizing orbits?
\end{enumerate}

\noindent
We will see that this point of view will  produce methods which will help to solve the cohomological inequality:
\begin{equation}
\label{equ-subactions}
A\ge m(A)+V\circ T-V.
\end{equation}

\begin{definition}
\label{def-cobord}
A coboundary is a function of the form $\psi\circ \s-\psi$.
\end{definition}
\exo Show that a coboundary has zero integral for any $\si-$invariant measure.

\bigskip
We mentioned above the notion of crystals or quasi-crystals.
In our settings crystals means periodic orbits; we will not deal with the question of quasi-crystal
in this text.
Consequently, an important point we want to emphasize here is the role of the periodic orbits. Nevertheless, the regularity of the potential is of prime importance in that study. The analogous results for the class of Continuous potentials and the class of H\"older Continuous potentials can be very different.

\begin{theorem}[see \cite{Bousch-walters}]
\label{theo-uniqumaxmeas}
Generically for the $\CC^{0}$-norm the potential $A$ has a unique maximizing measure. This measure is not-supported in a periodic orbit.
\end{theorem}
We shall give the proof of this theorem inspired from \cite{Bousch-walters}. We emphasize that this proofs can be extended (concerning the uniqueness of the maximizing measure) to any separable space.
A proof of generic uniqueness for H\"older continuous functions can be found in \cite{CLT}.
\begin{proof}
The set $\CC^{0}(\Om)$ is separable. Let $(\psi_{n})$ be a dense sequence in $\CC^{0}(\Om)$.
Two different measures, say $\mu$ and $\nu$, must give different values for some $\psi_{n}$.
This means
\begin{eqnarray*}
\left\{A,\ \#\CM_{max}>1\right\}&=&\left\{A,\ \exists n\ \exists \mu,\nu\in\CM_{max}\ \int\psi_{n}\,d\nu\neq\int \psi_{n}\,d\mu\right\}\\
&=& \bigcup_{n}\left\{A,\  \exists \mu,\nu\in\CM_{max}\ \int\psi_{n}\,d\nu\neq\int \psi_{n}\,d\mu\right\}\\
&\hskip -3cm=& \hskip -1.5cm\bigcup_{n}\bigcup_{m}\left\{A,\ \exists \mu,\nu\in\CM_{max}\ \left|\int\psi_{n}\,d\nu-\int \psi_{n}\,d\mu\right|\ge\frac1m\right\}.
\end{eqnarray*}
Set $\disp F_{n,m}:=\left\{A,\ \exists \mu,\nu\in\CM_{max}\ \left|\int\psi_{n}\,d\nu-\int \psi_{n}\,d\mu\right|\ge\frac1m\right\}$.
We claim that theses sets are closed. For this we need a lemma:
\begin{lemma}
\label{lem-comeasmax}
Let $(A_{k})$ be a sequence of continuous potentials converging to $A$. Let $\mu_{k}$ be any maximizing measure for $A_{k}$ and $\mu$ be an accumulation point for $\mu_{k}$.

Then,
$\disp\lim_{k\to+\8}m(A_{k})=m(A)$ and $\mu$ is a $A$-maximizing measure.
\end{lemma}
\begin{proof}[Proof of Lemma \ref{lem-comeasmax}]
For any $\eps>0$, and for $k$ sufficiently large,
$$A-\eps\le A_{k}\le A+\eps.$$
This shows  $m(A)-\eps\le m(A_{k})\le m(A)+\eps$.
Furthermore, we have
$\disp m(A_{k})=\int A_{k}\,d\mu_{k}$, $\disp m(A)=\int A\,d\mu$
and (up to a subsequence), $\disp \lim_{k\to+\8}\int A_{k}\,d\mu_{k}=\int A\,d\mu$, because $\mu_{k}$ converges to $\mu$ on the weak* topology and $A_{k}$ goes to $A$ on the strong topology.
\end{proof}
We can thus show that the set $F_{n,m}$ is closed in $\CC^{0}(\Om)$.
Indeed, considering a sequence $A_{k}$ converging to $A$ (for the strong topology), we get two sequences $(\mu_{k})$ and $(\nu_{k})$ of $A_{k}$-maximizing measures such that
$$\left|\int \psi_{n}\,d\mu_{k}-\int\psi_{n}\,d\nu_{k}\right|\ge \frac1m.$$
We pick a subsequence such that $\mu_{k}$ and $\nu_{k}$ converge  for this subsequence. Lemma \ref{lem-comeasmax} shows that the two limits, say $\mu$ and $\nu$,  are $A$-maximizing and they satisfy
$$\left|\int \psi_{n}\,d\mu-\int\psi_{n}\,d\nu\right|\ge \frac1m.$$
To complete the proof concerning generic uniqueness, we need to prove that the sets $F_{n,m}$ have empty interior.

For such $A$, we claim that the function $\eps\mapsto m(A+\eps.\psi_{n})$ is convex but not differentiable at $\eps=0$. Indeed, assume $\mu$ and $\nu$ are $A$-maximizing and
$$\int\psi_{n}\,d\mu\ge \int\psi_{n}\,d\nu+\frac1m,$$ then the right derivative is bigger than $\disp\int \psi_{n}\,d\mu$, and, the left derivative is lower than $\disp\int\psi_{n}\,d\nu$.

It is known that a convex function is differentiable Lebesgue everywhere (actually everywhere except on a countable set), which proves that there are infinitely many $\eps$ accumulating on $0$ such that $A+\eps.\psi_{n}$ cannot be in $F_{n,m}$.

\medskip
Let us now prove that, generically, the unique maximizing measure is not supported on a periodic orbit.

Let us consider some periodic orbit $\CO$ and $\mu_{\CO}$ the associated invariant measure. If $A$ is such that $\mu_{\CO}$ is not $A$-maximizing, then for every $A_{\eps}$ closed, $\mu_{\CO}$ is still not $A_{\eps}$-maximizing (see the proof of Lemma  \ref{lem-comeasmax}). This proves that the set of $A$ such that $\mu_{\CO}$ is $A$-maximizing is a closed set in $\CC^{0}$.

 In order to prove that it has empty interior, let consider $A$ such that $\mu_{\CO}$ is $A$-maximizing, and some measure $\mu$ closed to $\mu_{\CO}$ in the weak* topology. This can be chosen in such way  that $\disp\int A\,d\mu\ge m(A)-\eps$, with $\eps>0$ as small as wanted.

There exists a set $K_{\eps}$ such that $\mu_{\CO}(K_{\eps})<\eps$ and $\mu(K_{\eps})>1-\eps$, and, then, we can find a continuous function $0\le \psi_{\eps}\le 1$, null on the periodic orbit $\CO$, and, such that, $\disp\int \psi_{\eps}\,d\mu>1-2\eps$.

Then,
$$\int A+2\eps\psi_{\eps}\,d\mu_{\CO}= m(A)<m(A)-\eps+2\eps-4\eps^{2}\le \int A+2\eps.\psi_{\eps}\,d\mu_{\eps}.$$
This shows that $\mu_{\CO}$ is not $A+2\eps\psi_{\eps}$-maximizing, thus the set of potentials such that $\mu_{\CO}$ is maximizing has empty interior.

As there are only countably many periodic orbits, this proves that generically, a periodic orbit is not maximizing.
\end{proof}

It is known that generically on the Holder class the maximizing probability for $A$ is unique (see \cite{CLT}).

\bigskip
Let us thus mention an important conjecture:
\paragraph{Conjecture.}{\it Generically for the Lipschitz norm, the potential $A$ has a unique maximizing measure and it is supported by a periodic orbit.}

\medskip\noindent
The main results  in that direction are:

\begin{theorem}[see \cite{Yuan-Hunt}]
\label{th-yuan-hunt}
Let $\mu$ be a maximizing measure for the Lipchitz potential $A$. Assume $\mu$ is not supported by a periodic orbit. Then, for any $\eps>0$, there exists $A_{\eps}$ $\eps$-closed to $A$ for the Lipschitz norm such that $\mu$ is not $A_{\eps}$-maximizing.
\end{theorem}

\begin{theorem}[see \cite{CLT}]
\label{th-lockupperio}
Let $A$ be Lipschitz which has  a unique $A$-maximizing measure $\mu$. Assume that $\mu$ has periodic support. Then, there exists $\eps$ such that for every $A_{\eps}$ $\eps$-closed to $A$ (for the Lipschitz norm), $\mu$ is the unique $A_{\eps}$-maximizing measure.
\end{theorem}

Theorems \ref{th-lockupperio} and \ref{th-yuan-hunt} show that the unique possibility for a maximizing measure to keep being maximizing for small perturbation in the Lipschitz norm is that it is supported by a periodic orbit. This is the lock-up on periodic orbits. The extended version of Theorem \ref{theo-uniqumaxmeas} for the Lipschitz norm would imply that generically there is a unique maximizing measure. Nevertheless, this conjecture is not yet proved on the full extent. A partial result in this direction is \cite{CLT}.

Questions about ground states which have support on periodic orbits are questions about magnetization at temperature zero.

 \bigskip
 On the other hand, it is extremely simple to get an example with
  non-uniqueness for the maximizing measure. Let $\K$ be any $\s$-invariant compact set such that it contains the support of at least two different invariant measure (in other words $\K$ is not uniquely ergodic). Then set
 $$A:=-d(., \K).$$
 Then, $A$ is Lipschitz and any measure with support in the set $\K$ is $A$-maximizing.
 It is so simple to get examples where uniqueness fails that generic uniqueness cannot be seen as sufficient to consider the problem as solved.

 The problem of selection is so fascinating, that, in our opinion, is interesting in itself.

 %%%%%%%%%%%%%%%%%%%%%%%%%%%%%
 \subsection{First selection: entropy criterion}
 In this section we consider $A:\Om\to\R$ a Lipschitz continuous potential. Note that in that specific case, results also hold for H\"older continuous potentials. We denote by $\CM_{max}$ the set of maximizing measures.

 \begin{definition}
\label{def-matherset}
The Mather set of $A$ is the union of the support of all the $A$-maximizing measures.
\end{definition}

 \begin{theorem}
\label{theo-groundstatemaxentrop}
Any ground  state has maximal entropy among the set of maximizing measures. In other words, any accumulation point $\mu_{\8}$ for $\mu_{\be}$ satisfies
$$h_{\mu_{\8}}=\max\left\{ h_{\nu}, \nu\in\CM_{max}\right\}.$$
\end{theorem}
\begin{proof}
First, note that $\CM_{max}$ is closed, thus compact in $\CM_{\s}$. The entropy is upper-semi-continuous, then, there exists measures in $\CM_{max}$ with maximal entropy.

Let $\mu_{\8}$ be such a measure, and set $h_{max}=h_{\mu_{\8}}$.
Then, we get
\begin{equation}\label{equ1-selecentropymax}
h_{\max}+\be.m(A)=h_{\mu_{\8}}+\be.\int A\,d\mu_{\8}\le \CP(\be).
\end{equation}

We remind that $\be\mapsto\CP(\be)$ admits an asymptote as $\be\to+\8$, which means that there exists some $h$ such that
$$\CP(\be)=h+\be.m(A)+o(\be),$$
with $\disp\lim_{\be\to+\8}o(\be)=0$.
Then, Inequality \eqref{equ1-selecentropymax}   shows that $h_{max}\le h$.
Now consider $\mu_{\8}$ any accumulation point for $\mu_{\be}$.
Theorem \ref{prop-asymptopress} says that $\mu_{\8}$  is $A$-maximizing. On the other hand,
$$h+\be.m(A)+o(\be)=\CP(\be)=h_{\mu_{\be}}+\be.\int A\,d\mu_{\be}\le h_{\mu_{\be}}+\be.m(A)$$ yields
$\disp h_{max}\ge h_{\mu_{\8}}\ge \limsup_{\be\to+\8}h_{\mu_{\be}}\ge h\ge h_{max}$.
\end{proof}

\begin{remark}
\label{rem-entropyresidual}
In Statistical Mechanics, $h_{max}$ is called the residual entropy: it is the entropy of the system at temperature zero, when it reaches its ground state.
$\blacksquare$\end{remark}
 Consequently, if the Mather set admits a unique measure of maximal entropy, $\mu_{\be}$ converges to this measure, as $\be\to+\8$.

 We have seen in Proposition \ref{prop-asymptopress} that the pressure function $\be\mapsto \CP(\be)$ admits an asymptote as $\be\to+\8$. Actually this asymptote is given by
 $$h_{max}+\be.m(A).$$

Theorem \ref{theo-groundstatemaxentrop} justifies the study of the set of maximizing orbits. This is the goal of the next section.

%%%%%%%%%%%%%%%%%%%%%%%%%%%
%%%%%%%%%%%%%%%%%%%%%%%%%%%
\section{The Aubry set and the subcohomological inequality}

%%%%%%%%%%%%%%%%%%%%%%%%%%%%%
\subsection{The Aubry set and the Ma\~n\'e potential}

\begin{theorem}
\label{theo-aubryset}
There exists an invariant set $\CA$, called the {\em Aubry set}, which satisfies the following properties:
\begin{enumerate}
\item $\CA$ contains the Mather set, or, equivalently, any $A$-maximizing measure has its support in $\CA$.
\item Restricted to $\CA$, $A-m(A)$ is equal to a  Lipschitz coboundary.
\end{enumerate}
\end{theorem}
This theorem helps to answer the question: how can one detect maximizing orbits ?
The terminology is borrowed from the Aubry-Mather Theory (see \cite{Mat} \cite{CI} \cite{Fathi} \cite{Mane} \cite{Gomes1}).

\begin{definition}\label{def-manepot}
Given $A$, the  Ma\~n\'e potential is:
$$
S_A(x,y):=$$
$$\lim_{\epsilon \to 0}
\Big[\sup\Big\{\sum_{i=0}^{n-1}\big[A(\sigma^i(z))-m_A\big]\;\Big|\;n\in\mathbb{N},\;\sigma^n(z)=y,\;d(z,x)<\epsilon\Big\} \Big].
$$
\end{definition}

The Ma\~n\'e potential $S_A(x,y)$ describes in some sense the "optimal way" to go from  $x$ to $y$, following the dynamics, and also maximizing the ``cost'' given by $A$. For a fixed $\epsilon$, the supremum may be not attained by a finite piece of orbit.

We point out that for a fixed $x$ the function $y \mapsto S(x,y)$ is Holder. Nevertheless, for a fixed $y$ the function $x \mapsto S(x,y)$ is not  necessarily continuous.

\begin{proof}[Proof of Theorem \ref{theo-aubryset} ]
The proof has several steps.
We set
$\CA:=\{\,x\in \Om\,|\, S_{A}(x,x)=0\,\}$.
Let us check that $\CA$ satisfies the required conditions.

$\bullet$ some useful equations.

\noindent
Let us pick any calibrated subaction $V$. We remind that this means
$$V(\s(x)) \ge A(x)-m(A)+V(x),$$
which can be written on the form
\begin{equation}
\label{equ-calib-gneg}
A(x)=m(A)+V\circ \s(x)-V(x)+g(x),
\end{equation}
where $g$ is a non-positive Lipschitz-continuous function.

Let $z$ be in $\Om$. Equality \eqref{equ-calib-gneg} yields
\begin{equation}
\label{equ-cobor1}
S_{n}(A-m(A))(z)=S_{n}(g)(z)+V\circ \s^{n}(z)-V(z).
\end{equation}
Now, pick $y$ and consider $z$ such that $\s^{n}(z)=y$. This yields
$$  S_{n}(A-m(A))(z)\!=\!S_{n}(g)(z)+V(y)-V(z)\! \leq$$
\begin{equation}
\label{equ3-cobor}
\hskip -.5cm g(z)+V(y)-V(z)\!\leq V(y)-V(z).
\end{equation}
Therefore, for every $x$ and $y$, the continuity of $V$ shows
\begin{equation}
\label{equ-cobor2-mane}
S_{A}(x,y)\le g(x)+V(y)-V(x)\le V(y)-V(x).
\end{equation}

$\bullet$ The Aubry set contains the Mather set.

\noindent
We show that if $\mu$ is a $A$-maximizing ergodic measure, then, $\supp \mu\subset \CA$. This will imply, in particular,  that $\CA$ is not empty. Consider $x$ a generic point for $\mu$ (and, also in the support of $\mu$). As
$V$ is Lipschitz,  then, Equality \eqref{equ-cobor1} used for $z:=x$ yields $\disp\lim_{\ninf}\frac1nS_{n}(A-m(A))(x)=0$.
Consequently, $g$ is a non-positive function satisfying $\disp\int g\,d\mu=0$. Since $\mu$ is ergodic and $g$ is continuous, we have  that $g_{|\supp\mu}\equiv 0$. This implies, as we have shown before,
$$A(x)=m(A)+V\circ\s(x)-V(x),$$
for every $x\in \supp\mu$.

Now, $x$ returns infinitely many times as closed as wanted to itself (see Proposition \ref{prop-return}). This yields that for a given $\eps=\frac1{2^{N}}$, there are infinitely many $n_{i}$ such that
$$x=x_{0}x_{1}\ldots x_{n_{i}-1}x_{0}x_{1}\ldots x_{N-1}\ldots.$$
Equivalently, this means that the word $x_{0}\ldots x_{N-1}$ appears infinitely many times into $x=x_{0}x_{1}\ldots$.

This yields that for such $n_{i}$, the word
$$z=x_{0}x_{1}\ldots x_{n_{i}-1}x,$$
coincides with $x$ for at least $n_{i}+N$ digits. Lipschitz regularity for $A$ and $g$ yields
$$\left|S_{n_{i}}(A)(z)-S_{n_{i}}(A)(x)\right|\le C.\frac1{2^{N}}, \text{ and }\left|S_{n_{i}}(g)(z)-S_{n_{i}}(g)(x)\right|\le C.\frac1{2^{N}}.$$
Remind that $g\circ \s^{k}(x)=0$ for every $k$ because $x$ belongs to $\supp\mu$. Remember that  $V$ is  Lipschitz continuous.
Therefore, we get
$$\left|S_{n_{i}}(A-m(A))(z)\right|\le C\frac1{2^{N}}, \s^{n_{i}}(z)=x\text{ and }d(z,x)\le \frac1{2^{n_{i}+N}}.$$
This yields that $S_{A}(x,x)\ge 0$ and Inequality \eqref{equ-cobor2-mane} yields $S_{A}(x,x)\le 0$. Therefore, $x$ belongs to $\CA$.

$\bullet$  We prove that $A-m(A)$ restricted to $\CA$  is a coboundary.

\noindent
Let $x$ be in $\CA$. Note that Inequality \eqref{equ-cobor2-mane} yields
$$0\le S_{A}(x,x)\le g(x)\le 0,$$
which shows that $g$ is equal to 0 on $\CA$. Therefore, $A-m(A)$ is a coboundary on $\CA$.

The last point to check is that $\CA$ is $\s$-invariant.
Let $x$ be in $\CA$. Take a fixed $\eps_{0}>0$ and suppose that  $z$ satisfies $\s^{n}(z)=x$ and $d(x,z)<\eps$.
Then,
\begin{eqnarray*}
S_{n}(A)(z)&= & A(z)+A\circ \s(z)+\ldots +A\circ \s^{n-1}(z)\\
&=& A(\s(z))+\ldots +A\circ \s^{n-2}(\s(z))+A(z)\\
&=& A(\s(z))+\ldots +A\circ \s^{n-2}(\s(z))+A(x)+A(z)-A(x)\\
&=& S_{n}(A)(\s(z))+ A(z)-A(x).
\end{eqnarray*}
Note that $d(\s(z),\s(x))=2d(z,x)<2\eps$ and $|A(z)-A(x)|\le C.d(z,x)\le 2C\eps$. Taking the supremum over all the possible $n$, for fixed $\eps$, and then letting $\eps\to0$ we get
$$S_{A}(x,x)=S_{A}(\s(x),\s(x)).$$
\end{proof}

\begin{proposition}
\label{prop-aubry-compact}
The Aubry set is compact.
\end{proposition}
\begin{proof}
The definition of $S_{A}$ yields
$$S_{A}(x,x)=\lim_{\eps\to0}\sup_{n}\{S_{n}(g)(z),\ \s^{n}(z)=x\ d(x,z)<\eps\}.$$
The function $g$ is non-positive, thus $S_{n}(g)(z)$ is always non-positive.
We claim that
 $S_{A}(x,x)=0$ holds, if and only if, for every $n$, there exists $z$, such that $\s^{n}(z)=x$ and $S_{n}(g)(z)=0$.
Now, Lipschitz continuity of $g$ shows that this condition is closed: if $x$ does not satisfies this condition, there exists $n$ such that for every $z$, satisfying $\s^{n}(z)=x$,  we get $S_{n}(g)(z)<0$. This is obviously true for every $x'$ sufficiently close to $x$.
Therefore, the set of $x$ such that $S(x,x)=0$ is a closed set.

We now prove the claim. If the later condition is not satisfied, then pick $n_{0}$ such that for every $n_{0}$-preimage $z$ of $x$, $S_{n_{0}}(g)(z)<0$. The set of $n_{0}$-preimages is finite, thus, there exists $\eps$ such that for every $z$
 $$\s^{n_{0}}(z)=x\ \Longrightarrow S_{n_{0}}(g)(z)<-\eps.$$
Now, for $n>n_{0}$ and $z$ such that $\s^{n}(z)=x$,
$$S_{n}(g)(z)=S_{n-n_{0}}(g)(z)+ S_{n_{0}}(g)(\s^{n-n_{0}}(z))\le S_{n_{0}}(g)(\s^{n-n_{0}}(z))<-\eps.$$
This shows $S_{A}(x,x)\le -\eps<0$.

Conversely, assume that $S_{A}(x,x)=0$. Pick $n_{0}$ and then consider a sequence of $n$-preimages $z_{n}$ of $x$ converging to $x$ such that
$$S_{n}(g)(z_{n})\to_{\ninf}0.$$
Consider the new sequence of points $\xi_{n}:=\s^{n-n_{0}}(z_{n})$. Each $\xi_{n}$ is a $n_{0}$-preimage of $x$, then,  there exists $z$, such that, $\s^{n_{0}}(z)=x$, and for infinitely many $n$, we have $\xi_{n}=z$.
Therefore, considering only these $n$'s we get
$$0\ge S_{n_{0}}(g)(z)\ge S_{n}(g)(z_{n}),$$
and the right  hand side term goes to $0$, if $n\to+\8$. This yields $S_{n_{0}}(g)(z)=0$, and this holds for every $n$.

\end{proof}

%%%%%%%%%%%%%%%%%%%%%%%%%%%%%%%%
\subsection{The Aubry set for locally constant potentials}\label{subsec-aubrylocconst}
Here we consider a potential $A$ depending on 2 coordinates, that is: $A(x_{0}x_{1}x_{2}\ldots)=A(x_{0},x_{1})$. We set $A(i,j)$ for $A(x)$ with $x=ij\ldots$.

Note that in that case $H_{\be}$ is constant in each 1-cylinder set $[i]$. This also holds for $\disp\frac1\be\log H_{\be}$, and, then, for any accumulation point $V$.
Indeed, recall that $L H_{\be} = \la_\be H_{\be}$.
We can get the function $H_{\be}$ by the expression
$$
    H_{\be}(x)  = \lim_{n \to +\infty}{\frac{1}{n} \sum_{k=0}^{n-1} \frac{L^k 1(x)}{\la_{\be}^k} }.
$$
Note that
$$
   L 1(x_0 x_1 \ldots ) = \sum_{i \in \{1, \ldots, d\}} e^{\be A(i, x_0)} 1(x_0 x_1 \ldots) =
    \sum_{i \in \{1, \ldots, d\}} e^{\be A(i, x_0)}.
$$
Iterating $k$ times the action of the operator $L$, we get
$$
  L^k 1(x_0 x_1 \ldots) = \sum_{i_k, i_{k-1}, \ldots,  i_1 \in \{1, \ldots, d\}}
   e^{\be(A(i_k, i_{k-1})+ \cdots+ A(i_1, x_0))}
$$
Then, the function $H_{\be}(x_0 x_1 \ldots)$ depends only on $x_0$ and so it is
constant in each cylinder of size one $[i]$, as claimed.

If $x$ is in $\Om$ and $z$ is such that $\s^{n}(z)=x$ and, moreover, $d(x,z)<1$, then, $V(x)=V(z)$. Hence,
$$S_{n}(A-m(A))(z)=S_{n}(g)(z)\le 0.$$

If $y$ is in $\CA$, $g_{|\CA}\equiv0$, then $S_{n}(A-m(A))(y)=V\circ\s^{n}(y)-V(y)$. If $\s^{n}(y)$ and $y$ are in the same 1-cylinder, then $S_{n}(A-m(A))(y)=0$.
Furthermore, if $z$ is the periodic orbit given by the concatenation of $y_{0}y_{1}\ldots y_{n-1}$, then, we get $S_{n}(A-m(A))(z)=S_{n}(A-m(A))(y)$, because all the transitions $z_{i}\to z_{i+1}$ are the same as for $y$ (for $i\le n-1$).

This shows that $m(A)$ is reached by periodic orbits.

\begin{definition}
\label{def-simpleloop}
A periodic orbit obtained as the concatenation of $z_{0}\ldots z_{n-1}$ is said to be {\em simple} if all the  digits $z_{i}$ are different.
\end{definition}

\Eg
$123123123\ldots $ is a simple periodic orbit of length 3. By the other hand $121412141214\ldots$ is not a simple periodic orbit.

One simple periodic orbit of length $n$ furnishes $n$ {\em bricks}, which are the words producing the $n$ points of the orbit :

\Eg
The bricks of $123123123\ldots$   are 123, 231 and 312.

Then, the Aubry set $\CA$ is constructed as follows :

\begin{enumerate}
\item List all the simple periodic orbits. This is a finite set.
\item Pick the ones such that their Birkhoff means are maximal.  This maximal value is $m(A)$. Such a simple periodic orbit is also called maximizing.
\item Consider the associated bricks for all these simple maximizing periodic orbits.
\item The set $\CA$ is the subshift of finite type constructed from these bricks:
\begin{enumerate}
\item Two bricks can be combined if they have a common digit. On one of the simple loops one glues the other simple loop. The bricks $x_{0}x_{1}\ldots x_{n}$ and $x_{n}y_{1}\ldots y_{k}$ produce the periodic orbit $x_{0}x_{1}\ldots x_{n}y_{1}y_{2}\ldots y_{k}x_{n}x_{0}x_{1}\ldots x_{n}y_{1}y_{2}\ldots y_{k}x_{n}\ldots$
\Eg It is easy to see that 123  and 345  produce the new orbit $123453123453\ldots$.
\item  $\CA$ is the closure of the set of all the periodic orbits obtained by this process.
\end{enumerate}
\end{enumerate}

We can also define $\CA$ from its transition matrix.
If $i$ does not appear in any of the bricks, we set $T_{ij}=0$, for every $j$.
If $i$ appears in a brick, set $T_{ij}=1$, if $ij$ appears in a brick (for $j\neq i$). If $i$ is also a brick (that means that the fixed point $iiiiii\ldots$ is a maximizing orbit) set also $T_{ii}=1$. Set $T_{ij}=0$ otherwise.
Then, $\CA=\S_{T}$.

\Eg
If the bricks are (up to permutations) $abc$,  $cde$, $fgh$, $gi$ and $fj$, the transition matrix  restricted to these letters is
$$T=\left(\begin{array}{cccccccccc}0 & 1 & 0 & 0 & 0 & 0 & 0 & 0 & 0 & 0 \\0 & 0 & 1 & 0 & 0 & 0 & 0 & 0 & 0 & 0 \\1 & 0 & 0 & 1 & 0 & 0 & 0 & 0 & 0 & 0 \\0 & 0 & 0 & 0 & 1 & 0 & 0 & 0 & 0 & 0 \\0 & 0 & 1 & 0 & 0 & 0 & 0 & 0 & 0 & 0 \\0 & 0 & 0 & 0 & 0 & 0 & 1 & 0 & 0 & 1 \\0 & 0 & 0 & 0 & 0 & 0 & 0 & 1 & 1 & 0 \\0 & 0 & 0 & 0 & 0 & 1 & 0 & 0 & 0 & 0 \\0 & 0 & 0 & 0 & 0 & 0 & 1 & 0 & 0 & 0 \\0 & 0 & 0 & 0 & 0 & 1 & 0 & 0 & 0 & 0\end{array}\right).$$

\begin{figure}[htbp]
\begin{center}
\includegraphics[scale=0.5]{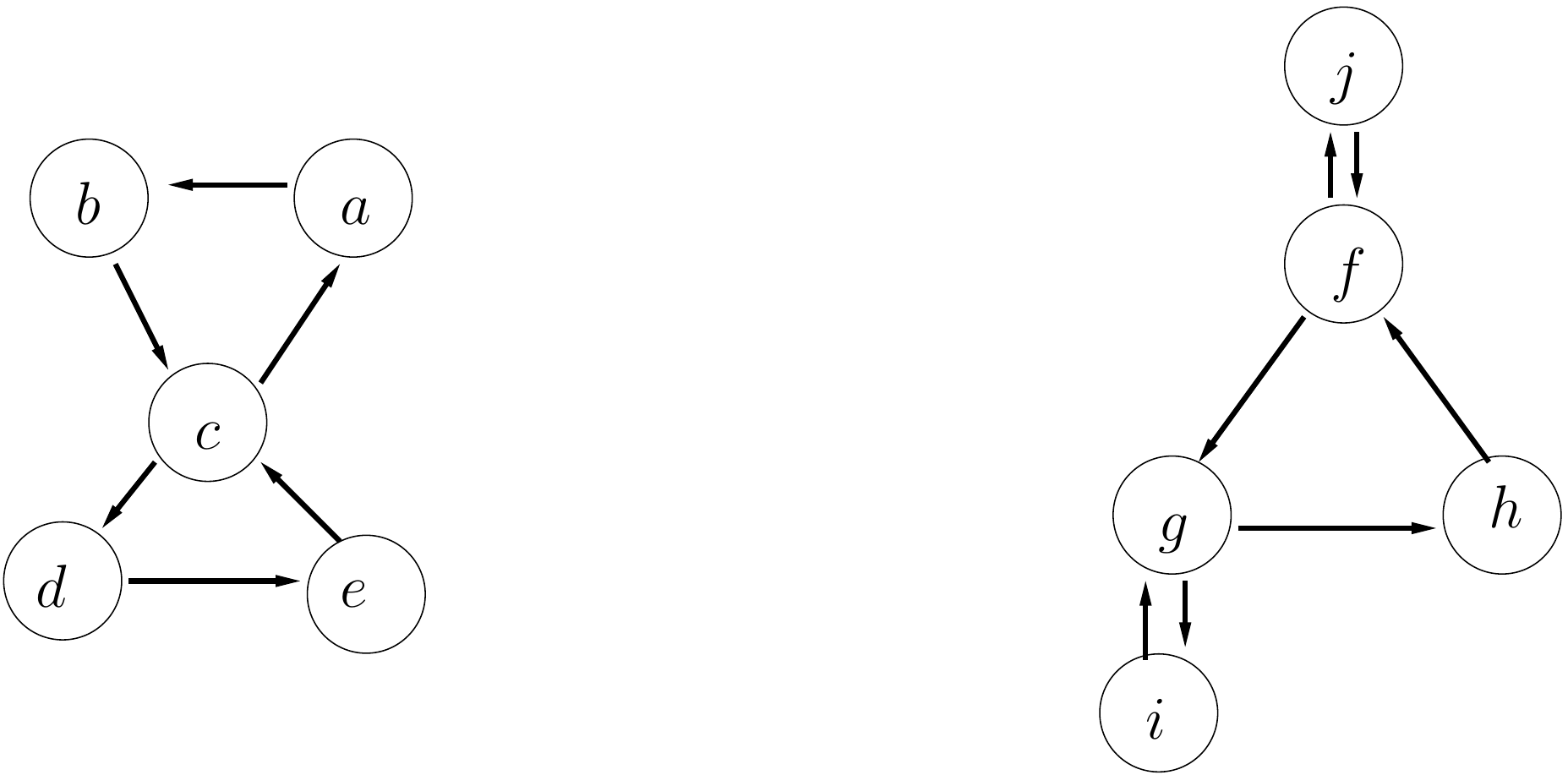}
\caption{The graph for $T$}
\label{fig-graphT}
\end{center}
\end{figure}

The set $\CA$ is a subshift of finite type. It can thus be decomposed in irreducible components, say $\CA_{1}\ldots ,\CA_{r}$. Each component admits a unique measure of maximal entropy, say $\mu_{1},\ldots, \mu_{r}$. Let $h_{i}$ be the associated entropies. We assume that the order has been chosen in such way that
$$h_{1}\ge h_{2}\ge \ldots \ge h_{r}.$$
In that case, the topological entropy for $\CA$ is $h_{1}$. More precisely, assume that $j_{0}$ is such that
$$h_{1}=h_{2}=\ldots = h_{j_{0}}>h_{j_{0}+1}\ge h_{j_{0}+2}\ldots.$$
Then, $\CA$ admits exactly $j_{0}$ ergodic measures of maximal entropy $h_{1}$. Any ground state is a convex combination of these $j_{0}$ ergodic measures.

\bigskip

In that special case, it is proved that there is only one ground state:
\begin{theorem}[see \cite{Bremont,Lep1,CGU}]
\label{th-cv-potloc}
If $A$ depends only on two coordinates, then $\mu_{\be}$ converges, as $\be\to+\8$.
\end{theorem}

\Eg
In the previous example, $\CA$ has two irreducible components. The first one has entropy $\frac13\log 2\sim 0.23$. The second one has entropy $\sim 0.398$. In that case the ground state is the unique measure of maximal entropy, which has support in the second irreducible component.

%%%%%%%%%%%%%%%%%%%%%%%%%%%%%%%%%%%%%%
\subsection{Some consequences of Theorem \ref{th-cv-potloc}}
A dual viewpoint for the selection problem is the following: for every $\be$ the probability $\mu_{\be}$ has full support. One can say that a measure, in our case the probability $\mu_{\be}$, can be represented by  its set of generic points (see Remark \ref{rem-genericset}). This set of points is dense in $\Omega$.

For each $\beta>0$, this set remains dense in $\Omega$, but in the limit, when $\be\to+\8$, this set  accumulates  on the Aubry set $\CA$. More precisely, it is going to accumulate on the irreducible components which have positive weight for some ground state.

The selection problem is thus to determine what are the components of $\CA$ where this set of generic points for $\mu_\beta$ accumulates.

\medskip
It may happen that an irreducible component of $\CA$ has maximal entropy but has no weight at temperature zero. In \cite{Lep1}, the author introduced the notion of isolation rate between the irreducible components and showed that  only the most isolated component have weight at temperature zero.

\newpage

%%%%%%%%%%%%%%%%%%%%%%%%
%%%%%%%%%%%%%%%%%%%%
%%%%%%%%%%%%%%%%%%%%%%%
\chapter{The Peierl's barrier}\label{chap-peierl}

%%%%%%%%%%%%%%%%%%%%%
%%%%%%%%%%%%%%%%%%%%%%%%%
\section{Irreducible components of the Aubry set}

%%%%%%%%%%%%%%%%
\subsection{Definition of the Peierl's barrier}
We have seen above that if $A$ depends on two coordinates the Aubry set $\CA$ is a  subshift of finite type. It thus has well-defined irreducible components, each one being the support of a unique  measure of (relative) maximal entropy. In the case of a more  general potential $A$, there are no reasons why $\CA$ should be a subshift of finite type. Actually it can be any invariant subset as it was shown above : pick any compact set $\CA$ and consider $A:=-d(.,\CA)$.

Given that, it is not obvious how define the irreducible components of $\CA$ and how to determine the measures of maximal entropy.

\begin{definition}
\label{def-peierl}
The Peierl's barrier  between $x$ and $y$ is defined by
$$h(x,y):=\lim_{\eps\to0}\limsup_{\ninf}\left\{S_{n}(A-m(A))(z),\ \s^{n}(z)=y\ d(x,z)<\eps\right\}.$$
\end{definition}

We remind that we got $A=m(A)+V\circ \s-V+g$, where $V$ is a calibrated subaction (obtained via a converging subsequence for $\disp\frac1\be\log H_{\be}$) and $g$ is a non-positive Lipschitz function.
Replacing this expression for $A$ into the definition of the Peierl's barrier we get
\begin{eqnarray*}
h(x,y)\hskip-0.1cm=\hskip-0.1cm\lim_{\eps\to0}\limsup_{\ninf}\left\{S_{n}(g)(z)+V(y)-V(z),\ \s^{n}(z)=y\ d(x,z)<\eps\right\}\\
=\lim_{\eps\to0}\limsup_{\ninf}\left\{S_{n}(g)(z),\ \s^{n}(z)=y\ d(x,z)<\eps\right\}+V(y)-V(x).
\end{eqnarray*}

This shows that to compute $h(x,y)$, we morally have to find a sequence of pre-images for $y$ which converges as fast as possible to $x$.

\noindent We shall see later that it is of prime importance the equality
$h(x,y)+h(y,x)=h(x,x)$.

\begin{theorem}
\label{theo-peierlsuba}
For any $x$, the Peierl's barrier $y\mapsto h(x,y)$ is a Lipschitz calibrated subaction. Moreover, $h(x,x)=0$, if and only if, $x$ belongs to $\CA$.
\end{theorem}
\begin{proof}
Pick $x$ and $y$. Consider $z$ such that $\s^{n}(z)=y$, $d(x,z)<\eps$. Note that in $\Om$, $z$ is just the concatenation $z_{0}\ldots z_{n-1}y$. For $y'$ close to $y$ (namely $y_{0}=y'_{0}$),  we consider $z':=z_{0}\ldots z_{n-1}y'$ ; the Lipschitz regularity for $g$ yields
$$|S_{n}(g)(z)-S_{n}(g)(z')|\le C. d(y,y'),$$
for some constant $C$. If we consider a sequence of $z$ realizing the $\limsup$ and then take $\eps\to0$, we get
$$h(x,y)\le h(x,y')+C. d(y,y').$$
The same argument shows $h(x,y')\le h(x,y)+C.d(y,y')$, and, then $y\mapsto h(x,y)$ is Lipschitz continuous.

Let us show that this function is a calibrated subaction. In this way, consider $y$, $n$ and $\eps$, such that,
$\s^{n}(z)=y$ and $d(x,z)<\eps$. Then, $d(x,z)<\eps$ and $\s^{n+1}(z)=\s(y)$. Moreover,
\begin{equation}
\label{equ-suba-peierl}
S_{n+1}(A-m(A))(z)=S_{n}(A-m(A))(z)+A(y)-m(A).
\end{equation}
Consider a sequence of $z$ realizing the $\limsup$ for $h(x,y)$, now taking the limit along the subsequence, and, considering $\eps\to0$, we get
$$h(x,y)+A(y)-m(A)\le h(x,\s(y)).$$
This shows that $y\mapsto h(x,y)$ is a subaction. It remains to show that for a fixed value for $y'=\s(y)$, the equality is achieved by one of the preimages of $y'$.
This  follows from taking the $z$'s and the $n$'s in Equality \eqref{equ-suba-peierl} which realize the $\limsup$ for the left hand-side of the equality. Then, we get
$$h(x,y')\le h(x,y)+A(y)-m(A),$$
with $\s(y)=y'$. As the reverse inequality holds, the global equality follows.

\medskip
Both definitions of the Ma\~n\'e potential and the Peierl's barrier are very similar, except that in the first case we consider the supremum, and in the other case we consider the $\limsup$. This immediately shows that
$$h(x,y)\le S_{A}(x,y),$$
and then, $h(x,x)=0$ yields $S_{A}(x,x)=0$ (due to Inequality \ref{equ-cobor2-mane}), thus $x$ belongs to $\CA$.

Let us prove the converse. Take $x$ in $\CA$. Consider $\rho>0$ small and $\eps_{0}$, such that, for every $\eps<\eps_{0}$,
 \begin{equation}
\label{equ-def-manesioux}
\sup_{n}\left\{S_{n}(A-m(A))(z),\ \s^{n}(z)=x\ d(z,x)<\eps\right\}\ge -\rho.
\end{equation}
In the following, we assume for simplicity that $x$ is not periodic, but the proof can be easily extended to that case.

Inequality\eqref{equ-def-manesioux} holds for every $\eps$. We will construct a subsequence  $(n_{k})$ by induction. We pick any $\eps$ and consider $n_{0}$ realizing the supremum up to $-\rho$:
$$S_{n_{0}}(A-m(A))(z_{0})> -2\rho\text{ with }\s^{n_{0}}(z_{0})=x\text{ and }d(x,z_{0})<\eps.$$
Now we use Inequality\eqref{equ-def-manesioux} but with\footnote{Here we use that this distance is positive, that is that $x$ is not periodic.}

$$\disp\eps_{1}<\min\left\{d(x,z), \s^{n}(z)=y,\ n\le n_{0}\right\}.$$

We get $n_{1}>n_{0}$  and $z_{1}$, such that,
$$S_{n_{1}}(A-m(A))(z_{1})> -2\rho\text{ with }\s^{n_{1}}(z_{1})=x\text{ and }d(x,z_{1})<\eps_{1}.$$
We then proceed by induction with
$$\eps_{k+1}<\disp\min\left\{d(x,z), \s^{n}(z)=y\ n\le n_{k}\right\}.$$

Then, we have
\begin{eqnarray*}
-2\rho&< & S_{n_{k}}(A-m(A))(z_{k}),\text{ with }\s^{n_{k}}(z_{k})=y\text{ and }d(x,z_{k})<\eps\\
&\text{then}&\\
-2\rho&\le & \limsup_{\ninf}\left\{S_{n}(A-m(A))(z),\ \s^{n}(z)=x\ d(x,z)<\eps\right\}.
\end{eqnarray*}
This holds for every $0<\eps<\eps_{0},$ and then $h(x,x)\ge -2\rho$. Now, we take $\rho\to 0$, and we get $h(x,x)\ge 0$. The reverse equality is always true, so $h(x,x)=0$.
\end{proof}

%%%%%%%%%%%%%%%%%%%%%%%%%%%%%%%%%%%%
\subsection{The irreducible components of the Aubry set}

Now, we show that the Peierl's barrier allows to define ``irreducible'' components of the Aubry set.
\begin{lemma}
\label{lem-inetrianpeierl}
For any $x$, $y$ and $z$
\begin{equation}
\label{equ-peierltriang}
h(x,y)\ge h(x,z)+h(z,y).
\end{equation}
\end{lemma}
\begin{proof}
Let $\eps>0$ be fixed.
Consider a preimage $y'$ of $y$ close to $z$ and a preimage $z'$ of $z$ close to $x$. For small $\eps$, $z'$ satisfies $z'=z'_{0}\ldots z'_{n-1}z$ and then $y'':=z'_{0}\ldots z'_{n-1}y'$ is also a preimage of $y$. The cocycle relation yields, if $\s^{m}(y')=y$,
$$S_{n+m}(A-m(A))(y'')=S_{n}(A-m(A))(y'')+S_{m}(A-m(A))(y').$$
The Lipschitz regularity shows that $S_{n}(A-m(A))(y'')$ differs from $S_{n}(A-m(A))(z')$ by the term $\pm C.\eps$.
If we assume that the $n$'s and the $m's$ are chosen to realize the respective $\limsup$, the term on the left-hand side is lower than the $\limsup$. Then, the Lemma is proved.
\end{proof}

\begin{lemma}
\label{lem-peierlinvariant}
For any $x$ in $\CA$, $h(x,\s(x))+h(\s(x),x)=0$.
\end{lemma}
\begin{proof}
Lemma \ref{lem-inetrianpeierl} and Theorem \ref{theo-peierlsuba} show that
$$h(x,\s(x))+h(\s(x),x)\le h(x,x)=0.$$
It remains to prove it is non-negative.
The Peierl's barrier is a subaction thus
$$h(x,\s(x))\ge A-m(A)+h(x,x)=A-m(A).$$
Pick $\eps>0$ and  consider $y$, a preimage of $x$ $\eps$-close to $\s(x)$. Suppose $y=y_{0}\ldots y_{n-1}x$. Then, $x':=x_{0}y_{0}\ldots y_{n-1}x$  is a preimage of $x$ $\eps$-close\footnote{Actually $\frac\eps2$-close to $x$.} to $x$. In particular $A(x')=A(x)\pm C.\eps$. We emphasize that it is equivalent to have $x'\to x$ or $y\to\s(x)$.

We assume that these quantities are chosen in such way that $S_{n+1}(A-m(A))(x')$ converges to the $\limsup$, if $n$ goes to $+\8$.
Now,
$$S_{n+1}(A-m(A))(x')=A(x')-m(A)+S_{n}(A-m(A))(y).$$
Taking $n\to+\8$ and then $\eps\to0$, the right-hand side term is lower than  $h(\s(x),x)+A(x)-m(A)$ and the left-hand side term goes to $h(x,x)=0$. This yields
$$0\le h(\s(x),x)+A(x)-m(A)\le h(\s(x),x)+h(x,\s(x)).$$
\end{proof}

\begin{lemma}
\label{lem-peierltriangequal}
Let $x$ $y$ and $z$ be in $\CA$. If $h(x,y)+h(y,x)=0$ and $h(y,z)+h(z,y)=0$, then $h(x,z)+h(z,x)=0$.
\end{lemma}
\begin{proof}
Inequality \eqref{equ-peierltriang} shows
$$h(x,z)+h(z,x)\ge h(x,y)+h(y,z)+h(z,y)+h(y,x)=0+0=0.$$
It also yields $h(x,z)+h(z,x)\le h(x,x)=0$.
\end{proof}
Lemma \ref{lem-peierltriangequal} proves that $h(x,y)+h(y,x)=0$ is a transitive relation. Since it is obviously symmetric and reflexive, then it is an equivalence relation on $\CA$.

\begin{definition}
\label{def-compo-aubry}
The  equivalence classes  for the relation
$$h(x,y)+h(y,x)=0,$$
are called irreducible components of $\CA$.
\end{definition}

Note that $x$ and $\s(x)$ belong to the same class, which shows that the classes are invariant.
The continuity for the Peierl's barrier has been proved with respect to the second variable for a fixed first variable. It is thus not clear that an irreducible component is closed.

%%%%%%%%%%%%%%%%%
\subsection{The locally constant case}
 If the potential is locally constant, we have seen in Subsection \ref{subsec-aubrylocconst} that the Aubry set $\CA$ is a subshift of finite type, for which the notion of irreducible component has already been defined (see \ref{def-irreducshift}). We have to check that the two notions coincide.

 We have seen that the irreducible components of a subshift of finite type are exactly the transitive components. We shall use this description to show that irreducible components of the Aubry (in the sense of the Peierl's barrier) set are the irreducible components (with respect to subshifts).

\begin{lemma}
\label{lem-manepeierllocal}
Assume $A$ depends only on two coordinates. Let $x$ and $y$ be in $\Omega$. Assume that $d(z,x)\le \frac14$, $\s^{n}(z)=y$, and, moreover, that there exists $1\le k<n,$ such that, $d(\s^{k}(z),x)\le \frac14$. Then,
$$S_{n}(A-m_{A})(z)\le S_{n-k}(A-m_{A})(\s^{k}(z)).$$
\end{lemma}
 \begin{proof}
We consider a calibrated subaction $V$. We have seen that it  depends only on  one coordinate. We remind that for every $\xi$,
\begin{equation}
\label{equ-pratiquecobord}
A(\xi)=m_{A}+V\circ \s(\xi)-V(\xi)+g(\xi)
\end{equation}
holds, where $g$ is a non-positive function (depending only on 2 coordinates).
Now we have
$$S_{k}(A-m_{A})(z)=S_{k}(g)(z)+V(\s^{k}(z))-V(z)=S_{k}(g)(z)\le 0.$$
\end{proof}
From Lemma \ref{lem-manepeierllocal} we claim that for every $x$ and $y$,
\begin{equation}\label{equa-manepeierllocal}
h(x,y)=S_{A}(x,y)=\max\left\{S_{n}(A-m_{A})(z)\, \s^{n}(z)=y\,\, d(z,x)\le\frac14\right\}.
\end{equation}

\begin{lemma}
\label{lem-contihprem}
Consider some fixed $y$ in $\Om$. The map $x\mapsto h(x,y)$ is continuous
\end{lemma}
\begin{proof}
Consider a sequence $(x_{n})$ converging to $x$. Assume that all these $x_{n}$ and $x$ coincide for at least 2 digits.
Then,
$$d(z,x_{n})\le \frac14\iff d(z,x)\le\frac14.$$
For $y$ and $n$, consider any $z$ realizing the maximum into the definition of $S_{A}(x_{n},y)$. It also realizes the maximum for $S_{A}(x,y)$ and more generally for every $S_{A}(x_{k},y)$.
\end{proof}

Lemma \ref{lem-contihprem} shows that any irreducible components of the Aubry in the sense  of definition \ref{def-compo-aubry} is closed. It is also invariant. We thus have to show it is transitive.

Let us consider some components and pick two open sets $U$ and $V$ (for the components). We still assume that Equality \eqref{equ-pratiquecobord} holds.  Consider $x\in U\cap \CA$ and $y\in V\cap\CA$. By definition
$$h(x,y)+h(y,x)=0.$$
By Equality \eqref{equa-manepeierllocal}, $h(x,y)$ is realized by some $S_{n}(A-m_{A})(z)$ and $h(y,x)$ is realized by some $S_{m}(A-m_{A})(z')$. Moreover, we can also assume that $z$ belongs to $U$ and $z'$ belongs to $V$.

Indeed, if it is not the case, we can always follow preimages of $x$ in the component $\CA$ which are exactly on the set $g^{-1}(\{0\})$.

From the two pieces of orbits which are $z, \s(z),\ldots, \s^{n}(z)$ and $z', \s(z'),\ldots, \s^{m}(z')$ we can construct a periodic orbit. Denote by $\xi$ the point  of this periodic orbit in $U$.
Then, we claim that
$$h(y,\xi)=h(y,x)=S_{m}(A-m_{A})(z')$$
and,
$$h(\xi,y)=h(x,y)=S_{n}(A-m_{A})(z).$$

This shows that $\xi$ belongs to $\CA$ and to the same component than $y$. Therefore $\s^{-m}(U)\cap V\neq\emptyset$ and the component is transitive.

 %%%%%%%%%%%%%%%%%%%%%%%%%
 %%%%%%%%%%%%%%%%%%%%%%%%%
 \section{On the road to solve the subcohomological inequality}

 %%%%%%%%%%%%%%
 \subsection{Peierl's barrier and calibrated subactions}
 From properties of the Peierl's barrier we get the following result: a calibrated subaction is entirely determined by its values on the Aubry set.

 \begin{theorem}[see  \cite{GL1}{Th. 10}]
\label{theo-subcpeierl}
Any calibrated subaction $u$
satisfies for any $y$
\begin{equation} \label{equ-Peisubacdeter}
  u(y) = \sup_{\mathbf x \in \CA} [h(\mathbf x, y)+  u(\mathbf x) ],
\end{equation}
\end{theorem}
 \begin{proof}
We show that any calibrated subaction $u$ is entirely determined by its values on the Aubry set $\CA$.

Let us thus consider some calibrated subaction $u$. Let $y$ be in $\Om$. Let $y_{-1}$ be any preimage of $y$ such that
$$u(y)=A(y_{-1})-m(A)+u(y_{-1}).$$
More generally we consider a sequence $y_{-n}$ such that $\s(y_{-n})=y_{-n+1}$ and
$$u(y_{-n+1})=A(y_{-n})-m(A)+u(y_{-n}).$$

We claim that any accumulation point for $(y_{-n})$ belongs to $\CA$. Indeed, let us consider some converging sequence $y_{-n_{k}}$, converging to $x$.

For $k'>k$ we get $\s^{n_{k'}-n_{k}}(y_{-n_{k'}})=y_{-n_{k}}$. Therefore, we have
$$S_{n_{k'}-n_{k}}(A-m(A))(y_{-n_{k'}})=u(y_{-n_{k}})-u(y_{-n_{k'}}),$$
and, $y_{n_{k}}\to x$ yields
$$S_{A}(x,x)\ge 0. $$

As $S_{A}(x,x)\le 0$ always holds (see Inequality \eqref{equ-cobor2-mane}) we get  that the limit point $x$ belongs to $\CA$.

Now we have that $\disp S_{n_{k}}(A-m(A))(y_{n_{k}})=u(y)-u(y_{n_{k}})$, which yields
$$h(x,y)\ge u(y)-u(x).$$
In particular $u(y)\le \sup_{x'\in \CA}\{h(x',y)+u(x')\}$ holds.

Actually, the reasoning we have just done allows to get a stronger result. Consider $z$ in $\Om$, then we can write
$$A(z)=m(A)+u\circ \s(z)-u(z)+g(z),$$
where $g$ is a non-positive Lipschitz function.
Now, consider $z_{n}$ such that $\s^{n}(z)=y$. Therefore,
$$S_{n}(A-m(A))(z_{n})=u(y)-u(z)+S_{n}(g)(z)\le u(y)-u(z).$$
This shows that
$$h(x',y)\le u(y)-u(x')$$
always holds (consider any $x'$ and take a subsequence of $z_{n}$ converging to it).
Then, Equality \eqref{equ-Peisubacdeter} holds.
\end{proof}

Moreover, the irreducible component get a special importance :

\begin{theorem}
\label{theo-subac-peierlirreduc}
If $x$ and $z$ are in the same irreducible component of $\CA$, then for any $y$,
$$h(x,y)+u(x)=h(z,y)+u(z).$$
\end{theorem}
 \begin{proof}
We remind two inequalities  and an important equality:
\begin{eqnarray*}
h(x,y)&\ge& h(x,z)+h(z,y),\\
u(x)&\ge& h(z,x)+u(z),\\
0&=&h(x,z)+h(z,x).
\end{eqnarray*}
Then,
\begin{eqnarray*}
u(x)+h(x,y)&\ge & u(x)+h(x,z)+h(z,y)\\
&\ge & u(x)-h(z,x)+h(z,y)+u(z)-u(z)\\
&\ge & u(x)-u(z)-h(z,x)+h(z,y)+u(z)\\
&\ge& h(z,y)+u(z).
\end{eqnarray*}
Exchanging the roles of $x$ and $z$ we get that the reverse inequality also holds.
\end{proof}

 %%%%%%%%%%%%%%%%%%%%%%%%%%%%%%%%%
 \subsection{Selection of calibrated subactions}
 We remind that for a given $\be>0$, the equilibrium state $\mu_{\be}$ is also a Gibbs measure obtained by the product of the {\em eigenfunction} $H_{\be}$ and the {\em eigenprobability}\footnote{which is also the conformal measure.} $\nu_{\be}$.
 We also remind that any accumulation point for the family $\disp\frac1\be\log H_{\be}$ is a calibrated subaction.

Uniqueness of the maximizing measure gives a partial answer to these questions:

\begin{theorem}
\label{theo-uniquecalibra}
Assume that there is a unique $A$-maximizing measure, then all the calibrated subactions are equal up to an additive constant.
\end{theorem}
 \begin{proof}
In that case $\CA$ is uniquely ergodic and has thus a single irreducible component. If $x_{0}$ is any point of $\CA$, Theorems \ref{theo-subac-peierlirreduc} and \ref{theo-subcpeierl} show that any calibrated subaction is entirely determined by its value on $x_{0}$.
\end{proof}

 We point out that even in that simple case, the convergence for $\disp\frac1\be\log H_{\be}$ is not clear.

 In that direction we mention one of the results in \cite{Lep3}. For simplicity we state it using the setting of \cite{BLL}.

\begin{theorem}[see \cite{Lep3}]
\label{theo-selecsubacreno}
Assume that $\Om=\{1,2,3\}^{\N}$ and $A$ satisfies
 $$A(x):=
\begin{cases}
-d(x,1^{\8})\text{ if }x=1\ldots,\\
-3d(x,2^{\8})\text{ if }x=2\ldots,\\
-\al<0 \text{ if }x=3\ldots.\\
\end{cases}
$$
Then, $\nu_{\be}\to\delta_{1^{\8}}$, as $\be\to+\8$, and, moreover, $\disp\frac1\be\log H_{\be}$ converges.
\end{theorem}
 This theorem shows that flatness is a criterion for selection: the Aubry set in that case is reduced to $\{1^{\8}\}\cup \{2^{\8}\}$ and the two unique ergodic maximizing measures are the Dirac measures $\delta_{1^{\8}}$ and $\delta_{2^{\8}}$. The potential is ``more flat'' in $1^{\8}$ than in $2^{\8}$. Therefore, this Theorem says that the locus where the potential is flatter gets  all the mass in the limit of the eigenmeasure, when $\beta \to \infty.$ In that case this is sufficient to determine all the calibrated subactions.

More generally if the Aubry set $\CA$ is not a subshift of finite type, the problem concerning selection is that
\begin{enumerate}
\item there is no satisfactory  theory for the analysis of the measure of maximal entropy for general subshifts.
\item we do not know about the existence or uniqueness of conformal measures (one of the key points in Theorem \ref{theo-selecsubacreno} in the procedure of selection of calibrated subactions).
\end{enumerate}
We shall also see that the problem of selection of subaction is related to the multiplicity of an eigenvector in the Max-Plus formalism.

%%%%%%%%%%%%%%%%%%%%%%
%%%%%%%%%%%%%%%%%%%%
\newpage

%%%%%%%%%%%%%%%%%%%%%
%%%%%%%%%%%%%%%%%%%
\chapter{The Involution Kernel \label{inv}}

\section{Introduction and main definitions}

In this section we will present the involution kernel which is a concept that is sometimes useful for understanding problems
of different areas: large deviations (\cite{BLT}), issues of differentiability of the main eigenfunction and piecewise differentiability of the subaction \cite{LOS} \cite{LMMS}, optimal transport (see \cite{CLO} \cite{LO},\cite{LM2},\cite{LOT}), etc... It is also related to the Gromov distance  on hyperbolic spaces (see \cite{LT1}).

The main issue here is that is sometimes helpful for understanding a problem to look for the dual problem.

We recall that the   Bernoulli  space is the set
$\{1,2,...,d\}^\mathbb{N}=\Omega$. A general element $x$ in the   Bernoulli  space
$\Omega$ will be denoted in this section by $ x=(x_0,x_1,..,x_n,..)$. The function $\sigma$ denotes the shift acting on $\Omega$.

We will consider another copy of $\Omega$ which will be denoted by $\tilde{\Omega}$. Points in this set are denoted by $w= (w_0,w_1,w_2,..,w_n,..).$

We denote
$\{1,2,...,d\}^\mathbb{Z}=\hat{\Omega}=\tilde{\Omega}\times \Omega   = \{1,2,...,d\}^\mathbb{N}\times  \{1,2,...,d\}^\mathbb{N}$. Points in this set
are denoted by
$$(w\,|\, x)=(w,x)= (...w_n...w_3,w_2,w_1,w_0\,|\,x_0,x_1,..,x_n,..).$$

For a fixed $i\in\{1,2,..,d\}$, the function $\psi_i:\Omega \to \Omega$
indicates the $i$-th inverse branch of $\sigma$, $i\in \{1,2,..,d\}$. This means $\psi_{i} (x)= (i, x_0,x_1,..,x_n,..)$. We can also use the notation
$$\psi_{w} (x)= (w_0, x_0,x_1,..,x_n,..),$$
$w\in \Omega$,  in the case, $ w=(w_0,w_1,..,w_n,..)$. In other words $\psi_{w} =\psi_{w_0}.$
\bigskip

We also denote by
$\tilde{\sigma}$ the shift on $\tilde{\Omega}$. Finally,  $\mathbb{ T}^{-1}$ is the backward
shift on $\hat{\Omega}$ given by
$$\mathbb{ T}^{-1}(w,x) = ( \tilde{\sigma} (w),
\psi_{w_0} (x)).$$

\bigskip
We will present some general results for a H\"older function $A$ which does not necessarily depends on two coordinates. Later, in this section, we will assume that $A$ depends on two coordinates.

\bigskip

It is known \cite{BLT} \cite{LMMS} that given $A:\Omega \to \mathbb{R}$ H\"older, in the variable $A(x)$, there exists
a dual function $A^* :\tilde{\Omega} \to \mathbb{R}$, in the variable $A^*(w)$, and $W:\tilde{\Omega}\times  \Omega\to \mathbb{R}$, such that

$$A^*(w)= A\circ \mathbb{ T}^{-1}(w,x)+ W \circ \mathbb{ T}^{-1}(w,x) -
W(w,x).$$

The functions $A^* :\tilde{\Omega} \to \mathbb{R} $ and $W:\tilde{\Omega}\times  \Omega\to \mathbb{R}$ are both H\"older. We say that $W$ is {\bf the involution kernel} and $A^*$ is  {\bf the dual potential} for $A$.

The $A^*$and $W$ are not uniquely defined.

The expression for $A^* $ can be also written as

$$A^*(w)= A (w_0,x_0,x_1,..)+$$
$$ W (...w_2,w_1\,|\,w_0,x_0,x_1,..) -
W(..,w_1,w_0\,|\,x_0,x_1,..).$$

We say that $A$ is symmetric if $A=A^*$.

Suppose $A$ is fixed, and $W$ and $A^*$ are also fixed.
For a given real parameter $\beta$ we have

$$\beta\, A^*(w)= \beta A\circ \mathbb{ T}^{-1}(w,x)+ \beta W \circ \mathbb{ T}^{-1}(w,x) - \beta
W(w,x).$$

It follows that for any real $\beta$ we have that $\beta W$ is the involution kernel and $\beta A^*$ is the dual potential for $\beta A$

\bigskip

Given $A$, we denote
$$\Delta(x,x',w)=\sum_{n\geq1}A\circ\tau_{w,n}(x)
-A\circ\tau_{w,n}(x').$$

The involution kernel $W$ can be explicitly computed in the following way: for any $(w,x)$ we define $W$ by
$W(w,x)=\Delta_A(x,x',w)$,  where we choose a point $x'$ for good \cite{BLT}.

It is known the following relation: for any $x,x',w\in \Omega$, we
have that $W(w,x)-W(w,x')=\Delta(x,x',w)$.

We denote $\phi_{\beta A}=\phi_\beta: \Omega \to \mathbb{R}$ the main eigenfunction of the Ruelle operator for $\beta A$ and $\lambda(\beta A)=\lambda(\beta )$ the main eigenvalue.
In the same way we  denote  $\phi_{\beta A^*}= \phi_{\beta}^*$, where, $\phi_{\beta A^*}: \tilde{\Omega} \to \mathbb{R}$, the main eigenfunction of the Ruelle operator for $\beta A^*$ and $\lambda(\beta A)^*$, which is the  corresponding main eigenvalue.

One can show (see \cite{LMMS}) that $\lambda(\beta A)= \lambda(\beta A)^*$.

$\nu_{\beta }=\nu_{\beta A}$ denotes the normalized  eigenprobability for the Ruelle operator for  $\beta A$, and,
$\nu_{\beta}^*=\nu_{\beta A^*}$ denotes the normalized  eigenprobability for the Ruelle operator for $ \beta A^*$.

Finally, the probabilities $\mu_{\beta}=\mu_{\beta A}=\phi_{\beta A}\nu_{\beta A}$ and  $\mu_{\beta}^*=\mu_{\beta A^*}=\phi_A^* \nu_A^*$, are, respectively, the equilibrium states for $\beta A$ and $\beta A^*$.

\bigskip

In the case $A$ is symmetric we have $\nu_{\beta}=\nu^*_{\beta}$, $\mu_{\beta}=\mu^*_{\beta}$, etc...
\bigskip

We denote the $\beta$ normalizing constant by
$$c(\beta) =  \log \int \int e^{\beta\, W(w,x)} \, d \nu_{\beta A^*}(w) \, d \nu_{\beta A} (x)  ,$$
and
$$ \gamma = \lim_{\beta \to \infty} \frac{1}{\beta} c(\beta).$$

Therefore,
$$ \gamma= \lim_{\beta \to \infty}\, \frac{1}{\beta} \,c(\beta) = sup_{w,x}\,\{ W(w,x)-V(x) -V^*(w) -I(x)- I^* (w)\}.$$

The probability $e^{\beta\, W(w,x)} \, d \nu_{\beta A^*}(w) \, d \nu_{\beta A} (x)$ (after normalization) is invariant for the shift
$\hat{\sigma}$ acting on $\{1,2,..,d\}^\mathbb{Z}.$

It is known (see \cite{BLT}) that
\begin{equation}\label{back1}
 \phi_{\beta A}(x) = \int e^{\beta W(w,x)-c(\beta)} \,d
\nu_{\beta A^*}(w),
\end{equation}
is the normalized eigenfunction for the Ruelle operator of $\beta A$.

Moreover,
\begin{equation}\label{back2}
 \phi_{\beta A^*}(w) = \int e^{\beta W(w,x)-c(\beta)} \,d
\nu_{\beta A}(x),
\end{equation}

Note from above that we also get
\begin{equation}\label{back3}
 \phi_{\beta A}(x) = \int e^{\beta W(w,x)-c(\beta)}\,  (\phi_{\beta A^*}(w))^{-1} \,d
\mu_{\beta A^*}(w),
\end{equation}

and this is a relation between   $\phi_{\beta A^*}$ and  $\phi_{\beta A}.$

The main point above is that, via an integral Kernel $e^{W}$, one can get information of the eigenfunction via the eigenmeasure of the dual problem.

We denote respectively by $V$ and $V^*$ any calibrated subaction for $A$ and $A^*$.

If the maximizing probability for $A$ is unique, the maximizing probability for $A^*$ is also unique (see \cite{LMMS}). We can also get a large deviation principle for $\mu_{\beta A^*}.$

We denote by $I^*$ the deviation function for $A^*$ (see \cite{BLT}). We presented on the previous section \ref{ldev} some basic properties of the deviation function $I$ for $A$. The only difference now is that we consider the same for $A^*$.

\bigskip

Suppose $V$ is the limit of  a subsequence $\frac{1}{\beta_n}\, \log \phi_{\beta_n}$, where $\phi_{\beta_n}$ is an eigenfunction
of the Ruelle operator for $\beta_n A$. Suppose $V^*$ is obtained in an analogous way for $A^*$ (using a common subsequence).  Then, there exists $\gamma$ such that
\begin{equation}\label{exp}
\gamma+ V(x ) =
\sup_{ w \in \Omega}\, [\,W(w, x) - V^* (w)- I^*(w)\,].
\end{equation}

This follows from Varadhan's Integral Lemma (see \cite{DZ}) and the fact that $I$ is the Large deviation function for the family $\mu_{\beta\, A}$ (see \cite{LOT} \cite{LOS} \cite{CLT}).

Then, $V$ is the $W(w, x)  - I^*(w)- \gamma$ transform of $V^*$ (see \cite{Vil1} for definitons)

There is a dual expression
\begin{equation}\label{exp}
\gamma + V^* (w ) =
\sup_{ x \in \Omega}\, [\,W(w, x) - V (x)- I(x)\,].
\end{equation}

\medskip

Note that the above equation is not for a general pair of calibrated subactions $V$ and $V^*$, it is an equation for selected subactions.
\medskip

We point out that when $A=A^*$ any calibrated subaction $V$ obtained via the limit of eigenfunction of the Ruelle operator when $\beta$ to $\infty$ will satisfy the equation

\begin{equation}\label{exp}
\gamma + V (w ) =
\sup_{ x \in \Omega}\, [\,W(w, x) - V (x)- I(x)\,]
\end{equation}
for some constant $\gamma$.

\bigskip

In the general case ($A\neq A^*$)  we have that the limit calibrated subaction $V$ satisfies the equation

$$
  V(x) = \max_{y \in \Omega}{ \left\{ W(y, x) - \max_{z \in \tilde{\Omega} }{ \{ W(y, z) - V(z)  \} }     \right\} }.
$$

Note that there is no $\gamma$ in the above equation.
\medskip

\medskip

Under the hypothesis of twist condition for $A$ (to be defined bellow), and in the case the maximizing probability has support in a periodic orbit, the above equation can help to
obtain explicit expression for subactions (see \cite{LOS} \cite{CLO}).

We consider now on $\Omega=\{1, ...,d\}^\mathbb{N}$ the lexicographic order. This order is obtained from the order $1\leq 2\leq 3\leq ...\leq d$.

Following \cite{LOT} we define:

\begin{definition} \label{rom} We say a  continuous  $G: \hat{\Omega}=\tilde{\Omega}\times \Omega
 \to \mathbb{R}$ satisfies the  {\bf twist} condition on
$ \hat{\Omega}$, if  for any $(a,b)\in
\hat{\Omega}=\tilde{\Omega}\times \Omega $ and $(a',b')\in\tilde{\Omega}\times \Omega
$, with $a'> a$, $b'>b$, we have
\begin{equation} \label{ab}
G(a,b) + G(a',b')  <  G(a,b') + G(a',b).
\end{equation}
\end{definition}

\begin{definition} We say a continuous $A: \Omega \to \mathbb{R}$   satisfies the twist condition, if some  of its involution kernels $W$
satisfies the twist  condition.

\end{definition}

\bigskip

The twist condition plays in Ergodic Optimization, in some sense,  the same role as the "convexity hypothesis of the Lagrangian" which is of fundamental importance in Aubry-Mather theory (see \cite{Mat}, \cite{CI}, \cite{Fathi} or \cite{GT1}).
\bigskip

An alternative definition for the twist condition would be to claim that for any $a'> a$, $b'>b$ we have
\begin{equation} \label{ba}
G(a,b) + G(a',b')  >  G(a,b') + G(a',b).
\end{equation}

This will make no difference in the nature of the results we will get.
\bigskip

Several examples of potentials satisfying the twist condition appear in \cite{LOT}.

Linear combination with positive coefficients of  potentials $A$  which  satisfy the  twist condition also satisfies the  twist condition.

Given $x$, we denote $w_x\in  \tilde{\Omega}$ any point such that

$$
\gamma+ V(x ) =
 \, [\,W(w_x, x) - V^* (w_x)- I^*(w_x)\,].
$$

The proof of the above appears in \cite{LOS}.

\vskip 0.5cm
\begin{proposition} If $A$ is twist, then
$x \to w_x$  is
monotonous non-increasing, where  $w_x$ was chosen to be optimal.
\end{proposition}

For a proof see \cite{LOS}.

The results above were used to show how to get  subactions in an explicit form for potentials $A$ which satisfy some restrictions (see \cite{LOS} \cite{LOT}.

\section{Examples}

We point out that, for some kind of questions,  it is  more easy to manipulate expressions with the involution kernel than with the Peierl's barrier.

For example, note that if
$A_1$ and $A_2$ are two potentials and

$$A_1^*(w)= A_1\circ \mathbb{ T}^{-1}(w,x)+ W_1 \circ \mathbb{ T}^{-1}(w,x) -
W_1(w,x),$$

$$A_2^*(w)= A_2\circ \mathbb{ T}^{-1}(w,x)+ W_2 \circ \mathbb{ T}^{-1}(w,x) -
W_2(w,x),$$
then the involution kernel for $(A_1 + A_2)$ is $(W_1+W_2)$ and its dual potential is $(A_1^* + A_2^*).$

We use the notation $W(w,x)=W(...w_2, w_1,w_0\,|\,x_0,x_1,x_2,..)$ and
$$\hat{\chi}_{[z_k,.. z_2,z_1,z_0\,| \,x_0,x_1,...x_n]}(w,x),$$
denotes the indicator function in $\hat{\Omega}$ of the cylinder set
$$[z_k,.. z_2,z_1,z_0\,|\, x_0,x_1,...x_n].$$

\bigskip
All the above was for case  of general potentials. Now, we analyze briefly the case where the potential depends on two coordinates.

\bigskip

The real function on $\tilde{\Omega}$ denoted by $\tilde{\chi}_{[a_1, a_0]}$ is such that is $1$, if $w=(...w_2, w_1,w_0)=( ...w_2, a_0,a_1)$, and, zero otherwise.

\Eg

Given $A=\chi_{[a_0, a_1]}$, the indicator function of the cylinder $[a_0,a_1]$,  defined on $\Omega$, then, for any
$$(w,x)= (...w_2, w_1,w_0\,|\,x_0,x_1,x_2,...),$$ we have

$$ \tilde{\chi}_{[a_1, a_0]}(...w_2, w_1,w_0)=  A(w_0,x_0,...) +  $$
$$\hat{\chi}_{[a_0\,| \,a_1]} (   ...w_2, w_1\,|\,w_0,x_0,x_1,...) -\, \hat{\chi}_{[a_0\,| \,a_1]}(...w_2, w_1,w_0\,|\,x_0,x_1..).$$

Therefore the dual potential of $A=\chi_{[a_0, a_1]}$ is $A^*=\tilde{\chi}_{[a_1, a_0]}$ and its involution kernel is $W=\hat{\chi}_{[a_0\,| \,a_1]}$.

\medskip

We point out that in the case the potential depends on infinite coordinates there is no simple expression for the dual potential (and the involution kernel)

\bigskip

Any potential $A$ which depends on two coordinates on the full shift $\Omega=\{1, ...,d\}^\mathbb{N}$ can be written in the form
$$A = \sum_{i,j\in \{1, ...,d\}} \alpha_{i,j} \, \chi_{[i,j]}$$
and therefore, has involution kernel
$$W = \sum_{i,j\in \{1, ...,d-\}} \alpha_{i,j} \, \hat{\chi}_{[i,j]},$$
and dual potential
$$A^* = \sum_{i,j\in \{1, ...,d\}} \alpha_{i,j} \, \chi_{[j,i]}.$$

In this way if we consider $A$ as a matrix, {\bf the dual $A^*$ is the transpose of $A$.}

\bigskip

\Eg If $A:\{1,2\}^\mathbb{N}\to \mathbb{R}$, which depends on two coordinates, satisfies: $A(1,1)=2,$ $A(2,2)=5,\,$ $ A(1,2)=7$, and $A(2,1)=6$, then, the involution kernel $W$ is
$$W=\, 2\, \hat{\chi}_{[1\,|\,1]}+ 5 \,\hat{\chi}_{[2\,|\,2]} + 7\, \hat{\chi}_{[1\,|\,2]}+ 6\, \hat{\chi}_{[2,1]}.$$

This involution kernel $W$ satisfies the twist condition.

\bigskip

The dual $A^*$ of $A$ (via the involution kernel $W^*$) increase the scope of the concept of transpose for the case $A$ depends on an infinite number of coordinates.
\bigskip

\Eg

When $\Omega=\{0,1\}^\mathbb{N}$ we
we denote $M_k$ the cylinder $[\underbrace{1111...111}_{k}\,0]$, $k \in \mathbb{N}$, $k\geq 1$.

A well known class of potentials (see \cite{Hof} \cite{LZ} \cite{LF} \cite{BLL1}) which are not H\"older is the following: suppose $\gamma$ is a real positive parameter, then take $A=A_\gamma$ such that

\begin{equation*}
\begin{split}&A_\gamma(\underline{x})=-\gamma \log \bigg(\frac{k+1}{k}\bigg)
\quad \text{ if } \underline{x} \in M_k, \,\, k\neq 0\\
&A_\gamma(\underline{x})=-\log(\zeta(\gamma))\quad \text{ if } \underline{x} \in [0]\\
&A_\gamma(11...))=0,
\end{split}
\end{equation*}
where $\zeta(\gamma)=1^{-\gamma}+2^{-\gamma}+\cdots.$

We denote bellow  by $W(k,j)$ the value of the involution kernel in the set
of points $ (w|x) =(... w_1 w_0\,|\,x_0 x_1 x_2...)$ such that $w\in M_k$ and $x\in M_j$, $k,j \in \mathbb{N}$.

The involution  kernel $W$ for these potentials  satisfy
$$e^{W(k,j)}=\frac{(j+1+k)^{-\gamma}}{(j+1)^{-\gamma}(k+1)^{-\gamma}}\zeta(\gamma),$$
$k,j \in \mathbb{N}$.
Note that $$W(k,j)=W(j,k).$$

The dual potential $A^*$ satisfies $A^*=A$. This involution kernel satisfies the twist condition.

These class of potentials $A_\gamma$ present the phenomena of phase transition on the one dimensional lattice $\Omega=\{0,1\}^\mathbb{N}$.

\medskip

The above example came from discussions with A. Araujo.
\medskip

\Eg Suppose the maximizing probability is unique for $A$.

Suppose $V$ is the limit of $\frac{1}{\beta}\, \log \phi_{\beta_n}$, where $\phi_{\beta}$ is the eigenfunction
of the Ruelle operator for $\beta A$. Suppose $V^*$ is the subaction for $A^*$.

If $\hat{\mu}_{max}$ is the
natural extension of the maximizing probability $\mu_{\infty\, A}$, then
for all $(p^*,p)$ in the support of $\hat{\mu}_{max}$ we have the following expression taken from Proposition 10 in \cite{BLT}

$$ V(p)\, + \,V^* (p^*)\,= \,  W(p^*,p ) -  \gamma \,.$$

The above expression appears in a natural way in problems in transport theory (see \cite{Vil1}). It is called the complementary slackness condition. In this case the involution kernel is the natural dynamically defined cost to be considered in the problem (see \cite{LOT} \cite{LOS} \cite{CLO} \cite{LO1} \cite{GL3}). The subactions $V$ and $V^*$ define the dual Kantorovich pair associated to the problem.
\medskip

We denote by $I^*$ the deviation function for the family of Gibbs states $\mu_{\beta A^*}$, $\beta \to \infty$,  according to what was described in section \ref{ldev}. Remember that $I^*$ is zero over the support of the limit measure $\mu_{\infty\, A}^*$.

If  $(p^*,p)$ in the support of $\hat{\mu}_{max}$ (then,
$p\in [0,1]$  is in the support of $\mu_{\infty\, A}$  and  $p^*\in \Omega$  is in the support of $\mu_{\infty\, A}^*$), then
$$V(p) =  \sup_{ w \in \Omega}\, \, W(w,p) - V^*
( w)- I^*(w) -  \gamma\,=$$
$$ \, W( p^*,p)  - V^*
( p^*)- I^*(p^*)\,-  \gamma=          \, W( p^*,p) - V^*
( p^*)\,-  \gamma.
$$

Remember that $V^*$ is the $W- I^*-\gamma$ transform of $V$.

We point out that in problems on Ergodic Transport the following question is important.
It is known that if $V$ is calibrated  for $A$, then, it is true that, for any $z$ in the support of the maximizing probability $\mu_{\infty \,A}$, we have  $V(z)+A(z)-m(A)=V(T(z).$ Generically (on the Holder class) on $A$, one can show that  this equality it true just on the support of the maximizing probability.
An important issue
is to know  if  this last property is also true for the maximizing probability of $A^*$ (see \cite{CLO} for a generic result).

%%%%%%%%%%%%%%%%%%%%%%%%%%%%%%%%%%%%%%%%%%
%%%%%%%%%%%%%%%%%%%%%%%%%%%%%%%%%%%%%%%%%%
%%%%%%%%%%%%%%%%%%%%%%%%%%%%%%%%%%%%%%%%%%
%%%%%%%%%%%%%%%%%%%%%%%%%%%%%%%%%%%%%%%%%%

\chapter{Max Plus algebra}\label{chap-maxplus}

 In this chapter we will not consider dynamical systems and
 invariant measures
 for a while in order to introduce  the so called Max Plus algebra. The results of this section will be used in the next one.

\section{Motivation}

  The Max-plus algebra is essentially the algebra of the real
  numbers with two binary operations,
  $a \oplus b = \max(a, b)$ and $a \otimes b = a+b$;
  there are many distinct motivations to introduce this mathematical
  object, some of them being, for example, from problems in operational research.
  Let us imagine, for example, a factory where some worker $i$ needs
  to wait for some of his colleagues $j$ and $k$ to finish their tasks,
  which  takes the times, respectively,  $T_{ij}$ and $T_{ik}$, in order to
  do his job (of the worker $i$) which takes the time $a_i$. Hence, in the factory,  the total time of this work
  is given by $a_i + \max{\{T_{ij}, T_{ik}\}}$; this value can be rewritten
  in the max-plus notation as $a_i \otimes (T_{ij} \oplus T_{ik})$. This shows
  that in a larger system involving a larger number of workers and
  distinct steps, the  max-plus algebra is a convenient and elegant way to formulate the
  problem of distribution of tasks in order. One can use the techniques of this theory in order to make the whole process
  more efficient.

  From a pure mathematical point of view we can also see this algebra, for example,
  as the structure of the exponential growth of real functions.
  Given a function $h \colon \RR \to \RR$ we define the exponential
  growth of $h$ as the limit (if, of course, the limit exists)
  $$
    e(h) = \lim_{x \to +\infty}{ \frac{1}{x} \log{h(x)}  }.
  $$
  To fix ideas,
 consider the simple case
 $$
     f(x) = e^{ax}   \qquad \text{and} \qquad g(x) = e^{bx}.
 $$
 From the definition it is easy to see that $e(f)= a$ and $e(g) = b$.
Now, what is the  exponential growth of $f g$ or $f+g$?
For $f g = e^{ax} e^{bx} = e^{a+b}$ we have that the function
$fg$ grows with a rate of $a+b$, say $e(f)+e(g)$. Hence
$e(fg) = e(f) + e(g) = e(f) \otimes e(g)$. For $f+g$,
we have that $f+g=e^{ax}+e^{bx} = e^{\max{ \{a, b \}} x}(1+o(1))$. Hence,
$e(f+g) = \max{ \{e(f), e(g) \}} = e(f) \oplus e(g)$. For this reason,
it is not a surprise that this technique appears also in the setting
of zero temperature limits, where we are exactly talking about comparing
certain exponential growth rates.

General references on Max-Plus Algebra are \cite{ABG}, \cite{AGG}, \cite{CD}, \cite{CD2}, \cite{Baca}, \cite{Chung}, \cite{Far} and \cite{Conc}.

\section{Notation and basic properties}

 In this text we use
$$
    \bar{\RR} = \RR \cup {-\infty}
$$
with the convention that $x + (-\infty) = \infty$ for any $x \in \bar{\RR}$.

We endow this set with two operations:
$$
    a \oplus b = \max(a, b)
$$
$$
   a \otimes b = a+b .
$$
With this notation, the convention above is rewritten as $a \otimes -\infty = -\infty$ and
we also have $a \oplus -\infty = a$, showing that $-\infty$ is the neutral element for
the binary operation $\oplus$.

For the operation $\otimes$ we have that
$ a \otimes 0 = a+0 = a$ for any $a \in \bar{\RR}$, showing that
$0$ is its neutral element.

The so called max-plus algebra is then the semi-ring over $\bar{\RR}$ defined by
the operations $\oplus$ and $\otimes$.

Some of the main properties of this algebra are listed below:
\begin{lemma}
given $a, b$ and $c$ in $\bar{\RR}$ ,we have
\begin{itemize}
\item[1-] Associativity: $a \oplus ( b \oplus c ) = (a \oplus  b ) \oplus c $
            and $a \otimes ( b \otimes c ) = (a \otimes  b ) \otimes c $
\item[2-] Commutativity:  $a \oplus  b  = b \oplus  a $ and  $a \otimes  b  = b \otimes  a $
\item[3-] Distributivity: $a \otimes (b \oplus c) = (a \otimes b) \oplus (a \otimes c)$
\item[4-] Additive identity: $a \oplus (-\infty) = (-\infty) \oplus a = a$
\item[5-] Multiplicative identity: $a \otimes 0 = 0 \otimes a = a$
\item[6-] Multiplicative inverse: if $a \neq -\infty$ then there exists a unique
                   $b$ such that $a \otimes b = 0$
\item[7-] Absorbing element: $a \otimes -\infty = -\infty \otimes a = -\infty$
\item[8-] Idempotency of addition: $a \oplus a = a$.
\end{itemize}
\end{lemma}
\begin{proof}
  We will show just some of the above properties, leaving the others
  to the reader.

Distributivity, for example, follows from
$$
  a \otimes (b \oplus c ) = a + \max(b, c) = \max(a+b, a+c) = (a \otimes b) \oplus (a \otimes c) .
$$

For the multiplicative inverse, just notice that for any $a$ we can take $b=-a$
and so
$$
 a \otimes b =  a + (-a) = 0 .
$$
\end{proof}

\section{Linear algebra}

\subsection{Vectors}

A $d$-dimensional vector $v$ is an element of $\bar{\RR}^d$, which is  denoted by
$v=(v_1, v_2, \ldots, v_d)$, or, as is usual, represented as a column vector
$$
  v =
\left[
  \begin{array}{c}
    v_1  \\
    v_2  \\
    \vdots \\
    v_d
  \end{array}
 \right].
$$

Given two vectors $u$ and $v$ in $\bar{\RR}^d$ and $\la \in \bar{\RR}^d$,
we can define the sum of two vectors as
$$
  u \oplus v := (u_1 \oplus v_1, u_2 \oplus v_2, \ldots, u_d \oplus v_d),
$$
and, the product by an scalar as
$$
  \la \otimes u := ( \la \otimes u_1, \la \otimes u_2, \ldots, \la \otimes u_d).
$$
%(notice that $\la \otimes u = u \otimes \la$).

\subsection{Matrices}

An $m \times n$  matrix $A$ is defined as in the usual case.
Given two $m \times n$ matrices $A$ and $B$, we define
$A \oplus B$ as the matrix whose entries are
$$
   (A \oplus B)_{ij} := A_{ij} \oplus B_{ij} = \max{\{A_{ij}, B_{ij} \}}.
$$
Given $\la \in \bar{\RR}$, we also define the matrix
$\la \otimes A$ as
$$
  (\la \otimes A)_{ij} = \la \otimes A_{ij} = \la + A_{ij}.
$$

From the basic properties of the operations $\oplus$ and $\otimes$,
it is not hard to see that the matrix operations above satisfies
the following properties:
\begin{lemma}
Given $m \times n$ matrices $A, B$ and $C$, and some
$\la \in \bar{\RR}$ we have:

There exists a matrix $[-\infty]$ such that
$$
   A \oplus [-\infty] = A ,
$$
$$
   A \oplus B = B \oplus A ,
$$
$$
  A \oplus (B \oplus C) = (A \oplus B) \oplus C ,
$$
$$
 \la \otimes A = A \otimes \la ,
$$
$$
  \la \otimes (A \oplus B) = (\la \otimes A) \oplus (\la \otimes B) .
$$
\end{lemma}

Given an $m \times n$ matrix $A$ and an $n \times l$ matrix $B$ we can
define the matrix product $AB$ as
$$
  (AB)_{ij} = \bigoplus_{k}( A_{ik} \otimes B_{kj}) = \max_k\left( A_{ik}+B_{kj} \right) .
$$
\begin{lemma}
Moreover,
$$
 (AB)C = A(BC) ,
$$
$$
 \la \otimes AB = A (\la \otimes B) = AB \otimes \la .
$$
\end{lemma}

If $m = n$ we say that the matrix $A$ is a square matrix of order $n$.
Consider the matrix
$$
   I_n =
   \left[
     \begin{array}{ccccc}
       0       &   -\infty  &  -\infty  &  \ldots  &  -\infty   \\
      -\infty  &      0     &  -\infty  & \ldots   &  -\infty   \\
        \vdots &    \vdots  &   \vdots  &   \vdots &    \vdots  \\
      -\infty  &    \ldots  &   \ldots  &  -\infty &     0
       \end{array}
   \right].
$$
Then, we can show that
$$
     A I_n = I_n A = A,
$$
for any order $n$ matrix $A$.

\subsection{Eigenvectors and eigenvalues}

Now consider a $n \times n$ matrix $A$ whose entries are elements
of $\bar{\RR}$ and a column vector $v$.

We define the product
$A v$ such that
$$
 (Ax)_i = \bigoplus_j (A_{ij} \otimes v_j) =  \max_j{(A_{ij} + v_j) }.
$$
For a given $\la \in \bar{\RR}$ we also define
$$
 \la v = (\la \otimes v_1, \ldots, \la \otimes v_n) = (\la + v_1, \ldots, \la + v_n).
$$

In this setting it is a very natural question to look for
max plus eigenvectors and eigenvalues for $A$, in the sense
that, $A v = \la v$;
this notation can be translated in
terms of our usual operations as
$$
   \max_j{(A_{ij} + v_j) }  = \la + v_i   \;\;\; \text{for any $i = 1, \ldots, n$}.
$$

%Some examples:
%\begin{exem}
\Eg $$
 \left[
  \begin{array}{cc}
   1  &   0  \\
   0  &   1
  \end{array}
 \right]
 \left[
  \begin{array}{cc}
   0   \\
   -1
  \end{array}
 \right] =
 \left[
 \begin{array}{cc}
   \max{(1, -1)}   \\
   \max{(0 ,0)}
  \end{array}
 \right] =
 \left[
 \begin{array}{cc}
   1   \\
   0
  \end{array}
 \right] =
  1
  \left[
   \begin{array}{cc}
   0   \\
   -1
  \end{array}
 \right].
$$
We also have
$$
 \left[
  \begin{array}{cc}
   1  &   0  \\
   0  &   1
  \end{array}
 \right]
 \left[
  \begin{array}{cc}
   -1   \\
    0
  \end{array}
 \right] =
 \left[
 \begin{array}{cc}
   \max{(0, 0)}   \\
   \max{(-1 , 1)}
  \end{array}
 \right] =
 \left[
 \begin{array}{cc}
    0   \\
    1
  \end{array}
 \right] =
  1
  \left[
   \begin{array}{cc}
    -1   \\
     0
  \end{array}
 \right].
$$

In this case we see that $1$ is an eigenvalue associated with two
distinct eigenvectors.
%\end{exem}

%\begin{exem}
\Eg Consider
$$
   \left[
     \begin{array}{cc}
       -\infty   &    a   \\
          b      &  \infty
     \end{array}
   \right],
$$
which has eigenvalue $\la = (a+b)/2$ and eigenvector
$$
   \left[
     \begin{array}{c}
         x     \\
        x+ (b-a)/2
     \end{array}
   \right]
     =
      x \otimes
     \left[
     \begin{array}{c}
         0     \\
         (b-a)/2
     \end{array}
   \right]
$$
(where any choice of $x$ is allowed).
%\end{exem}

The most important result, with respect to our purposes, is that matrices with real entries have a unique eigenvalue.
\begin{theorem}\label{theo-eigenmaxplusmatrix}
  Let $A$ be a $d \times d$ matrix with all entries $a_{ij} \in \RR$; then, there
  exists a real number $\la$ and a vector $v$, such that, $A v = \la v$; moreover,
  the eigenvalue $\la$ is unique.
\end{theorem}
\begin{proof}
 First of all, notice that if $M u = \mu u$, then
 $\la M v= \la \otimes \mu v = (\la + \mu) v$.
 Note that
 $$ \la M =
 \left[
  \begin{array}{cccc}
   \la+m_{11}  &   \la+m_{12}  &  \cdots  & \la+m_{1d} \\
   \la+m_{21}  &   \la+m_{22}  &  \cdots  & \la+m_{2d} \\
     \vdots    &    \vdots     &  \vdots  &  \vdots    \\
   \la+m_{d1}  &    \cdots     &  \cdots  &  \la+m_{dd}
  \end{array}
 \right]
 $$
 Hence, there is no loss of generality if we assume that the entries
 of $A$ are all non-negative. So, we get
$$
    0 \leq a_{ij}  \leq L .
$$
Now, let us define the map $T \colon \RR^n \to \RR^n$ as
$$
  (Tx)_i = \max_j{(A_{ij} + x_j )}  -  \min_k{ \max_j{(A_{kj}+x_j)}  }.
$$
  It is
easy to see that the expression depends continuously on
the vector $x$.  It is also clear from the definition that $(Tx)_i \geq 0$.
On the other hand, we have
$$
   (Tx)_i \leq  \max_j{(L + x_j)} - \min_k{ \max_j{(0+x_j)}  }  = $$
$$\max_j{(L+x_j)} - \max_j{(x_j)} = L  .
$$
In particular, this shows that the region $\{x_j : 0 \leq x_j \leq L \}$ is mapped
inside itself by $T$; since $T$ is continuous, this implies (by means of Brouwer fixed
point theorem) that $T$ has at least one fixed point $v$.
Hence,
$$
 v = T(v) \Rightarrow v_i = (Tv)_i = \max_j{(A_{ij}+ v_j)} - \min_k{ \max_j{(A_{kj} + v_j )} }.
$$
Denoting
$$
  \la = \min_k{ \max_j{(A_{kj} + v_j ) } },
$$
then, the expression above implies
$$
     v = Av - \la  \Rightarrow \la + v = Av  \Rightarrow \la v = Av,
$$
in the max-plus sense, as claimed.

For the uniqueness, let us assume, by contradiction, that we have
two distinct eigenvalues $\la$ and $\mu$. In other words, there
exists vectors $v$ and $u$, such that,
$$
   A v = \la v  \qquad \text{and}  \qquad A u = \mu u.
$$
Without loss of generality we can assume $\la < \mu$.
It is possible to take a large $t$, such that,
$t v  \geq u$ (in the sense that $tv_i \geq u_i$, for each $i \in \{1, \ldots, d\}$).
Then,
$$
     tv \oplus u = tv.
$$
Hence,
$$
  A^n( tv \oplus u ) = A^n(tv) \oplus A^n(u) = A^n (tv)  \Rightarrow
$$
$$
  t A^n(v) \oplus A^n(u) = t A^n(v)  \Rightarrow t \la^n v \oplus \mu^n = t \la^n v,
$$
which is equivalent to say that for any $n$, we have $t \la^n v \geq \mu^n u$,
and this is a contradiction, since $\la < \mu$. Then, we get $\la = \mu$ and the max-plus
eigenvalue is unique.
\end{proof}

If we drop the hypothesis of real entries the situation is quite different. Consider, for
example,
$$
  A =
     \left[
       \begin{array}{cccc}
         1    &   1    &   -\infty   &  -\infty    \\
         1    &   1    &   -\infty   &  -\infty    \\
      -\infty  &  -\infty  &   2    &   2           \\
      -\infty  &  -\infty  &   2    &   2
       \end{array}
     \right] .
$$
Then, it is not very hard to see that
$$
   A  \left[
       \begin{array}{c}
            -\infty    \\
            -\infty    \\
               1       \\
               1
       \end{array}
     \right]
       =
       \left[
       \begin{array}{c}
            -\infty    \\
            -\infty    \\
               3       \\
               3
       \end{array}
     \right]  =
     2 \otimes
      \left[
       \begin{array}{c}
            -\infty    \\
            -\infty    \\
               1       \\
               1
       \end{array}
     \right],
$$
and, that
$$
  A  \left[
       \begin{array}{c}
                1   \\
                1   \\
            -\infty    \\
            -\infty
       \end{array}
     \right]
       =
       \left[
       \begin{array}{c}
                2    \\
                2    \\
            -\infty    \\
            -\infty
       \end{array}
     \right]  =
     1 \otimes
      \left[
       \begin{array}{c}
               1      \\
               1      \\
            -\infty    \\
            -\infty
       \end{array}
     \right].
$$
Hence, $1$ and $2$ are max-plus eigenvalues of $A$ showing that the uniqueness
of the eigenvalue does not hold for matrices with $-\infty$ entries.

\section{Final remarks}

 There exists some variations of the max-plus algebra, one of then
 being the min-plus algebra, where the binary operation $\oplus$
 is replaced by $a \oplus b = \min(a, b)$; this algebra
 is also known as tropical algebra. In this context its
 interesting, for example, to study the behavior of
 polynomials like
$$
   p(x) = a_0 \oplus (a_1 \otimes x) \oplus (a_2 \otimes x \otimes x) = \min{\{  a_0, a_1+ x, a_2 + 2 x  \}} .
$$
Its graph, for instance, is a union of the segments (some of them of
infinite length) and the reader is invited to sketch this picture.
The geometrical investigation of those objects is known as
tropical geometry.

%%%%%%%%%%%%%%%%%%%%%%%%%%%%%
%%%%%%%%%%%%%%%%%%%%%%%%%%%%%%
%%%%%%%%%%%%%%%%%%%%%%%%%%%%%%%%%%%%%
%%%%%%%%%%%%%%%%%%%%%%%%%%%%%%%%%%%%%%%%%%%%%%%%%%%

\chapter{An explicit example with all the computations}\label{chap-explexam}

We consider the following particular case: $\Om:=\{1,2,3\}^{\N}$ and $A$ is a  non-positive potential depending only on two coordinates.
For each pair $(i,j)$ we set $A(i,j)=-\eps_{ij}$. We assume that
$$\eps_{11}=\eps_{22}=0,$$
and for every other pair $\eps_{ij}>0$.

Note that there are only two ergodic $A$-maximizing measures, namely, $\delta_{1^{\8}}$ 	and $\delta_{2^{\8}}$, which are the Dirac measures at $1^{\8}:=111\ldots$ and at $2^{\8}:=222\ldots$.
The Aubry set is exactly the union of these two fixed points and each one is an irreducible component.

We remind that for each $\be$, the unique equilibrium state is given by
$$\mu_{\be}=H_{\be}\nu_{\be},$$
where $H_{\be}$ and $\nu_{\be}$ are the eigenvectors for the transfer operator. Its spectrum is $\disp e^{\CP(\be)}$.

We have seen in Subsection \ref{subsec-moreresulpressure} that $\disp \CP'(\be)=\int A\,d\mu_{\be}$.  This quantity is negative because $A$ is non-positive and negative on a set of positive measure (for $\mu_{\be}$ and for any $\be$). We have seen (see the comments after Remark \ref{rem-entropyresidual}) that the asymptote of the pressure is of the form
$$h_{max}+\be.m(A),$$
where $h_{max}$ is the residual entropy and is the entropy of the Aubry set. In our case, we have $h_{max}=0$ and $m(A)=0$. Then, $\disp\lim_{\be\to+\8}\CP(\be)=0$.
The first role of the max-plus algebra seems to determine how the pressure goes to $0$ as $\be$ goes to $+\8$.

\begin{proposition}
\label{prop-pressure}
There exists a positive sub exponential function $g$  and a positive real number $\rho$ such that $\CP(\be)=g(\be)e^{-\rho.\be}$.
\end{proposition}
\begin{proof}
First we consider any accumulation point $-\rho$ for $\disp\frac1\be\log\CP(\be)$. We shall prove that this $-\rho$ is actually unique, namely, it does not depend on the chosen subsequence.

Considering the subsequence which realizes the expression for $-\rho$,  we can always extract another subsequence such that $\disp\frac1\be\log H_{\be}$ converges. We denote by $V$ this limit, and we have seen above that $V$ is indeed a calibrated subaction.

For simplicity we shall write $\binf$ even if we consider a restricted subsequence.

Moreover, $H_{\be}$ and $V$ only depend on one coordinate, and we shall write $H_{\be}(i)$ or $V(i)$ for $H_{\be}(x)$ and $V(x)$ with $x=i\ldots$.

The eigenfunction $H_{\be}$ is an eigenvector for $\CL_{\be}$ and this yields
\begin{eqnarray*}
e^{\cpb}H_{\be}(1)&=&e^{\be.A(1,1)}H_{\be}(1)+e^{\be. A(2,1)}H_{\be}(2)+e^{\be.A(3,1)}H_{\be}(3)\\
e^{\cpb}H_{\be}(2)&=&e^{\be.A(1,2)}H_{\be}(1)+e^{\be. A(2,2)}H_{\be}(2)+e^{\be.A(3,2)}H_{\be}(3).
\end{eqnarray*}
Replacing with the values for $A$ we get the two following equations:
\begin{subeqnarray}
(e^{\cpb}-1)H_{\be}(1)&=& e^{-\be\eps_{21}}H_{\be}(2)+e^{-\be\eps_{31}}H_{\be}(3),\slabel{equ1-transfer}\\
(e^{\cpb}-1)H_{\be}(2)&=& e^{-\be\eps_{12}}H_{\be}(1)+e^{-\be\eps_{32}}H_{\be}(3).\slabel{equ2-transfer}
\end{subeqnarray}

Now, $\disp\frac\cpb\be\to_{\binf}0$, which yields $\disp \lim_{\binf}\frac{e^{\cpb}-1}{\cpb}=1$ and, finally, $\disp \lim_{\binf}\frac1\be\log(e^{\cpb}-1)=-\rho$.

Taking $\disp \frac1\be\log$ and doing $\binf$ in (\ref{equ1-transfer}) and (\ref{equ2-transfer})  we get
\begin{subeqnarray}
 -\rho+V(1)&=& \max(-\eps_{21}V(2),-\eps_{31}V(3))\slabel{equ1-subaclineaire},\\
  -\rho+V(2)&=& \max(-\eps_{12}V(1),-\eps_{32}V(3))\slabel{equ2-subaclineaire},
\end{subeqnarray}
which can be written under the form
\begin{equation}
\label{equ1-subacvector}
-\rho\otimes \left(\begin{array}{c}V(1) \\V(2)\end{array}\right)= \left(\begin{array}{ccc}-\8 & -\eps_{21} & -\eps_{31} \\-\eps_{12} & -\8 & -\eps_{32}\end{array}\right)\otimes \left(\begin{array}{c}V(1) \\V(2) \\V(3)\end{array}\right).
\end{equation}

Now, we use Theorem \ref{theo-subcpeierl} to get an expression of $V(3)$ in terms of $V(1)$ and $V(2)$.
Indeed, we have
$$V(3)=\max(V(1)+h(\1,3),V(2)+h(\2,3)),$$
where $h$ is the Peierl's barrier and 3 means any point starting with 3. Copying the work done to get Equality \eqref{equa-manepeierllocal}, we claim that
$$h(\1,3)=-\eps_{13}\text{ and }h(\2,3)=-\eps_{23}.$$
This yields,
\begin{equation}
\label{equ2-subacvector}
\left(\begin{array}{c}V(1) \\V(2) \\V(3)\end{array}\right)=\left(\begin{array}{cc}0 & -\8 \\-\8 & 0 \\-\eps_{13} & -\eps_{23}\end{array}\right) \otimes\left(\begin{array}{c}V(1) \\V(2)\end{array}\right).
\end{equation}
Merging \eqref{equ1-subacvector} and \eqref{equ2-subacvector}  we finally get
$$
\hskip-0.3cm-\rho\!\otimes\!\begin{pmatrix}V(1) \\ V(2)\end{pmatrix}\!\!=\!
 \begin{pmatrix}-\eps_{13}-\eps_{31} & -\eps_{21}\oplus(-\eps_{31}-\eps_{23}) \\-\eps_{12}\oplus(-\eps_{32}-\eps_{13}) & -\eps_{23}-\eps_{32}\end{pmatrix}\!\otimes\! \begin{pmatrix}V(1) \\V(2)\end{pmatrix}.
$$
\begin{equation}
\label{equ3-subacvector}
\end{equation}
This last equation shows that $-\rho$ is an eigenvalue for the above  matrix. We have seen (see Theorem \ref{theo-eigenmaxplusmatrix}) that such matrix admits an unique eigenvalue (but not necessarily an unique eigenvector). This shows that $\disp\frac1\be\log\cpb$ admits a unique accumulation point as $\binf$, thus converges.

Setting $g(\be):=\cpb.e^{\rho\be}$ we get the proof of the proposition.
\end{proof}
\begin{remark}
\label{rem-g-function}
We point out that the max-plus algebra allows to determine the value for $-\rho$, but the function $g(\be)$ remains unknown. For the convergence or, not, of $\mu_{\be}$,$ \beta \to \infty$,  it is necessary a better understanding of the behavior of $g(\be)$.

$\blacksquare$\end{remark}
Following the reasoning we did before, we get
\begin{equation}
\label{equ-rho}
-\rho= \max
\begin{cases}
 A(1,3)+A(3,1)=-\eps_{13}-\eps_{31},\\
 \\
 A(3,2)+A(2,3)=-\eps_{32}-\eps_{23},\\
 \\
 \disp \frac{A(2,1)+A(1,2)}2=-\frac{\eps_{12}+\eps_{21}}2,\\
 \\
\disp \frac{A(2,1)+A(1,3)+A(3,2)}2=-\frac{\eps_{21}+\eps_{13}+\eps_{32}}2,\\
\\
\disp  \frac{A(1,2)+A(2,3)+A(3,1)}2=-\frac{\eps_{12}+\eps_{23}+\eps_{31}}2,\\
\\
\disp \frac{A(2,3)+A(3,1)+A(1,3)+A(3,2)}2.
\end{cases}
\end{equation}
We emphasize that the last quantity is actually the mean value of the two first ones, and then $-\rho\ge \frac{A(2,3)+A(3,1)+A(1,3)+A(3,2)}2$ always holds, as soon as, $-\rho\ge A(1,3)+A(3,1)$ and $-\rho\ge A(3,2)+A(2,3)$ hold.

We can now finish the proof of the convergence of $\mu_{\be}$. We recall that any accumulation point must be of the form
$$\al. \delta_{\1}+(1-\al)\de_{\2},$$
with $\al\in[0,1]$. It is thus sufficient to show that $\disp \frac{\mu_{\be}([1])}{\mu_{\be}([2])}$ converges as $\binf$ to prove the convergence of $\mu_{\be}$. However, we emphasize that this is very particular to our case (two ergodic $A$-maximizing measures). This reasoning may not work for a more general case. Nevertheless, one of the by-product results of our proof is that for the general case, it seems possible to determine the convergence we get here, in   a similar (but more complex) way. The complexity is an issue which is due,  essentially, to the large amount of possible combinatorics.

\medskip
First, we get some equations for the measure $\nu_{\be}$. We remind that this measure is $\be.A$-conformal. This yields,
\begin{eqnarray*}
\nu_{\be}([1])&=& \nu_{\be}\left(\bigsqcup_{n=1}^{\8}[1^{n}2]\sqcup[1^{n}3]\right)\\
&=& \sum_{n=1}^{+\8}\nu_{\be}([1^{n}2])+\nu_{\be}([1^{n}3])\\
&\hskip-2.5cm=& \hskip-1.5cm\sum_{n=1}^{+\8}e^{\be.S_{n}(A)(1^{n}2)-n\cpb}\nu_{\be}([2])+\sum_{n=1}^{+\8}e^{\be.S_{n}(A)(1^{n}3)-n\cpb}\nu_{\be}([3])\\
&=& \frac{e^{-\be\eps_{12}-\cpb}}{1-e^{-\cpb}}\nu_{\be}([2])+\frac{e^{-\be\eps_{13}-\cpb}}{1-e^{-\cpb}}\nu_{\be}([3]).
\end{eqnarray*}
We remind that $\nu_{\be}([1])+\nu_{\be}([2])+\nu_{\be}([3])=1$, then, we get a linear equation between $\nu_{\be}([1])$ and $\nu_{\be}([2])$. Doing the same work with the cylinder $\disp [2]=\bigsqcup_{n=1}^{+\8}[2^{n}1]\sqcup[2^{n}3]$ we get the following system:
\begin{equation}
\label{equa1-nubelin}
\begin{cases}
(e^{\cpb}-1+e^{-\be.\eps_{13}})\nu_{\be}([1])+(e^{-\be.\eps_{13}}-e^{-\be.\eps_{12}})\nu_{\be}([2])= e^{-\be.\eps_{13}},\\
 (e^{-\be.\eps_{23}}-e^{-\be.\eps_{21}})\nu_{\be}([1])+(e^{\cpb}-1+e^{-\be.\eps_{23}})\nu_{\be}([2])= e^{-\be.\eps_{23}}.
 \end{cases}
\end{equation}

The determinant of the system is
\begin{eqnarray*}
 \Delta(\be):=(e^{\cpb}-1)^{2}+(e^{\cpb}-1)(e^{-\be.\eps_{13}}+e^{-\be.\eps_{23}})+\\
\hskip2cm e^{-\be(\eps_{12}+\eps_{23})}+e^{-\be.(\eps_{21}+\eps_{13})}-e^{\be.(\eps_{12}+\eps_{21})},\hfill
\end{eqnarray*}

and, we get
\begin{equation}
\label{equ1-rapportnube}
\frac{\nu_{\be}([1])}{\nu_{\be}([2])}= \frac{(e^{\cpb}-1)e^{-\be.\eps_{13}}+e^{-\be.(\eps_{12}+\eps_{23})}}{(e^{\cpb}-1)e^{-\be.\eps_{23}}+e^{-\be.(\eps_{21}+\eps_{13})}}.
\end{equation}

On the other hand,  Equations \eqref{equ1-subaclineaire} and  \eqref{equ1-subaclineaire} yield the following formula:

\begin{equation}
\label{equ1-rapporthbe}
\frac{H_{\be}(1)}{H_{\be}(2)}=\frac{(e^{\cpb}-1)e^{-\be.\eps_{31}}+e^{-\be.(\eps_{21}+\eps_{32})}}{(e^{\cpb}-1)e^{-\be.\eps_{32}}+e^{-\be.(\eps_{12}+\eps_{31})}}.
\end{equation}
Therefore, from Equations \eqref{equ1-rapportnube} and \eqref{equ1-rapporthbe} we get
\begin{eqnarray}
\frac{\mu_{\be}([1])}{\mu_{\be}([2])}&=&\frac{H_{\be}(1)}{H_{\be}(2)}\frac{\nu_{\be}([1])}{\nu_{\be}([2])}\nonumber\\
&\hskip-3.5cm=& \hskip-2cm\frac{\disp(e^{\cpb}-1)e^{-\be.\eps_{13}}+e^{-\be.(\eps_{12}+\eps_{23})}}{\disp(e^{\cpb}-1)e^{-\be.\eps_{23}}+e^{-\be.(\eps_{21}+\eps_{13})}}\frac{\disp(e^{\cpb}-1)e^{-\be.\eps_{31}}+e^{-\be.(\eps_{21}+\eps_{32})}}{\disp(e^{\cpb}-1)e^{-\be.\eps_{32}}+e^{-\be.(\eps_{12}+\eps_{31})}}\nonumber\\
&=&
\frac{\left(\disp(e^{\cpb}-1)e^{\be.(\eps_{12}+\eps_{21}-\eps_{13})}+e^{\be.(\eps_{21}-\eps_{23})}\right)}
{\disp\left((e^{\cpb}-1)e^{\be.(\eps_{12}+\eps_{21}-\eps_{23})}+e^{\be.(\eps_{12}-\eps_{13})}\right)}\times\nonumber\\
&&
\hskip 2cm \frac{\disp\left((e^{\cpb}-1)e^{\be.(\eps_{12}+\eps_{21}-\eps_{31})}+e^{\be.(\eps_{12}-\eps_{32})}\right)}
{\disp\left((e^{\cpb}-1)e^{\be.(\eps_{12}+\eps_{21}-\eps_{32})}+e^{\be.(\eps_{21}-\eps_{31})}\right)}.\nonumber\\
&&
\label{equ1-rapportmube}
\end{eqnarray}

Convergence will follow from the next proposition.

\begin{proposition}
\label{prop-funcg}
The function $g$ admits a limit as $\be$ goes to $+\8$.
\end{proposition}
\begin{proof}
We remind that $\nu_{\be}$ is the eigenmeasure for the dual transfer operator. This yields:
\begin{eqnarray}
e^{\cpb}&=&\int \CL_{\be}(\BBone)\,d\nu_{\be}\nonumber\\
&=&\hskip-0.3cm(1+e^{-\be.\eps_{21}}+e^{-\be.\eps_{31}})\nu_{\be}([1])+(1+e^{-\be.\eps_{12}}+e^{-\be.\eps_{32}})\nu_{\be}([2])\nonumber\\
&&\hskip 1cm+(e^{-\be.\eps_{13}}+e^{-\be.\eps_{23}}+e^{-\be.\eps_{33}})\nu_{\be}([3]).\nonumber\\
&&\label{equ1-expcpb}
\end{eqnarray}
Let us set
$$
\begin{cases}
 X:=e^{\cpb},\\ A:=e^{-\be.\eps_{13}},\\ A':=e^{-\be.\eps_{23}},\\ B:=e^{-\be.(\eps_{12}+\eps_{23})},\\ B':=e^{-\be.(\eps_{21}+\eps_{13})},\\ C:=e^{-\be.(\eps_{12}+\eps_{21})},\\
 a:=e^{-\be.\eps_{21}}+e^{-\be.\eps_{31}}, \\
 b:= e^{-\be.\eps_{12}}+e^{-\be.\eps_{32}},\\
 c:=e^{-\be.\eps_{13}}+e^{-\be.\eps_{23}}+e^{-\be.\eps_{33}}.
\end{cases}
$$
 From  the system \eqref{equa1-nubelin}, we get exact values for $\nu_{\be}([1])$, $\nu_{\be}([2])$ and $\nu_{\be}([3])$. Replacing these values in \eqref{equ1-expcpb}, this yields
\begin{equation}
\label{equ1-glim}
X\!=\!\frac{(1\!+\!a)(A(X\!-\!1)\!+\!B)\!+\!(1\!+\!b)(A'(X\!-\!1)\!+\!B')+c((X\!-\!1)^{2}-\!C)}{(X-1)^{2}+(A+A')(X-1)+B+B'-C}.
\end{equation}
This can also be written
$$X\!-\!1\!=\!\frac{a(A(X\!-\!1)\!+\!B\!)\!+\!b(A'(X\!-\!1)\!+B')\!+\!(c-1)((X\!-\!1)^{2}-C)}{(X-1)^{2}+(A+A')(X-1)+B+B'-C},$$
and this yields
\begin{eqnarray*}
\hskip-3cm(X-1)^{3}+(A+A'+1-c)(X-1)^{2}\\
+(B+B'-C-a.A-b.A')(X-1)
+C(c-1)-a.B-b.B'=0.
\end{eqnarray*}
\begin{equation}
\label{equ2-glim}
\end{equation}
Remember that all the terms $A,A',B,\ldots$ go exponentially fast to 0 as $\be\to+\8$, $X-1$ behaves like $g(\be)e^{-\rho.\be}$, and, moreover $g$ is sub-exponential.
We can thus use Taylor development to replace $X-1$ by $g(\be)e^{-\rho.\be}$, and, keep in each summand of \eqref{equ2-glim} the largest term; larger here means  that we are comparing it to other terms, if there is a sum, but also to terms of the other summands.

It is for instance clear that $(A+A'+1-c)(X-1)^{2}$ has as dominating term $g^{2}(\be)e^{-2\rho.\be}$, whereas $(X-1)^{3}$  has as dominating term $g^{3}(\be)e^{-3\rho.\be}$, which is exponentially smaller than $g^{2}(\be)e^{-2\rho.\be}$.

In the same direction, note that the term $B$ is killed by part of $b.A'$ and similarly $B'$ is killed by part of $a.A$. The term in $X-1$ is actually equal to
\begin{eqnarray*}
&e^{-\be.(\eps_{12}+\eps_{23})}+e^{-\be.(\eps_{21}+\eps_{13})}-e^{-\be(\eps_{12}+\eps_{21})}&\\
&-e^{-\be.(\eps_{21}+\eps_{13})}-e^{-\be(\eps_{13}+\eps_{31})}-e^{-\be.(\eps_{12}+\eps_{23})}-e^{-\be(\eps_{23}+\eps_{32})}&\\
&=-e^{-\be(\eps_{12}+\eps_{21})}-e^{-\be(\eps_{13}+\eps_{31})}-e^{-\be(\eps_{23}+\eps_{32})}.&
\end{eqnarray*}
Note that this term will be multiplied by $g(\be)e^{-\rho.\be}$ and compared to $g^{2}(\be)e^{-2\rho.\be}$. The fact that $\disp\rho\le \frac{\eps_{12}+\eps_{21}}2$ shows that the term $-e^{-\be(\eps_{12}+\eps_{21})}$ can also be forgotten because it will furnish a quantity  exponentially smaller than $g^{2}(\be)e^{-2\rho.\be}$.

The term without $(X-1)$ is (with change of sign)
\begin{eqnarray*}
&e^{-\be.(\eps_{12}+\eps_{21})}\oplus e^{-\be(\eps_{12}+\eps_{23}+\eps_{21})}\oplus e^{-\be(\eps_{12}+\eps_{23}+\eps_{31})}&\\
&\oplus e^{-\be(\eps_{21}+\eps_{12}+\eps_{13})}\oplus e^{-\be(\eps_{21}+\eps_{13}+\eps_{32})}&\\
=&e^{-\be.(\eps_{12}+\eps_{21})}\oplus e^{-\be(\eps_{12}+\eps_{23}+\eps_{31})}\oplus e^{-\be(\eps_{21}+\eps_{12}+\eps_{13})}&
\end{eqnarray*}
from the max-plus formalism.
Comparing all these terms with the term in $g^{2}(\be)e^{-2\rho.\be}$, we actually recover \ref{equ-rho}: comparing the terms in $2\rho$ and the terms in $\rho$, leads to compare $-\rho$ with $-(\eps_{13}+\eps_{31})\oplus(-\eps_{23}-\eps_{32})$, and comparing the term with $2\rho$, with the terms without $\rho$, leads to compare $-2\rho$ with
$-(\eps_{12}+\eps_{21})\oplus (-\eps_{12}-\eps_{23}-\eps_{31})\oplus(-\eps_{21}-\eps_{31}-\eps_{32})$.

Now, considering the dominating term at exponential scale in  \eqref{equ2-glim} yields an equality of the form
\begin{equation}
\label{equ3-glim}
\wt a g^{2}(\be)-\wt b g(\be)-\wt c =\text{ term exponentially small},
\end{equation}
where $\wt a$ is either 0 or 1, $\wt b\in \{0,1,2,3\}$ and $\wt c\in \{0,1,2,3\}$ and not all the coefficients $\wt a$, $\wt b$ and $\wt c$ are zero\footnote{because we exactly consider the dominating exponential scale.}.
As $g$ is positive and all the coefficient are not zero, we get that necessarily $\wt a=1$. Now, considering any accumulation point $G$ for $g(\be)$, as $\be$ goes to $+\8$, we get
\begin{equation}
\label{equ4-glim}
G^{2}-\wt b G-\wt c=0.
\end{equation}
Such  equation admits all its roots in $\R$, and, at least one of them is non-negative.
But, the key point here is that the roots form a finite set, and this set contains the set of accumulation points for $g(\be)$ as $\be\to+\8$. On the other hand, $g$ is a continuous function, thus the set of accumulation points for $g$ is an interval. This shows that it is reduced to a single point, and then $g(\be)$ converges as $\be\to+\8$.
\end{proof}

We remind that $\cpb$ goes to 0 as $\be$ goes to $+\8$, and then $e^{\cpb}-1$ behaves like $g(\be)e^{-\be.\rho}$. We replace this in Equation \eqref{equ1-rapportmube}.
The final expression is thus:
\begin{eqnarray}
\frac{\mu_{\be}([1])}{\mu_{\be}([2])}&=&\frac{\left(\disp g(\be)e^{\be.(\eps_{12}+\eps_{21}-\eps_{13}-\rho)}+e^{\be.(\eps_{21}-\eps_{23})}\right)}
{\disp\left(g(\be)e^{\be.(\eps_{12}+\eps_{21}-\eps_{23}-\rho)}+e^{\be.(\eps_{12}-\eps_{13})}\right)}\nonumber\\
&&
\hskip2cm \times\frac{\disp
\left(g(\be)e^{\be.(\eps_{12}+\eps_{21}-\eps_{31}-\rho)}+e^{\be.(\eps_{12}-\eps_{32})}\right)}
{\disp\left(g(\be)e^{\be.(\eps_{12}+\eps_{21}-\eps_{32}-\rho)}+e^{\be.(\eps_{21}-\eps_{31})}\right)}.\nonumber\\
&&
\label{equ2-rapportmube}
\end{eqnarray}
We know by Proposition \ref{prop-funcg} that $g(\be)$ converges to a nonnegative limit, then Equality \eqref{equ2-rapportmube} has a limit if $\be\to+\8$ (in $[0,+\8]$). If the limit is $0$, this means that $\mu_{\be}$ goes to $\disp\delta_{\2}$, if the limit is $+\8$ this means that $\mu_{\be}$ goes to $\disp\delta_{\1}$, if it is equal to $\al\in]0,+\8[$, then $\mu_{\be}$ goes to $\disp\frac{\al}{\al+1}\delta_{1}+\frac1{\al+1}\delta_{\2}$.

  \backmatter


\begin{thebibliography}{99}
\footnotesize


\bibitem{TZ1}
S. Addas-Zanata and F. Tal,  Support of maximizing measures for typical $C^0$   dynamics on compact manifolds. Discrete Contin. Dyn. Syst. 26 (2010), no. 3

\bibitem{TZ2}
S. Addas-Zanata and F. Tal, Maximizing measures for endomorphisms of the circle. Nonlinearity 21 (2008), no. 10, 2347-2359.

\bibitem{TZ3}
S. Addas-Zanata and F. Tal,  On maximizing measures of homeomorphisms on compact manifolds. Fund. Math. 200 (2008), no. 2, 145-159



\bibitem{ABG} M. Akian, R. Babat and A.  Guterman,
Handbook of Linear Algebra,
Leslie Hogben (Editor), Section 25, 25-1 to 25-17, Chapmann and Hall, (2006)

\bibitem{AGG}
M. Akian, S. Gaubert, and A.  Guterman,
Linear independence over tropical semirings and beyond. (English summary) Tropical and idempotent Mathematics, 1–38,
Contemp. Math., 495, Amer. Math. Soc., Providence, RI, (2009)


\bibitem{A1}
V. Anagnostopoulou, and  O. Jenkinson, Which beta-shifts have a largest invariant measure? J. Lond. Math. Soc. (2) 79 (2009), no. 2, 445-464

\bibitem{A2} V. Anagnostopoulou, K. Díaz-Ordazand,  O. Jenkinson and C. Richard,
Entrance time functions for flat spot maps. Nonlinearity 23 (2010), no. 6, 147-1494


\bibitem{A3}
V. Anagnostopoulou, K. Díaz-Ordazand,  O. Jenkinson and C. Richard,
Anagnostopoulou, V.; Díaz-Ordaz, K.; Jenkinson, O.; Richard, C. The flat spot standard family: variation of the entrance time median. Dyn. Syst. 27 (2012), no. 1, 29-43.

\bibitem{A4} V. Anagnostopoulou, K. Díaz-Ordazand,  O. Jenkinson and C. Richard,
Sturmian maximizing measures for the piecewise-linear cosine family. Bull. Braz. Math. Soc. (N.S.) 43 (2012), no. 2, 285-302.



\bibitem{Baca}
N. Bacaer, Convergence of numerical methods and parameter dependence of min-plus
eigenvalue problems, Frenkel-Kontorova models and homogenization of Hamilton-
Jacobi equations, ESAIM: Math. Model. Numer. Anal. 35 (2001) 1185-1195.



\bibitem{Bal} V.  Baladi, Positive Transfer Operators and decay of correlations, World Scientific

\bibitem{BLT}
 A. T.  Baraviera, A. O. Lopes and P. Thieullen, A large deviation
principle for equilibrium states of H\"older potencials: the zero
temperature case, \emph{Stochastics and Dynamics} \textbf{6}
(2006), 77-96.

\bibitem{BCLMS}
A. T. Baraviera, L. Cioletti, A. O.  Lopes, J. Mohr and R. R. Souza,
On the general one-dimensional XY model: positive and zero temperature, selection and non-selection. \emph{Rev.Math. Phis.} 23 (2011),no. 10, 1063-1113, 82Bxx.

\bibitem{Bamon}
R. Bamón, J. Kiwi, J. Rivera-Letelier and  R. Urzúa, On the topology of solenoidal attractors of the cylinder. Ann. Inst. H. Poincaré Anal. Non Linéaire 23 (2006), no. 2, 209-236

\bibitem{BLL} A. Baraviera, R. Leplaideur and A. O. Lopes, Selection of measures for a potential with
two maxima at the zero temperature limit, \emph{SIAM Journal on Applied Dynamical Systems}, Vol. 11, n 1, 243-260 (2012)


\bibitem{BLL1} A. Baraviera, R. Leplaideur and A. O. Lopes,
The potential point of view for Renormalization, Stoch. and Dynamics, Vol 12. N 4, (2012) 1250005(1-34)



\bibitem{BLM}
 A. T. Baraviera, A. O. Lopes and J. K. Mengue, On the selection of subaction and measure for a subclass of potentials defined by P. Walters. \emph{ Ergodic Theory and Dynamical Systems}, Available on CJO2012.

\bibitem{Bartle} R. Bartle, The elements of integration, John Wiley (1966)

\bibitem{BG} R. Bissacot and E. Garibaldi, Weak KAM methods and ergodic optimal problems for countable Markov shifts, \emph{ Bull. Braz. Math. Soc.} 41, N 3, 321-338, 210.


\bibitem{BS} R. Bissacot and R. Freire Jr., On the existence of maximizing measures for irreducible
countable Markov shifts: a dynamical proof, to appear in Erg. Theo. and Dyn. Systems

\bibitem{BN} J. Bochi and  A. Navas,
A geometric path from zero Lyapunov exponents to rotation cocycles, Arxiv (2011)



\bibitem{Bousch0}
T. Bousch
Le poisson n'a pas d'arêtes,  Ann. Inst. H. Poincaré Probab. Statist. 36 (2000), no. 4, 489-508.

\bibitem{Bousch1}
T. Bousch, Le poisson n'a pas d'ar\^etes, \emph{Annales de
l'Institut Henri Poincar\'e, Probabilit\'es et Statistiques},
\textbf{36} (2000), 489-508.

\bibitem{Bousch2}
T. Bousch, Un lemme de Mañé bilatéral,  C. R. Math. Acad. Sci. Paris 335 (2002), no. 6, 533-536

\bibitem{Bousch3}
T. Bousch and O. Jenkinson,
Cohomology classes of dynamically non-negative $C^k$   functions. Invent. Math. 148 (2002), no. 1, 207-217.

\bibitem{Bousch4}
T. Bousch and J.  Mairesse Asymptotic height optimization for topical IFS, Tetris heaps, and the finiteness conjecture. J. Amer. Math. Soc. 15 (2002), no. 1, 77-111

\bibitem{Bousch-walters}
T. Bousch,
La condition de Walters.
{\em Ann. Sci. ENS}, 34, (2001)

\bibitem{Bousch1} T. Bousch,
Le lemme de Mañé-Conze-Guivarc'h pour les systèmes amphidynamiques rectifiables, Ann. Fac. Sci. Toulouse Math. (6) 20 (2011), no. 1, 1-14



\bibitem{Bow}
R. Bowen, Equilibrium States and the Ergodic Theory of Anosov Diffeomorphisms, Lecture Notes in Math., vol. 470. Springer, Berlin (1975)

\bibitem{BQ}
X. Bressaud e A. Quas, Rate of approximation of minimizing measures,  \emph{ Nonlinearity} \textbf{20} no. 4, (2007), 845-853.



\bibitem{Bran}
F. Branco,  Sub-actions and maximizing measures for one-dimensional transformations with a critical point. Discrete Contin. Dyn. Syst. 17 (2007), no. 2, 271-280.


\bibitem{Bra}
S. Branton, Sub-actions for Young towers. Discrete Contin. Dyn. Syst. 22 (2008), no. 3, 541-556.

\bibitem{Bremont}
J.~Br{\'e}mont,
\newblock Gibbs measures at temperature zero.
\newblock {\em Nonlinearity}, 16(2): 419--426, 2003.


\bibitem{Bremont2}
J.~Br{\'e}mont,
Finite flowers and maximizing measures for generic Lipschitz functions on the circle. Nonlinearity 19 (2006), no. 4, 813-828.

\bibitem{Bremont3}
J.~Br{\'e}mont,
Entropy and maximizing measures of generic continuous functions. C. R. Math. Acad. Sci. Paris 346 (2008), no. 3-4, 199-201.

\bibitem{Bremont4}
J.~Br{\'e}mont and S. Seuret,
The singularity spectrum of the fish's boundary. Ergodic Theory Dynam. Systems 28 (2008), no. 1, 49-66.

\bibitem{Bremont5}
J. Brémont,  Dynamics of injective quasi-contractions. Ergodic Theory Dynam. Systems 26 (2006), no. 1, 1944.

\bibitem{BQ}
 X. Bressaud and A. Quas,  Rate of approximation of minimizing measures. Nonlinearity 20 (2007), no. 4, 845-853.


\bibitem{But}
P. Butkovic, Max-linear Systems: Theory and Algorithms, Springer Monographs in Mathematics, Springer-Verlag

\bibitem{CM}
D. Collier and I. D.  Morris, Approximating the maximum ergodic average via periodic orbits. Ergodic Theory Dynam. Systems 28 (2008), no. 4, 1081-1090

\bibitem{CD}
W. Chou and R. J. Duffin,
An additive eigenvalue problem of physics related to linear programming,
\emph{Advances in Applied Mathematics} \textbf{8},  486-498, 1987.

\bibitem{CD2} W. Chou and R. Grifiths, Ground states of one-dimensional systems using effective
potentials, Phys. Rev. B 34 (1986) 6219-6234.


\bibitem{Chung} M.  Chung, Eigenvalues and eigenvectors in the Max-Plus Algebra,
Master's Thesis, University of Colorado, Denver (1995)

\bibitem{Conc} M. C. Concordel, Periodic homogenization of Hamilton-Jacobi equations: Additive
eigenvalues and variational formula, Indiana Univ. Math. J. 45 (1996) 1095-1118.





\bibitem{CGU} J.R. Chazottes, J.M. Gambaudo and E. Ulgade, Zero-temperature limit of one dimensional Gibbs states via renormalization: the case of locally constant potentials, \emph{Erg. Theo. and Dyn. Sys.}  31 (2011), no. 4, 1109-1161

\bibitem{CH}J. R. Chazottes and M. Hochman, On the zero-temperature limit of Gibbs states, \emph{Commun. Math. Phys. } vol 297, N1, 2010.



\bibitem{CI}
G. Contreras and R. Iturriaga. \emph{ Global minimizers of autonomous
Lagrangians}, 22$^\circ$ Co\-l\'o\-quio Brasileiro de Matem\'atica,
IMPA, 1999.

\bibitem{CLT}
G. Contreras, A. O. Lopes and Ph. Thieullen.  Lyapunov minimizing measures for
expanding maps of the circle, Ergodic Theory and Dynamical Systems, Vol 21, 1379-1409, 2001.

\bibitem{CLO} G. Contreras,  A. Lopes and E. Oliveira,
Ergodic Transport Theory, periodic maximizing probabilities  and  the twist condition, to appear in \emph{Modeling, Optimization, Dynamics and Bioeconomy}, Springer Proceedings in Mathematics, Edit. David Zilberman and Alberto Pinto.



\bibitem {CG}     J. P. Conze and Y. Guivarc'h,  Croissance des sommes ergodiques et
principe variationnel, manuscript circa (1993).


\bibitem {DUZ}
A. Davie, M. Urbanski annd A. Zdunik, Maximizing measures of metrizable non-compact spaces. Proc. Edinb. Math. Soc. (2) 50 (2007), no. 1, 123-151

\bibitem{DZ} A. Dembo and O. Zeitouni,  Large Deviations Techniques and Applications, Springer Verlag, 1998.


\bibitem{Ellis} R. Ellis,  Entropy, Large Deviations, and Statistical Mechanics, Springer Verlag, 2005


\bibitem{EL} R. Exel and A. Lopes,
$C^*$-Algebras, approximately proper equivalence
relations and Thermodynamic Formalism, Erg Theo and Dyn Sys,
1051-1082, Vol 24,  t (2004).


\bibitem{Far} K. G. Farlow, Max-Plus Algebra,
Master's Thesis,
Virginia Polytechnic Institute and State University (2009).


\bibitem{Fathi}
A. Fathi, Th\'eor\`eme KAM faible et th\'eorie de Mather sur les syst\`emes lagrangiens,
\emph{Comptes Rendus de l'Acad\'emie des Sciences, S\'erie I, Math\'ematique} Vol 324   1043-1046, 1997.

\bibitem{Fer} P. Fernandez, Medida e Integra\c c\~ao, (1976) IMPA


\bibitem{LF} A. Fisher and A. Lopes,
Exact bounds for the polynomial decay of
correlation, 1/f noise and the central limit theorem for a
non-H\"older Potential,
Nonlinearity, Vol 14, Number 5, pp 1071-1104 (2001).

\bibitem{GL1} E. Garibaldi and A. O. Lopes,  On Aubry-Mather theory for symbolic Dynamics,
Ergodic Theory and Dynamical Systems, Vol 28 , Issue 3, 791-815 (2008)


\bibitem{GL2} E. Garibaldi and A. Lopes, Functions for relative maximization, Dynamical Systems, v. 22, 511-528, 2007

\bibitem{GL3} E. Garibaldi and A. O. Lopes,
The effective potential and transshipment in thermodynamic formalism at temperature zero, Stoch. and Dyn., Vol 13 - N 1, 1250009 (13 pages) (2013).




\bibitem{GLT}  E. Garibaldi, A. Lopes and P. Thieullen,
On calibrated and separating sub-actions,
Bull. of the Bras. Math. Soc. Vol 40, 577-602, (4)(2009)


\bibitem{GT1} E. Garibaldi and Ph. Thieullen,
Minimizing orbits in the discrete Aubry-Mather model, Nonlinearity, 24 (2011), no. 2, 563-611

\bibitem{GT2} E. Garibaldi and Ph. Thieullen,
 Description of some ground states by Puiseux technics, \emph{Journ. of Statis. Phys}, 146, no. 1, 125-180,  (2012)


\bibitem{Ge} H.-O. Georgii,  Gibbs Measures and Phase
Transitions. de Gruyter, Berlin, (1988).


\bibitem{Gomes1}
D. A. Gomes, Viscosity solution method and the discrete Aubry-Mather problem,
\emph{Discrete and Continuous Dynamical Systems, Series A} \textbf{13} (2005), 103-116.

\bibitem{GLM}
D. A. Gomes, A. O. Lopes and J. Mohr,
The Mather measure and a large deviation principle for the entropy penalized method.
Commun. Contemp. Math. 13 (2011), no.2, 235-268.

\bibitem{Grant} A. Grant,
Finding Optimal Orbits of Chaotic Systems, Phd thesis, Univ. of Maryland (2005)

\bibitem{Hof} F. Hofbauer,  Examples for the nonuniqueness of the equilibrium state. Trans. Amer. Math. Soc. 228, no. 223–241 (1977)

\bibitem{H1}
B. R. Hunt,
Maximum local Lyapunov dimension bounds the box dimension of chaotic attractors,
Nonlinearity 9 (1996), no. 4, 845-852.

\bibitem{HY}
B. R. Hunt and Y. Guocheng,  Optimal orbits of hyperbolic systems.
\emph{Nonlinearity}  \textbf{12}, (1999), 1207-1224.

\bibitem{HO}
B. Hunt B and E. Ott,  Optimal periodic orbits of chaotic systems occur at low period, Phys. Rev. E 54 32837 (1996)

\bibitem{HM}
E. Harriss and O. Jenkinson, Flattening functions on flowers. Ergodic Theory Dynam. Systems 27 (2007), no. 6, 1865-1886

\bibitem{Iom}
Godofredo Iommi,
Ergodic Optimization for Renewal Type Shifts, Monatshefte für Mathematik,
Volume 150, Number 2 (2007), 91-95,



\bibitem{Jenk1} O. Jenkinson.  Ergodic optimization, Discrete and Continuous
Dynamical Systems, Series A, V. 15, 197-224, 2006.

\bibitem{Jenk2}
O. Jenkinson and J. Steel,  Majorization of invariant measures for orientation-reversing maps. Ergodic Theory Dynam. Systems 30 (2010), no. 5, 1471-1483.


\bibitem{Jenk3}
O. Jenkinson and I. D.  Morris, Lyapunov optimizing measures for C 1   expanding maps of the circle. Ergodic Theory Dynam. Systems 28 (2008), no. 6, 1849-1860

\bibitem{Jenk4}
O. Jenkinson,  A partial order on $\times2$ -invariant measures. Math. Res. Lett. 15 (2008), no. 5, 893-900.


\bibitem{Jenk5}
O. Jenkinson,  Optimization and majorization of invariant measures. Electron. Res. Announc. Amer. Math. Soc. 13 (2007), 112

\bibitem{Jenk6}
O. Jenkinson,  Every ergodic measure is uniquely maximizing. Discrete Contin. Dyn. Syst. 16 (2006), no. 2, 383-392.

\bibitem{Jenk7}
O. Jenkinson,  Maximum hitting frequency and fastest mean return time. Nonlinearity 18 (2005), no. 5, 2305-2321

\bibitem{Jenk8}
O. Jenkinson, Directional entropy of rotation sets. C. R. Acad. Sci. Paris Sér. I Math. 332 (2001), no. 10, 921-926.

\bibitem{Jenk9}
O. Jenkinson, Geometric barycentres of invariant measures for circle maps. Ergodic Theory Dynam. Systems 21 (2001), no. 2, 511532.

\bibitem{JMU}
O. Jenkinson, R. D. Mauldin and M. Urbanski,
Zero Temperature Limits of Gibbs-Equilibrium States for Countable Alphabet Subshifts of Finite Type, Journ. of Statis. Physics,
Volume 119, Numbers 3-4 (2005)


\bibitem{JR}
T. Jordan and M. Rams, Multifractal analysis of weak Gibbs measures for non-uniformly expanding C 1   maps. Ergodic Theory Dynam. Systems 31 (2011), no. 1, 143164.

\bibitem{KT} S. Karlin and H. Taylor, A First Course in Stochastic Processes, Academic Press


\bibitem{Kel} G. Keller, Gibbs States in Ergodic Theory, Cambrige Press, 1998.




\bibitem{Kem} T. Kempton, Zero Temperature Limits of Gibbs Equilibrium States
for Countable Markov Shifts, \emph{J. Stat Phys} 143, 795-806, 2011.


\bibitem{krengel} U. Krengel, \emph{Ergodic theorems},
1985,
De Gruyter Publ.



\bibitem{Lan}
O. Lanford, Entropy and Equilibrium States in Classical Statistical Mechanics.
Statistical mechanics and mathematical problems. Battelle Rencontres,
Seattle, Wash., 1971. Lecture Notes in Physics, 20. Springer-Verlag,
Berlin-New York, 1973.




\bibitem{Lep1} R. Leplaideur, A dynamical proof for convergence of Gibbs measures at temperature zero, Nonlinearity, 18, N 6, 2847-2880, 2005.

\bibitem{Lep2} R. Leplaideur, Totally dissipative measures for the shift and conformal s -finite measures for the stable holonomies. Bull. Braz. Math. Soc.  41 (2010), no. 1, 136.

\bibitem{Lep3} R. Leplaideur,
Flatness is a criterion for selection of maximizing measures. J. Stat. Phys. 147 (2012), no. 4, 728757.

\bibitem{Lep4} R. Leplaideur,
Thermodynamic formalism for a family of non-uniformly hyperbolic horseshoes and the unstable Jacobian. Ergodic Theory Dynam. Systems 31  no. 2, 423447, (2011),

\bibitem{lind-marcus}
D.~{L}ind and B.~{M}arcus.
\newblock {\em {An introduction to Symbolic Dynamics and Coding}}.
\newblock Cambridge University Press, 1995.





\bibitem{LMST}
A. O. Lopes, J. Mohr, R. Souza and Ph. Thieullen,
Negative entropy, zero temperature and stationary Markov chains on the interval,
\emph{Bulletin of the Brazilian Mathematical Society} \textbf{40}, 1-52, 2009.




\bibitem{LoT}  A. O.  Lopes. Thermodynamic Formalism, Maximizing Probabilities and Large Deviations.  Work in progress. Lecture Notes - Dynamique en Cornouaille (2012)



\bibitem{LM1} A. O. Lopes and J. Mengue, Zeta measures  and Thermodynamic Formalism for temperature zero, \emph{Bulletin of the Brazilian Mathematical Society} 41 (3) pp 449-480 (2010)


\bibitem{LM2}
A. Lopes and J. Mengue,
Duality Theorems in Ergodic Transport, Journal of Statistical Physics. Vol 149, issue 5, pp 921942 (2012)




\bibitem{LOT}
A. O. Lopes, E. R. Oliveira and Ph. Thieullen,
The dual potential, the involution kernel and transport in ergodic optimization,
\emph{preprint} Arxiv.


\bibitem{LT1} A. Lopes and  P. Thieullen,
Eigenfunctions of the Laplacian and associated Ruelle operator,
Nonlinearity, Volume 21, Number 10, pp-2239-2254 (2008)

\bibitem{LRR} A. Lopes, V. Rosas and R. Ruggiero
Cohomology and Subcohomology for expansive geodesic flows,
Discrete and Continous Dynamical Systems" Volume: 17, N. 2, pp 403-422 (2007).

\bibitem{LT7} A. Lopes and  P. Thieullen, Subactions for Anosov Diffeomorphisms, Asterisque, volume 287, {\emph
Geometric Methods in Dynamics (II)} - pp 135-146 (2003).


\bibitem{LT2} A. Lopes and  P. Thieullen,
Subactions for Anosov
Flows, Erg Theo and Dyn Syst vol 25 Issue 2, pp605-628 (2005)



\bibitem{LT3} A. Lopes and  P. Thieullen, Mather measures and
the Bowen-Series transformation, Annal Inst Henry Poincare - Anal non-linear, 663-682, vol 5, (2006)



\bibitem{LO} A. O. Lopes and E. Oliveira, Entropy and variational principles for holonomic probabilities of IFS, \emph{Disc. and Cont. Dyn. Systems } series A, vol 23, N 3, 937-955, 2009.



\bibitem{LO1} A. O. Lopes and E. R.  Oliveira,
On the thin boundary of the fat attractor, preprint (2012)




%\bibitem{Lop2} A.O. Lopes, An analogy of charge distribution on Julia sets with the Brownian motion  , \emph{J. Math. Phys. } 30 9, 2120-2124, 1989.


\bibitem{LZ} A. O. Lopes,
The Zeta Function, Non-Differentiability
of Pressure and The Critical Exponent of Transition",
Advances in Mathematics\, Vol. 101, pp. 133-167,
(1993).



\bibitem{LM3} A. Lopes and J. Mengue,
Selection of measure and a Large Deviation Principle for the general one-dimensional XY model,
preprint, UFRGS, 2011.

\bibitem{LMMS} A. Lopes, J. K. Mengue, J. Mohr and R. R. Souza Entropy and Variational Principle for one-dimensional Lattice Systems with a general a-priori measure: positive and zero temperature, preprint (2012)


\bibitem{LOS} A. Lopes, E. R. Oliveira and D. Smania,
Ergodic Transport Theory and Piecewise Analytic Subactions for Analytic Dynamics, Bull. of the Braz. Math Soc. Vol 43 (3) 467-512 (2012)


\bibitem{LL} A. Lopes and S. Lopes, Uma introdu\c c\~ao aos Processos Estoc\'asticos para estudantes de Matem\'atica, a ser publicado

\bibitem{Lima} E. L. Lima, Curso de Análise , Vol II, IMPA (1981)


\bibitem{Mane}
R. Ma\~n\'e.   Generic properties and problems of minimizing
measures of Lagrangian systems, \emph{Nonlinearity}, Vol 9, 273-310, 1996.


\bibitem{Mat}
J. Mather, Action minimizing invariant measures for positive
definite Lagrangian Systems, \emph{Math. Z.}, 207 (2), pp 169-207, 1991


\bibitem{Mor1}
I. D. Morris,
Entropy for Zero-Temperature Limits of Gibbs-Equilibrium States for Countable-Alphabet Subshifts of Finite Type, Journ. of Statis. Physics,
Volume 126, Number 2 (2007), 315-324,

\bibitem{Mor2}
I. D. Morris,
A rapidly-converging lower bound for the joint spectral radius via multiplicative ergodic theory. Adv. Math. 225 (2010), no. 6, 34253445.

\bibitem{Mor3}
I. D. Morris,
Ergodic optimization for generic continuous functions. Discrete Contin. Dyn. Syst. 27 (2010), no. 1, 383388

\bibitem{Mor4}
I. D. Morris,
The Mañé-Conze-Guivarc'h lemma for intermittent maps of the circle. Ergodic Theory Dynam. Systems 29 (2009), no. 5, 16031611

\bibitem{Mor5}
I. D. Morris,
Maximizing measures of generic Hölder functions have zero entropy. Nonlinearity 21 (2008), no. 5, 9931000

\bibitem{Mor6}
I. D. Morris,  A sufficient condition for the subordination principle in ergodic optimization. Bull. Lond. Math. Soc. 39 (2007), no. 2, 214220

\bibitem{Mor6}
I. D. Morris,  Entropy for zero-temperature limits of Gibbs-equilibrium states for countable-alphabet subshifts of finite type. J. Stat. Phys. 126 (2007), no. 2, 315324.

\bibitem{krerley}
K. Oliveira and M. Viana,
Thermodynamical formalism for robust classes of potentials and non-uniformly hyperbolic maps. Ergodic Theory and Dynamical Systems  28-02  (2008),  pp 501-533.




\bibitem{PP}
W. Parry and M.  Pollicott. {\it Zeta functions and the periodic
orbit structure of hyperbolic dynamics}, Ast\'erisque
Vol {187-188} 1990

\bibitem{Pit}
V. Pit,
Invariant relations for the Bowen-Series transform,
Conform. Geom. Dyn. 16 (2012), 103-123.

\bibitem{PY}
M.  Pollicott and M. Yuri, Dynamical systems and Ergodic Theory, Cambrige Press, 1998

\bibitem{PY}
M. Pollicott and R. Sharp, Livsic theorems, maximizing measures and the stable norm. Dyn. Syst. 19 (2004), no. 1, 7588.

\bibitem{QS} A. Quas and J. Siefken,
Ergodic optimization of supercontinuous
functions on
shift spaces,
Ergodic Theory and Dynamical Systems / Volume 32 / Issue 06 2012,  2071 2082

\bibitem{Sin} Y. Sinai, Probabilty Theory, an introductory course,Springer Veralag


\bibitem{Sa} O. Sarig,  Thermodynamic formalism for countable
Markov shifts, \emph{ Ergodic Theory and Dynamical Systems} 19, 1565-1593, 1999

\bibitem{Sa2} O. Sarig, Lecture notes on thermodynamic formalism for topological Markov shifts, \emph{ Penn State}, 2009.


\bibitem{SSS} E. A. da Silva, R. R. da Silva and  R. R. Souza
The Analyticity of a Generalized Ruelle's Operator, to appear Bull. Braz. Math. Soc.


\bibitem{Sou} R. R. Souza, Sub-actions for weakly hyperbolic one-dimensional systems, \emph{Dynamical Systems} 18 \textbf{(2)}, 165-179, 2003.




 \bibitem{Spit} F. Spitzer.
A Variational characterization of finite Markov chains. \emph{The Annals of Mathematical Statistics}.
(43): N.1  303-307, 1972.


\bibitem{Steel}
J. Steel, Concave unimodal maps have no majorisation relations between their ergodic measures. Proc. Amer. Math. Soc. 139 (2011), no. 7, 25532558



\bibitem{Vil1}
C. Villani, Topics in optimal transportation, AMS, Providence, 2003.

\bibitem{Wal} P. Walters, Ergodic Theory, Springer Verlag

\bibitem{Yuan-Hunt}
G. Yuan and B. Hunt,
Optimal orbits of hyperbolic systems,
Nonlinearity 12 (1999), no. 4, 1207–1224.




\end{thebibliography}
\end{document}